\documentclass[11pt]{article}
\usepackage[margin =1in]{geometry}
\usepackage{amssymb,amsmath,amsthm}
\usepackage{enumerate}
\usepackage{hyperref}
\usepackage{cite}
\usepackage[capitalize]{cleveref}
\usepackage{enumitem}
\usepackage{color}
\usepackage{xcolor}
\definecolor{blue}{rgb}{0,0,0}

\usepackage{cancel}
\usepackage{pgf}
\usepackage{svg}
\usepackage{pgfplots}
\usepackage{import}

\usepackage[utf8]{inputenc} 
\usepackage[T1]{fontenc}    
\usepackage{hyperref}       
\usepackage{url}            
\usepackage{booktabs}       
\usepackage{amsfonts}       
\usepackage{nicefrac}       
\usepackage{microtype}
\usepackage{multirow}
\usepackage{xfrac}
\usepackage{todonotes}
\usepackage{titlesec}
\usepackage{caption}

\usepackage{multirow}

\usepackage{caption}
\usepackage{float}

\titleformat{\subsubsection}[runin]
{\normalfont\normalsize\bfseries}{\thesubsubsection}{1em}{}

\usepackage{graphicx}

\usepackage{appendix}
\numberwithin{equation}{section}

\newcommand{\inclu}[0] {\ar@{^{(}->}}

\newcommand{\argmin}{\operatornamewithlimits{argmin}}

\newtheorem{thm}{Theorem}[section]
\newtheorem{theorem}[thm]{Theorem}

\newtheorem{proposition}[thm]{Proposition}

\newtheorem{lemma}[thm]{Lemma}

\newtheorem{corollary}[thm]{Corollary}

\newtheorem{assumption}[thm]{Assumption}

\newtheorem{remark}[thm]{Remark}

\crefname{claim}{claim}{claims}
\Crefname{claim}{Claim}{Claims}
\crefname{lem}{lemma}{lemmas}
\Crefname{lem}{Lemma}{Lemmas}
\crefname{algorithm}{algorithm}{algorithms}
\Crefname{algorithm}{Algorithm}{Algorithms}

\theoremstyle{definition}

\theoremstyle{definition}

\theoremstyle{definition}
\newtheorem*{claim*}{Claim}

\usepackage{mathtools}

\usepackage[algo2e, ruled]{algorithm2e}

\usepackage[T1]{fontenc}
\usepackage{lmodern}
\crefname{figure}{Figure}{Figures}

\pgfplotsset{compat=1.18}

\begin{document}

    \title{Some Unified Theory for Variance Reduced Prox-Linear Methods}

    \author{Yue Wu\footnote{Johns Hopkins University, Department of Applied Mathematics and Statistics, \url{ywu166@jhu.edu}} \qquad Benjamin Grimmer\footnote{Johns Hopkins University, Department of Applied Mathematics and Statistics, \url{grimmer@jhu.edu}}}

	\date{}
	\maketitle

	\begin{abstract}
        This work considers the nonconvex, nonsmooth problem of minimizing a composite objective of the form $f(g(x))+h(x)$ where the inner mapping $g$ is a smooth finite summation or expectation amenable to variance reduction. In such settings, prox-linear methods can enjoy variance-reduced speed-ups despite the existence of nonsmoothness. We provide a unified convergence theory applicable to a wide range of common variance-reduced vector and Jacobian constructions. {\color{blue} All the technical conditions we required for variance-reduced methods can be summarized in a single unified assumption.} Our theory (i) only requires operator norm bounds on Jacobians (whereas prior works used potentially much larger Frobenius norms), (ii) provides state-of-the-art high probability guarantees, and (iii) allows inexactness in proximal computations.
	\end{abstract}

    \section{Introduction} \label{Sect:Intro}
In this work, we consider nonsmooth, nonconvex problems
\begin{equation}\label{eq:main}
\min_{x\in \mathbb{R}^n} \Phi(x) := f(g(x))+h(x)
\end{equation}
where $f\colon \mathbb{R}^m \to \mathbb{R}$ and $h\colon \mathbb{R}^n \to \mathbb{R}$ are convex functions and $g\colon \mathbb{R}^n \to \mathbb{R}^m$ is a differentiable mapping. Note that although $f$ is convex and $g$ is smooth, their composition may be neither convex nor smooth. 
This ``convex-composite'' optimization model is surprisingly versatile. As two classic example applications, 
\begin{itemize}
    \item {\bf Nonlinear Programming.} Consider minimizing an objective function $h(x)$ subject to functional constraints $g^{(\ell)}(x)\leq 0$ for $\ell=1\dots m$. Then letting $g(x) = [g^{(1)}(x),\dots,g^{(m)}(x)]$ and $f(z)$ be either an indicator function for the nonpositive orthant or an exact penalty $f(z)=\sum_{i=1}^m C \max\{z_i,0\}$ for sufficiently large $C$, any such nonlinear program can be cast in the form~\eqref{eq:main}. Of particular interest here are settings where each constraint $g^{(\ell)}(x)$ takes the form of a summation $\frac{1}{N}\sum_{j=1}^N g^{(\ell)}_j(x)$ as occurs across machine learning tasks.
    \item {\bf Nonlinear Equation Solving/Regression.} Consider solving a system of equations $0=g(x):=\mathbb{E}_{\xi \sim D}g_\xi(x)$, given only oracles for sampling from $D$ and first-order queries about individual samples $g_\xi(x)$. If one measures solution quality in some norm $f(z) = \|z\|$, minimizing solution error takes the form~\eqref{eq:main}. Any additional regularization can be modeled by $h(x)$, for example, setting $h(x)=\|x\|_1$.
\end{itemize}

We focus on reducing the number of first-order queries needed to elements of the finite summations $g_i$ or expectations $g_\xi$ as occur above. 
Our approach is based on leveraging two well-studied tools in first-order optimization, discussed briefly below: {\it variance reduction} and {\it prox-linear methods}. This combination was recently considered by Zhang and Xiao~\cite{Zhang2021} and Tran-Dinh et al.~\cite{tran-dinh2020a}, motivating our work. 

\noindent {\bf Variance Reduction.}
Throughout, we assume $g\colon \mathbb{R}^n \to \mathbb{R}^m$ is either a finite sum
\begin{equation} \label{eq:finite-sum-g}
    g(x) = \frac{1}{N}\sum_{j=1}^N g_j(x)
\end{equation}
or, more generally, an expectation
\begin{equation} \label{eq:expectation-g}
    g(x) = \mathbb{E}_{\xi\sim D} [g_\xi(x)],
\end{equation}
and that oracles for evaluating components $g_\xi(\cdot)$ and their Jacobians $g_\xi'(\cdot)$ are given. Given samples $\xi \sim D$, these oracle evaluations provide unbiased estimates of $g(\cdot)$ and $g'(\cdot)$. {\color{blue} Directly using these oracles or their averages (minibatching) has been extensively studied~\cite{Ghadimi2013,Ghadimi2016}.} Variance reduction techniques enable the construction of lower variance estimators where a high accuracy (large batch) estimate only needs to be computed every $\tau$ iterations. For the most classic style of update, due to~\cite{Johnson2013,Zhang2021}, every $\tau$ steps would use estimators of the form
\begin{equation}
    \begin{cases}
        \widetilde{g}_0 = \frac{1}{|\mathcal{A}_0|} \sum_{\xi\in \mathcal{A}_0} g_{\xi}(x_0)\\
        \widetilde{g}_i = \frac{1}{|\mathcal{A}_i|} \sum_{\xi\in \mathcal{A}_i} \Bigl( g_{\xi}(x_i) - g_{\xi}(x_0) \Bigr) + \widetilde{g}_0 \qquad \forall i=1,\dots,\tau-1
    \end{cases}\label{eq:basic-VR}
\end{equation}
where the batches $\mathcal{A}_i$ can be much smaller than $\mathcal{A}_0$. When $g$ is given by a finite summation~\eqref{eq:finite-sum-g}, $\widetilde{g}_0$ could be computed exactly. At the cost of additional Jacobian evaluations $g_\xi'(\cdot)$, further refined schemes have been considered~\cite{Zhang2021,Zhou18d}
\begin{equation}
    \widetilde{g}_i = \frac{1}{|\mathcal{A}_i|} \sum_{\xi\in \mathcal{A}_i} \Bigl( g_{\xi}(x_i) - g_{\xi}(x_0) - g'_{\xi}(x_0) (x_i-x_0) \Bigr) + \widetilde{g}_0 + \widetilde{J}_0 (x_i-x_0) 
    \label{eq:smart-VR}
\end{equation}
where $ \widetilde{J}_0 $ is an unbiased estimate of $g'(x_0)$.
Methods specifically targeting root-finding were recently given by~\cite{Davis2022VR} and generalizing to allow relative smoothness by~\cite{Liu2022StochasticCO}. See the survey~\cite{GowerSurvey2020} for more historical context. 
{\color{blue} Variance reduction also underlies more advanced algorithms. For example, the additive setting in~\cite{Reddi2016} can be addressed by variants of SAGA/SVRG estimators of~\cite{Defazio2014,Johnson2013}. Our work shares a similar spirit, but aims at a unified theory for variance reduced methods, focused on composite settings~\eqref{eq:main}.}

\noindent {\bf Prox-linear Methods.} Note a fundamental difficulty in~\eqref{eq:main} is that the composition of a convex function $f$ with a smooth function $g$ may be nonconvex. In contrast, the composition of a convex function with a linear function always remains convex. This motivates replacing $g(\cdot)$ by its linearization $g(x_k) + g'(x_k)(\cdot - x_k)$. Repeatedly minimizing this relaxed convex problem, with an added proximal term, is known as the ``prox-linear method''~\cite{Burke1995,Cartis2011,Drusvyatskiy2016ErrorBQ,Drusvyatskiy2019,Lewis2016,Nesterov2007}
\begin{equation}\label{eq:prox-linear}
    x_{+} = \argmin_{ y\in\mathbb{R}^n} \left\{f\left(g(x) + g'(x)(y -x)\right) + h(y) + \frac{M}{2}\|y-x\|^2_2\right\}
\end{equation}
given some proximal parameter $M>0$. If the above argmin is only computed approximately, perhaps via some first-order method for convex optimization using (sub)gradients of $f$, we refer to this as an ``inexact prox-linear method''. 

This prox-linear step provides a generalized notion of stationarity for composite nonsmooth, nonconvex problems. Denote the generalized gradient at some $x$ by
\begin{equation} \label{eq:GeneralizedGradient}
    \mathcal{G}_M(x) := M(x - x_+) \in \partial \left(f(g(x) + g'(x)(\cdot -x)) + h\right)(x_+)
\end{equation}
where $x_+$ is defined as the exact prox-linear step~\eqref{eq:prox-linear}. The optimality condition defining $x_+$ in~\eqref{eq:prox-linear} ensures
$ \mathcal{G}_M(x) \in \partial \left(f(g(x) + g'(x)(\cdot -x)) + h\right)(x_+) $. 
By the sum and chain rules of subdifferential calculus, there exists $\lambda \in \partial f(g(x) + g'(x)(x_+ -x)) $ and $\zeta\in\partial h(x_+)$ such that $\mathcal{G}_M(x) = g'(x)^\top\lambda + \zeta$. Hence if $\|\mathcal{G}_M(x)\|\leq \epsilon$, then together $\lambda$, $g'(x)$, and $\zeta$ provide a small subgradient certifying stationarity where each of these differential objects is taken at points near $x$. See the survey~\cite{drusvyatskiy2017proximalpointmethodrevisited} for more historical context on prox-linear methods and similar approximate notions of stationarity.

\noindent {\bf Our Contributions.} We analyze variance-reduced, prox-linear methods, iterating
\begin{equation}\label{eq:sketch-of-method}\begin{cases}
    \widetilde g_k, \widetilde J_k & \leftarrow \mathrm{Variance\ Reduced\ Estimates\ of\ } g(x_k),g'(x_k) \\
    x_{k+1}  &\leftarrow \mathrm{Approximate\ Minimizer\ of\ } f\left(\widetilde g_k + \widetilde J_k(x-x_k)\right) + h(x) + \frac{M}{2}\|x-x_k\|^2_2 . 
\end{cases}
\end{equation}
We provide a unified theory for the oracle complexity with respect to evaluations of the vector $g_\xi(\cdot)$ and its Jacobian $g'_\xi(\cdot)$, for a range of variance-reduced approaches to constructing $\widetilde g_k$ and $\widetilde J_k$. Our main theorem (Theorem~\ref{thm:unified-thm-short}) offers three main advances:
\begin{itemize}
    \item {\bf Operator Norm Assumptions.} Our theory only relies on uniform bounds on the variation of Jacobians in operator norm of the form
    \begin{equation}
    \label{eq:op-norm-constants}
    \| g'_\xi(x) - g'_\xi(y)\|_\mathtt{op} \leq L_{g}\|x-y\|_2, \qquad \text{and} \qquad \|g'_\xi(x) - g'(x)\|_\mathtt{op} \leq \sigma_{g'}.
    \end{equation}
    {\color{blue} See Section~\ref{Sect:Prelim} and Assumption~\ref{assumption:variance-like} for formal assumptions on this.} 
    Prior works have instead used the Frobenius norm (see related work discussion below). As a result, the ``constants'' in prior works may be up to a dimension-dependent factor of $\sqrt{\min\{n,m\}}$ times larger than those considered here.
    \item {\bf A Pareto Frontier of State-of-the-Art Guarantees.} Our theory provides guarantees that various prox-linear methods produce a $(\epsilon,\Delta)$-h.p.~stationary point, meaning with probability $1-\Delta$, some $x_k$ has $\|\mathcal{G}_M(x_k)\|_2^2 \leq \epsilon$. Depending on the relative cost of evaluating $g_\xi$ and $g_\xi'$ evaluations in~\eqref{eq:expectation-g} or the relative size of $1/\epsilon$ and $N$ in~\eqref{eq:finite-sum-g}, the best-known method varies. See the many state-of-the-art corollaries in Section~\ref{subsec:vr-corollaries}.
    \item {\bf Accounting of Inexact Proximal Computations.} Our theory allows for inexact prox-linear steps.
    Section~\ref{subsec:inexact-accounting} provides guarantees including the cost of subroutines.
    For example, guarantees follow for doubly stochastic problems where $f$ is also defined as an expectation, requiring inexact minimization. 
\end{itemize}
{\color{blue} As an immediate application, we get the first guarantees for the direct application of the SVRG update~\eqref{eq:basic-VR} to prox-linear methods under the general expectation setting of~\eqref{eq:expectation-g}.}

\noindent {\bf Outline.} The remainder of this section discusses related work. Section~\ref{Sect:Prelim} provides preliminaries and introduces the general algorithm considered. Section~\ref{Sect:MainResults} states our unified convergence theorem and applies it to produce state-of-the-art guarantees for several variance-reduction schemes. Finally, Section~\ref{Sect:Analysis} provides our technical analysis.

\subsection{Related Work}\label{subsec:relatedWork}







The setting~\eqref{eq:main} was recently addressed by two works~\cite{tran-dinh2020a, Zhang2021}. A key insight was their identification that prox-linear methods can benefit from variance reduction despite the existence of nonsmoothness. Although both of these prior works are motivated by bounds on operator norms of Jacobians, their proof techniques relied on uniformly bounding Jacobian matrices in the Frobenius norm\footnote{Both prior works~\cite{tran-dinh2020a,Zhang2021} rely on matrix generalizations of mean-squared error bounding lemmas typical to the analysis of methods with stochastic gradient vectors (see~\cite[Lemma 1]{spider} and~\cite[Lemma 2]{sarah} for the essential vector arguments being generalized). At their core, such lemmas rely on a classic bias-variance decomposition: given a space $\mathcal{E}$ with inner product $\langle\cdot,\cdot\rangle$, a random variable $X_\xi\in \mathcal{E}$ and some fixed $Y\in\mathcal{E}$, one has
$$ \mathbb{E}_\xi\|X_\xi - Y\|^2_{\langle \cdot,\cdot\rangle} = \|\mathbb{E}_\xi[X_\xi] - Y\|^2_{\langle \cdot,\cdot \rangle} + \mathbb{E}_\xi\|X_\xi - \mathbb{E}_{\xi'}[X_{\xi'}]\|^2_{\langle \cdot,\cdot\rangle} $$
where $\|\cdot\|_{\langle \cdot,\cdot\rangle}$ denotes the norm associated with the given inner product.
In the space of matrices, one could apply this reasoning with the trace inner product to relate Frobenius norms. However, such relationships do not hold for norms without an associated inner product (e.g., matrix operator norms), and so prior works, even if not denoted, require the potentially larger Frobenius norm.}.
Our analysis relies only on operator norm bounds, closing this theoretical gap and offering improvements by dimension-dependent factors. 
To make formal comparisons, denote our ``constants'' from \eqref{eq:op-norm-constants} as $(L_{g,\mathtt{op}}, \sigma_{g',\mathtt{op}})$ and their parallels using the Frobenius norm as by $(L_{g,\mathtt{Frob}}, \sigma_{g',\mathtt{Frob}})$. Note $L_{g,\mathtt{op}} \leq L_{g,\mathtt{Frob}}$ and $\sigma_{g',\mathtt{op}} \leq \sigma_{g',\mathtt{Frob}}$.

When $g$ is given by a finite summation~\eqref{eq:finite-sum-g}, Corollaries~\ref{cor:exact-eval-plus-standard}-\ref{cor:exact-eval-plus-smart-bounds} show stationary points can be reached with high probability using at most $O(N+N^{4/5}\frac{L_{g,\mathtt{op}}}{\epsilon})$ evaluations of $g_j$ and $g_j'$, improving prior expectation guarantees of $O(N+N^{4/5}\frac{L_{g,\mathtt{Frob}}}{\epsilon})$.

When $g$ is given by an expectation~\eqref{eq:expectation-g}, prior works assuming stronger Frobenius norm bounds proved $O(\sigma_{g',\mathtt{Frob}}^2/\epsilon^{3/2})$ evaluations of $g'_\xi(x)$ suffice to reach expected stationarity. Corollaries~\ref{cor:sample-mean-plus-standard}-\ref{cor:sample-mean-plus-smart-bounds} of our unified, operator norm-based, variance-reduced theory achieve high probability stationarity guarantees of $O(\sigma_{g',\mathtt{op}}^2/\epsilon^{5/3})$. For example, this yields an improvement whenever $\frac{\sigma_{g', \mathtt{Frob}}}{\sigma_{g', \mathtt{op}}} \ge 1/\epsilon^{1/12}$.


    \section{Preliminaries} \label{Sect:Prelim}
First, we briefly summarize our basic notations. 
Let $O(\cdot)$ and $\Theta(\cdot)$ denote their standard asymptotic notations, both w.r.t $\epsilon\to 0$ and $N\to \infty$. In addition, we use $\widetilde{\Theta}$ instead of $\Theta$ to omit the multiplicative logarithmic terms in $\epsilon$.
For any distribution $D$, we denote its support by $\mathrm{supp}(D)$. Throughout, $\|\cdot\|_2$ is the $2$-norm on Euclidean space and $\|\cdot \|_{\rm{op}}$ is the spectral norm of a matrix. We use several notions of Lipschitz continuity: A vector-valued function $\varphi: \mathbb{R}^n \to \mathbb{R}^m$ is $l$-Lipschitz if $\|\varphi(x) - \varphi(y)\|_2 \le l \|x-y\|_2$ for any $x,y \in \mathbb{R}^n$, a matrix-valued function $\varphi: \mathbb{R}^n \to \mathbb{R}^{m_1 \times m_2}$ is $L$-Lipschitz if $\|\varphi(x) - \varphi(y) \|_{\rm{op}} \le L \|x-y\|_2$ for any $x,y \in \mathbb{R}^n$.
For a convex function $\varphi: \mathbb{R}^n \to \mathbb{R}\cup \{+\infty\}$, a vector $v \in \mathbb{R}^n$ is a subgradient of $\varphi$ at $x_0 \in \mathbb{R}^n$ if $\varphi(x) \ge \varphi(x_0) + \langle v, x-x_0 \rangle$ for all $x \in \mathbb{R}^n$. The subdifferential of $\varphi$ at $x_0$, defined as the set of all subgradients of $\varphi$ at $x_0$, is denoted by $\partial \varphi(x_0)$. For $M\ge 0$, a function $\varphi(x)$ is $M$-strongly convex if $\varphi(x) - \frac{M}{2}\|x\|_2^2$ is convex.


Throughout, we assume the following conditions hold for $f,g,h$ defining~\eqref{eq:main}:
\begin{enumerate}
\item The function $f: \mathbb{R}^m \to \mathbb{R}$ is convex and $l_f$-Lipschitz. 
\item {\color{blue} The function $g: \mathbb{R}^n \to \mathbb{R}^m$ is $l_g$-Lipschitz and its Jacobian $g': \mathbb{R}^n \to \mathbb{R}^{m\times n}$ is $L_g$-Lipschitz in the operator norm.}
\item The function $h: \mathbb{R}^n \to \mathbb{R}\cup \{+\infty\}$ is closed, convex, and proper.
\end{enumerate}
The Lipschitz conditions of $f$, {\color{blue} $g$, and $g'$ give the following fact.}

\begin{proposition}
\label{prop:linear-approx-in-f}
For any $x,y \in \mathbb{R}^n$,
\begin{equation*}
\left| f(g(x)) - f\Bigl(g(y)+g'(y)(x-y) \Bigr) \right| \le \frac{l_f L_g}{2} \|x-y\|_2^2 .
\end{equation*}
\end{proposition}

\subsection{A General Variance Reduced Prox-Linear Method}
Algorithm~\ref{algo:general-framework} presents the general method our unified theory covers.
This method proceeds via two nested loops.
As inputs, we require a total number of outer iterations to be run $K$ and a number of iterations for each inner loop $\tau_0,..., \tau_{K-1}$. A typical variance-reduced method may compute an exact or high-accuracy estimate of $g(x)$ and $g'(x)$ once per outer loop while using cheaper estimates at each inner loop. 

As useful notations, let $\Sigma_\tau = \sum_{k=0}^{K-1} \tau_k$ denote the total number of iterations. For $K\in \mathbb{N}_+$ and $\boldsymbol{\tau} \in \mathbb{N}_+^K$, we define the index set $\mathcal{I}(K, \boldsymbol{\tau}) = \{ (k,i)\in \mathbb{N}^2: 0\le k\le K-1, 0\le i\le \tau_k-1 \}$. So $\mathcal{I}(K, \boldsymbol{\tau} )$ corresponds to all the inner iterations in Algorithm \ref{algo:general-framework}.
Algorithm \ref{algo:general-framework} then proceeds following the general pattern of~\eqref{eq:sketch-of-method} with the $(k,i)$-th iteration consists of an estimation step and an optimization step, using a predefined $\mathtt{estimator}$ and $\mathtt{solver}$ as described below. {\color{blue} Below, we explain the role of $\theta$ for the estimation step and $(\overline{\epsilon}, \overline{\delta})$ as tolerances for the optimization step.}

\LinesNumbered
\SetAlgoLongEnd

\begin{algorithm2e}
\caption{Generalized Variance Reduced, Inexact Prox-Linear Method}
\label{algo:general-framework}
\KwIn{Initialization $x_0^0$, $M>0$, Iteration bounds $K$,$\boldsymbol{\tau}= (\tau_0,..., \tau_{K-1})$,
an estimation method $\mathtt{estimator}(x,i; \theta)$,
a solver $\mathtt{solver}(s,\overline{\epsilon}, \overline{\delta} )$. }
\For{$k=0,..., K-1$}{

\For{$i=0,..., \tau_k-1$}{
Compute $\widetilde{g}_i^k$ and $\widetilde{J}_i^k$ using the predefined method, $(\widetilde{g}_i^k, \widetilde{J}_i^k) \gets \mathtt{estimator}(x_i^k, i; \theta)$. \\
Minimize $s_i^k(x):= f(\widetilde{g}_i^k + \widetilde{J}_i^k (x- x_i^k)) + h(x) + \frac{M}{2} \|x- x_i^k\|_2^2$ by the known solver, and get an inexact solution $x_{i+1}^k \gets \mathtt{solver}(s_i^k,\overline{\epsilon}, \overline{\delta} )$.

}
Set $x_0^{k+1} = x_{\tau_k}^k$.
}
\end{algorithm2e}

\subsubsection{Estimation Step}
At each step $(k,i)$, Algorithm~\ref{algo:general-framework} requires an estimator $\mathtt{estimator}(x,i; \theta)$, treated for now as a black-box, which produces stochastic estimates of $g(x_i^k)$ and $g'(x_i^k)$, denoted $\widetilde{g}_i^k$ and $\widetilde{J}_i^k$. As examples, see the several estimators~\eqref{eq:mini-batch}--\eqref{eq:exact-eval-plus-smart} in Section~\ref{subsec:vr-corollaries}.

As indicated by our notation, the estimator $\mathtt{estimator}(x,i; \theta)$ is allowed to depend on $i$ but not $k$. For example, the most classic variance reduction~\cite{Johnson2013} computes an exact (or high accuracy) estimates of $g(x_0^k)$ and $g'(x_0^k)$ when $i=0$ and then leverage these past estimates to cheaply estimate $g(x_i^k)$ and $g'(x_i^k)$ when $i>0$. This process is repeated at every outer iteration $k$. In particular, the estimators considered here will have a ``memory'' of the most recent $x^k_0$ and potentially the component evaluations $g_\xi$ and $g_\xi'$ previously computed there.
All additional parameters of $\mathtt{estimator}$ are captured by $\theta$, taken from some space $\Theta$. For example, if $\mathtt{estimator}$ is some mini-batch method, then $\theta$ contains the batch sizes used at each iteration.

For our guarantees to apply, we require abstract high probability bounds on the estimation errors $\|\widetilde{g}_i^k- g(x_i^k)\|_2$ and $\|\widetilde{J}_i^k- g'(x_i^k) \|_{\rm{op}}$ that grow at most quadratically and linearly in $\|x_i^k - x_0^k\|_2$. This is natural since as $\|x_i^k - x_0^k\|_2$ grows, any variance reduction scheme leveraging a memory of $x^k_0$ ought to incur larger errors. Any additional constraints on the selection of the parameters $\theta$ are captured by $\mathcal{C}(K, \boldsymbol{\tau}, \Delta)$.
\begin{assumption}[Abstract bounds for estimation errors]
\label{assumption:general-estimation-error-bounds}
For a fixed $\mathtt{estimator}$, there exist five non-negative functions of $(K, \boldsymbol{\tau}, \theta, \Delta)$, denoted as $\gamma_0, \gamma_1, \gamma_2, \lambda_0, \lambda_1$, such that for any $K\in \mathbb{N}_+$, $\boldsymbol{\tau} \in \mathbb{N}_+^K$ and $\Delta\in (0,1)$, there exists a set $\mathcal{C}(K, \boldsymbol{\tau}, \Delta) \subseteq \Theta$ such that for any $\theta \in \mathcal{C}$, with probability at least $1-\Delta$, the following two inequalities simultaneously hold for all $(k,i)\in \mathcal{I}(K, \boldsymbol{\tau} )$: 
\begin{align*}
&\|\widetilde{g}_i^k- g(x_i^k)\|_2 \le \gamma_0(K, \boldsymbol{\tau}, \theta, \Delta) + \gamma_1(K, \boldsymbol{\tau}, \theta, \Delta) \|x_i^k - x_0^k\|_2 + \gamma_2(K, \boldsymbol{\tau}, \theta, \Delta) \|x_i^k - x_0^k\|_2^2 ,\\
&\|\widetilde{J}_i^k- g'(x_i^k) \|_{\rm{op}} \le \lambda_0(K, \boldsymbol{\tau}, \theta, \Delta) + \lambda_1(K, \boldsymbol{\tau}, \theta, \Delta) \|x_i^k - x_0^k\|_2 .
\end{align*}
\end{assumption}
{\color{blue} Note the use of operator norm above. Since $\|\cdot\|_{\rm{op}} \le \|\cdot\|_{\rm{F}}$, this is weaker than assuming upper bounds on $\|\widetilde{J}_i^k- g'(x_i^k) \|_{\rm{F}}$.} 
The functions $\{\gamma_\ell \}_{\ell=0}^2$ and $\{\lambda_\ell \}_{\ell=0}^1$ may also depend on some quantities like $m,n$ and the Lipschitz constants $l_f, l_g, L_g$. Since these are all fixed constants, we omit them and only keep the algorithmic parameters $(K, \boldsymbol{\tau}, \theta, \Delta)$ in the arguments of the functions. 
{\color{blue} Note we require the two estimation error bounds grow quadratically and linearly in $\|x^k_i - x^k_0\|_2$ respectively. One mild shortcoming of this framework is that it does not capture variance reduced estimators like SARAH/SPIDER~\cite{spider,sarah} which inductively set
$$ \widetilde{g}^k_i = \widetilde{g}^k_{i-1} + \frac{1}{|\mathcal{A}_i|}\sum_{\xi\in\mathcal{A}_i} (g_\xi(x^k_i) - g_\xi(x^k_{i-1}))$$
using $\widetilde{g}^k_{i-1}$ rather than $\widetilde{g}^k_{0}$ as a reference value.
For these methods, estimation error can grow faster, scaling with $\sum_{j=1}^i \|x^k_j - x^k_{j-1}\|_2$ rather than the smaller $\|x^k_i - x^k_0\|_2$. 
In Section \ref{subsec:vr-corollaries}, we still}
provide specific examples of $\mathtt{estimator}$ with explicit forms for the set $\mathcal{C}(K, \boldsymbol{\tau}, \Delta)$ and functions $\{\gamma_\ell \}_{\ell=0}^2$, $\{\lambda_\ell \}_{\ell=0}^1$.

\subsubsection{Optimization Step} \label{subsec:solvers}
In the optimization step, we need to (inexactly) solve the subproblem $\min s_i^k(x)$. Formally, we assume access to a known solver, $\mathtt{solver}(s, \overline{\epsilon}, \overline{\delta} )$, that returns an inexact solution $x_\text{sol}$. Algorithm~\ref{algo:general-framework} uses $\mathtt{solver}$ in black-box fashion, only requiring the following assumption:
\begin{assumption}
\label{assumption:subroutine}
For $\overline{\epsilon}, \overline{\delta} >0$ and the problem $\min_x s(x)$, with probability at least $1- \overline{\delta}$, $\mathtt{solver}(s, \overline{\epsilon}, \overline{\delta} )$ returns an $\overline{\epsilon}$-optimal solution $x_\text{sol}$, i.e., $s(x_\text{sol}) \le \inf_x s(x) + \overline{\epsilon}$.
\end{assumption}
{\color{blue} Note this assumption holds for any stochastic method with an expectation guarantee, as Markov's inequality then ensures a high-probability guarantee at the cost of a $1/\overline{ \delta}$ factor. Stronger high-probability results were shown in earlier works like~\cite{Hazan2014,Rakhlin2012}, and by the refined analysis of~\cite{harvey19a}. More recently, the procedure given by~\cite{Davis2021} converts a wide class of stochastic algorithms into high-probability guaranteed methods.}
{\color{blue} We also consider four example subroutines as possible instantiations of $\mathtt{solver}$,} and provide bounds on the resulting total oracle complexities with respect to $f$ in Section~\ref{subsec:inexact-accounting}.


\section{Main Results} \label{Sect:MainResults}
Given any estimator and solver satisfying Assumptions~\ref{assumption:general-estimation-error-bounds} and~\ref{assumption:subroutine}, our main result provides a general set of conditions for algorithmic parameters which guarantees the production of an $\epsilon$-stationary point with high probability.
\begin{theorem}
\label{thm:unified-thm-short}
Suppose Assumption \ref{assumption:general-estimation-error-bounds} holds for $\mathtt{estimator}$, and Assumption \ref{assumption:subroutine} holds for $\mathtt{solver}$. Assume $\Phi^* := \inf_x \Phi(x) > -\infty$. Fix an $M>5 l_f L_g$. For any $\Delta \in (0,1)$ and $\epsilon >0$, with probability at least $1- \Delta$, Algorithm \ref{algo:general-framework}'s iterates satisfy:
\begin{equation*}
\frac{1}{\Sigma_\tau} \sum_{k=0}^{K-1} \sum_{i=0}^{\tau_k -1} \| \mathcal{G}_M(x_i^k) \|_2 ^2 \le \epsilon ,
\end{equation*}
provided the parameters $K, \boldsymbol{\tau}, \theta, \overline{\epsilon}, \overline{\delta}$ satisfy\footnote{
The $\tau_{\text{max}}$ here denotes $\max\{ \tau_0,..., \tau_{K-1}\}$.
}
\begin{align}
\label{eq: thm-unified-thm-short,eq-1} \theta &\in \mathcal{C}(K, \boldsymbol{\tau}, \Delta /2 ) ,\\
\label{eq: thm-unified-thm-short,eq-2} \overline{\delta} &\le \Delta/(2 \Sigma_\tau ) ,\\
\label{eq: thm-unified-thm-short,eq-3} \overline{\epsilon} &\le \epsilon/(5\cdot 30M) ,\\
\label{eq: thm-unified-thm-short,eq-4} \Sigma_\tau &\ge 5\cdot 30M (\Phi(x_0^0) - \Phi^*) / \epsilon ,\\
\label{eq: thm-unified-thm-short,eq-5} \gamma_0(K, \boldsymbol{\tau}, \theta, \Delta/2 ) &\le \epsilon/(5\cdot 125 l_f M) ,\\
\label{eq: thm-unified-thm-short,eq-6} \lambda_0^2(K, \boldsymbol{\tau}, \theta, \Delta/2 ) &\le L_g \epsilon/(5\cdot 95 l_f M) ,\\
\label{eq: thm-unified-thm-short,eq-7} \color{blue} \tau_{\text{max}}^2 \gamma_1^2(K, \boldsymbol{\tau}, \theta, \Delta/2 ) &\le \color{blue} L_g \epsilon/(5\cdot 135 l_f M) , \\
\label{eq: thm-unified-thm-short,eq-8} \color{blue} \tau_{\text{max}}^2 \gamma_2(K, \boldsymbol{\tau}, \theta, \Delta/2 ) &\le \color{blue} 6 L_g /25 ,\\
\label{eq: thm-unified-thm-short,eq-9} \color{blue} \tau_{\text{max}}^2 \lambda_1^2(K, \boldsymbol{\tau}, \theta, \Delta/2 ) &\le \color{blue} 6 L_g^2 /19 .
\end{align}
\end{theorem}
{\color{blue} The conditions \eqref{eq: thm-unified-thm-short,eq-1}--\eqref{eq: thm-unified-thm-short,eq-9} can be viewed as constraints on algorithmic parameters. By increasing batch sizes and/or epoch lengths (hence tightening $\gamma_\ell,\lambda_\ell$) and choosing $\overline\epsilon,\overline\delta$ small enough, the conditions \eqref{eq: thm-unified-thm-short,eq-2}--\eqref{eq: thm-unified-thm-short,eq-9} can be met for any given $\epsilon$.}

\subsection{Convergence Rate Corollaries for a Range of VR Schemes} \label{subsec:vr-corollaries}
Next we apply this result to several estimation schemes, providing optimized algorithmic parameters (e.g., batch sizes, loop durations $\tau_k$). These applications all amount to simple applications of concentration inequalities to establish a lemma ensuring Assumption~\ref{assumption:general-estimation-error-bounds} and then calculations based on Theorem~\ref{thm:unified-thm-short} to provide optimized parameter selections and guarantees. 
{\color{blue} While the resulted corollaries provide some interesting insights, the way we verify Assumption~\ref{assumption:general-estimation-error-bounds} is conceptually standard as appeared in prior works~\cite{tran-dinh2020a,Zhang2021}, except we upper bound $\|\widetilde{J}_i^k - g'(x_i^k)\|_{\rm{op}}$ differently because of the operator norm.}
Such sample calculations deriving Corollaries~\ref{cor:sample-mean-plus-standard}-\ref{cor:sample-mean-plus-standard-bounds} are given in Section~\ref{sec:sample-derivations}. As all remaining derivations of corollaries are effectively identical, they are omitted. An interested reader can find them online at~\cite{wu2024unifiedtheory}.

\noindent {\bf The Mini-Batch Method.} We first discuss a simple mini-batch method as a warm-up example. 
At the $(k,i)$-th iteration, we generate an index set $\mathcal{A}_i^k$ of size $A$ and another index set $\mathcal{B}_i^k$ of size $B$, both by sampling from distribution $D$. Then we construct $\widetilde{g}_i^k$ and $\widetilde{J}_i^k$ using the sample mean over the index sets, parameterized by $\theta = (A,B) \in \mathbb{N}_+^2$. We can explicitly express this estimator for use in Algorithm~\ref{algo:general-framework} as
\begin{equation}
\tag{Est\textsubscript{0}}
\label{eq:mini-batch}
\mathtt{estimator}_0 (x_i^k, i; \theta) :
\begin{cases}
\widetilde{g}_i^k = \frac{1}{A} \sum_{\xi\in \mathcal{A}_i^k} g_{\xi}(x_i^k) \\
\widetilde{J}_i^k = \frac{1}{B} \sum_{\xi\in \mathcal{B}_i^k} g'_{\xi}(x_i^k) .
\end{cases}
\end{equation}

The construction above is for the expectation setting in (\ref{eq:expectation-g}). For the special finite average case in (\ref{eq:finite-sum-g}), it reduces to sampling with replacement from $\{1,...,N\}$, then $\widetilde{g}_i^k = \frac{1}{A} \sum_{j\in \mathcal{A}_i^k} g_{j}(x_i^k)$ and $\widetilde{J}_i^k = \frac{1}{B} \sum_{j\in \mathcal{B}_i^k} g'_{j}(x_i^k)$. 
To control the estimation error of $\widetilde{g}_i^k$ and $\widetilde{J}_i^k$, we need the following assumption.
\begin{assumption}
\label{assumption:variance-like}
There exist constants $\sigma_g$ and $\sigma_{g'}$ such that for any $\xi \in \mathrm{supp}(D)$ and any $x\in \mathbb{R}^n$,
$\|g_{\xi}(x) - g(x)\|_2 \le \sigma_g$ and $\|g'_{\xi}(x) - g'(x) \|_{\rm{op}} \le \sigma_{g'}$.
\end{assumption}

Such uniform bounds suffice to ensure Assumption \ref{assumption:general-estimation-error-bounds} holds for the above mini-batching estimator. The following lemma provides explicit values for the associated set $\mathcal{C}(K, \boldsymbol{\tau}, \Delta)$, and functions $\{\gamma_\ell \}_{\ell=0}^2$, $\{\lambda_\ell \}_{\ell=0}^1$. This lemma follows as a consequence of standard concentration inequalities.

\begin{lemma}
\label{lemma:mini-batch}
Suppose Assumption \ref{assumption:variance-like} holds. If $\mathtt{estimator}(x_i^k, i; \theta)$ is defined by (\ref{eq:mini-batch}), where $\theta= (A,B)$ and $\Theta= \mathbb{N}_+^2$, then Assumption \ref{assumption:general-estimation-error-bounds} holds with the following choices of $\mathcal{C}(K, \boldsymbol{\tau}, \Delta)$, $\{\gamma_\ell \}_{\ell=0}^2$ and $\{\lambda_\ell \}_{\ell=0}^1$:
\begin{align*}
    &\mathcal{C}(K, \boldsymbol{\tau}, \Delta) = \left\{ (A,B) \in \mathbb{N}_+^2 : A \ge \frac{4}{9} \log \left( \frac{2(m+ 1) \Sigma_\tau}{\Delta} \right), \text{and } B \ge \frac{4}{9} \log \left( \frac{2(m+ n) \Sigma_\tau}{\Delta} \right) \right\} ,\\
    &\gamma_0(K, \boldsymbol{\tau}, \theta, \Delta ) = \frac{2 \sigma_g}{\sqrt{A}} \sqrt{\log \left( \frac{2(m+ 1) \Sigma_\tau}{\Delta} \right)}, \quad \lambda_0(K, \boldsymbol{\tau}, \theta, \Delta ) = \frac{2 \sigma_{g'}}{\sqrt{B}} \sqrt{\log \left( \frac{2(m+ n) \Sigma_\tau}{\Delta} \right)},\\
    &\gamma_1= \gamma_2= \lambda_1=0 .
\end{align*}
\end{lemma}

Substituting the results of Lemma \ref{lemma:mini-batch} into conditions (\ref{eq: thm-unified-thm-short,eq-1})--(\ref{eq: thm-unified-thm-short,eq-9}), Theorem \ref{thm:unified-thm-short} provides constraints on the parameters $K, \boldsymbol{\tau}, \theta$ which guarantee minibatching produces a stationary point with high probability. Furthermore, since $(K, \boldsymbol{\tau})$ controls the number of iterations in Algorithm \ref{algo:general-framework} and $\theta= (A,B)$ determines the batch sizes at each evaluation, one can optimize their selection over this feasible region. Directly doing so, the following corollary provides such optimized choices. 

\begin{corollary}
\label{cor:mini-batch}
Consider any $\Delta \in (0,1)$, $M>5 l_f L_g$, and any sufficiently small $\epsilon >0$. Suppose Assumption \ref{assumption:subroutine} holds for $\mathtt{solver}$, Assumption \ref{assumption:variance-like} holds for $g$, $\inf_x \Phi(x) > -\infty$, and $\mathtt{estimator}$ is defined by (\ref{eq:mini-batch}). Set $\Sigma_\tau = \lceil C_{\Sigma} \cdot \epsilon^{-1} \rceil$, $A = \lceil C_A \cdot \epsilon^{-2} \cdot \log ( \frac{4(m+ 1) \Sigma_\tau}{\Delta} ) \rceil$, $B = \lceil C_B \cdot \epsilon^{-1} \cdot \log ( \frac{4(m+ n) \Sigma_\tau}{\Delta} ) \rceil$, $\overline{\delta} \le \Delta/(2 \Sigma_\tau )$, $\overline{\epsilon} \le \epsilon/(5\cdot 30M)$, where $C_{\Sigma}, C_A, C_B$ are some constants, 
then with probability at least $1-\Delta$: (i) Algorithm $\ref{algo:general-framework}$'s iterates satisfy:
\begin{equation*}
\frac{1}{\Sigma_\tau} \sum_{k=0}^{K-1} \sum_{i=0}^{\tau_k -1} \| \mathcal{G}_M(x_i^k) \|_2 ^2 \le \epsilon ,
\end{equation*}
and (ii) the oracle complexities for evaluations and Jacobians of inner components $g_{\xi}(\cdot)$ respectively are at most 
\begin{equation*}
\widetilde{\Theta} \left( \epsilon^{-3} \log(1/\Delta) \right) \quad\text{and}\quad \widetilde{\Theta} \left( \epsilon^{-2} \log(1/\Delta) \right) . 
\end{equation*}
\end{corollary}

Note the two complexities $\widetilde{\Theta} \left( \epsilon^{-3} \log(1/\Delta) \right)$ and $\widetilde{\Theta} \left( \epsilon^{-2} \log(1/\Delta) \right)$ match the high probability guarantees in \cite{tran-dinh2020a}. Up to logarithm terms, our high probability results also agree with the expectation results of~\cite{Zhang2021}. In both cases, our theory improves prior Frobenius norm bounds to matrix operator norms. 

\subsubsection{Expectation Case Methods}
\label{subsubsec:expectation-methods}
Given $g$ is defined as an expectation~\eqref{eq:expectation-g}, we consider two different variance reduced schemes below, following the forms of~\eqref{eq:basic-VR} and~\eqref{eq:smart-VR}. Our unified theorem's guarantees for the first scheme requires fewer evaluations of $g_\xi'$ while the second requires fewer evaluations of $g_\xi$. As a result, both methods may be state-of-the-art depending on the relative cost of these two operations.

\noindent {\bf Application of Standard Variance Reduction for Expectations.}
First, we consider the variance-reduced estimator, defined in two cases, $i=0$ and $i>0$, as
\begin{equation}
\label{eq:sample-mean-plus-standard}
\tag{Est\textsubscript{1}}
\mathtt{estimator}_1 (x_i^k, i; \theta) :
\begin{cases}
\widetilde{g}_0^k = \frac{1}{A} \sum_{\xi\in \mathcal{A}_0^k} g_{\xi}(x_0^k) \\
\widetilde{J}_0^k = \frac{1}{B} \sum_{\xi\in \mathcal{B}_0^k} g'_{\xi}(x_0^k) \\
\widetilde{g}_i^k = \frac{1}{a} \sum_{\xi\in \mathcal{A}_i^k} \Bigl( g_{\xi}(x_i^k) - g_{\xi}(x_0^k) \Bigr) + \widetilde{g}_0^k\\
\widetilde{J}_i^k = \frac{1}{b} \sum_{\xi\in \mathcal{B}_i^k} \Bigl( g'_{\xi}(x_i^k) - g'_{\xi}(x_0^k) \Bigr) + \widetilde{J}_0^k.
\end{cases}
\end{equation}
This estimator simply applies the classic variance reduced update~\eqref{eq:basic-VR} independently to estimate both $g_\xi$ and $g_\xi'$: At the $(k,i)$-th iteration, we generate index sets $\mathcal{A}_i^k$ and $\mathcal{B}_i^k$ by sampling from distribution $D$. At the start of each epoch, namely $i=0$, the batch sizes are set to be $|\mathcal{A}_0^k| = A$ and $|\mathcal{B}_0^k| = B$. We still use the sample mean to construct $\widetilde{g}_0^k$ and $\widetilde{J}_0^k$, same as the mini-batch method. In the case $i>0$, we set $|\mathcal{A}_i^k| = a$ and $|\mathcal{B}_i^k| = b$, with $a<A$ and $b<B$.
It is also worth noting that, unlike the mini-batch method, $\mathtt{estimator}_1$ is history-dependent, since the construction of $\widetilde{g}_i^k$ and $\widetilde{J}_i^k$ 
{\color{blue} involve both the current iterate $x_i^k$ and the past iterate $x_0^k$. 
As a consequence, we need the following assumption to control the estimation error of $\widetilde{g}_i^k$ and $\widetilde{J}_i^k$ in the same spirit of Assumption~\ref{assumption:variance-like}.
\begin{assumption}
\label{assumption:uniform-Lipschitz}
There exist constants $\widehat{l}_g$ and $\widehat{L}_g$ such that for any $\xi \in \mathrm{supp}(D)$ and any $x,y\in \mathbb{R}^n$, $\|g_\xi(x)- g_\xi(y)\|_2 \le \widehat{l}_g \|x-y\|_2$ and $\|g'_\xi(x)- g'_\xi(y)\|_{\rm{op}} \le \widehat{L}_g \|x-y\|_2$.
\end{assumption}
Lipschitz continuity of $g_\xi$ and $g'_\xi$ uniformly for any $\xi$ enables us to use Bernstein inequality (Lemma~\ref{lemma:Matrix-Bernstein-corollary}) to verify Assumption \ref{assumption:general-estimation-error-bounds}. Different assumptions and concentration results could be applied instead. For example, our use of Bernstein could be tightened by additionally using mean-square bounds. Alternatively, one could leverage concentration inequalities~\cite{tropp2012user} directly based on moment bounds rather than uniform bounds. Any approach to validating Assumption \ref{assumption:general-estimation-error-bounds} suffices for our theory.
}

In (\ref{eq:sample-mean-plus-standard}), the parameter of $\mathtt{estimator}_1$ is $\theta = (A,B,a,b) \in \mathbb{N}_+^4$, which captures the batch sizes. The set $\mathcal{C}(K, \boldsymbol{\tau}, \Delta)$ and functions $\{\gamma_\ell \}_{\ell=0}^2$, $\{\lambda_\ell \}_{\ell=0}^1$ are given below.

\begin{lemma}
\label{lemma:sample-mean-plus-standard}
Suppose {\color{blue} Assumptions~\ref{assumption:variance-like},~\ref{assumption:uniform-Lipschitz} hold}. If $\mathtt{estimator}(x_i^k, i; \theta)$ is defined by (\ref{eq:sample-mean-plus-standard}), where $\theta= (A,B,a,b)$ and $\Theta= \mathbb{N}_+^4$, then Assumption \ref{assumption:general-estimation-error-bounds} holds with 
\begin{align*}
    &\mathcal{C}(K, \boldsymbol{\tau}, \Delta) = \left\{ (A,B,a,b) \in \mathbb{N}_+^4 : A,a \ge \frac{4}{9} \log \left( \frac{2(m+ 1) \Sigma_\tau}{\Delta} \right), \text{and } B,b \ge \frac{4}{9} \log \left( \frac{2(m+ n) \Sigma_\tau}{\Delta} \right) \right\} ,\\
    &\gamma_0(K, \boldsymbol{\tau}, \theta, \Delta ) = \frac{2 \sigma_g}{\sqrt{A}} \sqrt{\log \left( \frac{2(m+ 1) \Sigma_\tau}{\Delta} \right)}, \quad {\color{blue} \gamma_1(K, \boldsymbol{\tau}, \theta, \Delta ) = \frac{4 \widehat{l}_g}{\sqrt{a}} \sqrt{\log \left( \frac{2(m+ 1) \Sigma_\tau}{\Delta} \right)} }, \quad \gamma_2= 0,\\
    &\lambda_0(K, \boldsymbol{\tau}, \theta, \Delta ) = \frac{2 \sigma_{g'}}{\sqrt{B}} \sqrt{\log \left( \frac{2(m+ n) \Sigma_\tau}{\Delta} \right)}, \quad {\color{blue} \lambda_1(K, \boldsymbol{\tau}, \theta, \Delta ) = \frac{4 \widehat{L}_g}{\sqrt{b}} \sqrt{\log \left( \frac{2(m+ n) \Sigma_\tau}{\Delta} \right)} }.
\end{align*}
\end{lemma}

Using Lemma \ref{lemma:sample-mean-plus-standard} enables us to select the parameters $K, \boldsymbol{\tau}, \theta$ in Theorem~\ref{thm:unified-thm-short} and analyze the oracle complexities there yielding the following pair of results. Proofs of this lemma and both resulting corollaries are given in Section~\ref{sec:sample-derivations}.

\begin{corollary}[Algorithmic guarantee]
\label{cor:sample-mean-plus-standard}
Consider any $\Delta \in (0,1)$, $M>5 l_f L_g$, integer $\tau>0$, and any sufficiently small $\epsilon >0$. Suppose Assumption \ref{assumption:subroutine} holds for $\mathtt{solver}$, {\color{blue} Assumptions~\ref{assumption:variance-like},~\ref{assumption:uniform-Lipschitz} hold} for function $g$, $\inf_x \Phi(x) > -\infty$, $\mathtt{estimator}$ is defined by (\ref{eq:sample-mean-plus-standard}), and $\boldsymbol{\tau}$ is restricted to the form $\tau_0= \cdots= \tau_{K-1} = \tau$. Set $K = \lceil \frac{C_{\Sigma} \cdot \epsilon^{-1}}{\tau} \rceil$, $A = \lceil C_A \cdot \epsilon^{-2} \cdot \log ( \frac{4(m+ 1) K \tau}{\Delta} ) \rceil$, $B = \lceil C_B \cdot \epsilon^{-1} \cdot \log ( \frac{4(m+ n) K \tau}{\Delta} ) \rceil$, {\color{blue} $a = \lceil C_a \cdot \tau^2 \cdot \epsilon^{-1} \cdot \log ( \frac{4(m+ 1) K \tau}{\Delta} ) \rceil$, $b = \lceil C_b \cdot \tau^2 \cdot \log ( \frac{4(m+ n) K \tau}{\Delta} ) \rceil$}, $\overline{\delta} = \Delta/ (2K\tau)$, $\overline{\epsilon} = \epsilon/(5\cdot 30M)$, where $C_{\Sigma}, C_A, C_B, C_a, C_b$ are some constants, 
then with probability at least $1-\Delta$: (i) Algorithm $\ref{algo:general-framework}$'s iterates satisfy:
\begin{equation*}
\frac{1}{K \tau} \sum_{k=0}^{K-1} \sum_{i=0}^{\tau -1} \| \mathcal{G}_M(x_i^k) \|_2 ^2 \le \epsilon,
\end{equation*}
and (ii) provided $\tau = O(\epsilon^{-1})$, the oracle complexities for evaluations and Jacobians of inner components $g_{\xi}(\cdot)$ respectively are at most 
\begin{align}
\label{eq: cor-sample-mean-plus-standard,eq-1}
&\widetilde{\Theta} \left( (\epsilon^{-3} \tau^{-1} + 
\epsilon^{-2} \tau^2) \log(1/\Delta) \right)\\
\label{eq: cor-sample-mean-plus-standard,eq-2}
\text{and}\quad &\widetilde{\Theta} \left( (\epsilon^{-2} \tau^{-1} + \epsilon^{-1} \tau^2) \log(1/\Delta) \right) . 
\end{align}
\end{corollary}
The oracle complexity upper bounds for evaluations and Jacobians in Corollary \ref{cor:sample-mean-plus-standard} can be optimized through careful selection of the epoch length $\tau$.
\begin{corollary}[Optimized complexity bounds]
\label{cor:sample-mean-plus-standard-bounds}
The minimal asymptotic rates, with respect to $\tau$, of (\ref{eq: cor-sample-mean-plus-standard,eq-1}) and (\ref{eq: cor-sample-mean-plus-standard,eq-2}) are $\widetilde{\Theta} (\epsilon^{-8/3} \log(1/\Delta))$ and $\widetilde{\Theta} (\epsilon^{-5/3} \log(1/\Delta))$ respectively, which are simultaneously achieved by setting $\tau = \Theta(\epsilon^{-1/3})$.
\end{corollary}
Comparing the two rates in Corollary \ref{cor:sample-mean-plus-standard-bounds} with those in Corollary \ref{cor:mini-batch} suggests the variance reduced estimator $\mathtt{estimator}_1$ dominates the mini-batch method $\mathtt{estimator}_0$ in the sense that, $\mathtt{estimator}_1$ achieves $\frac{1}{K \tau} \sum_{k=1}^K \sum_{i=0}^{\tau -1} \| \mathcal{G}_M(x_i^k) \|_2 ^2 \le \epsilon$ in high probability with less evaluations of $g_\xi(\cdot)$ and $g'_\xi(\cdot)$.

\noindent {\bf Application of Modified Variance Reduction for Expectations.}
Next, we consider a modified variance-reduced estimator that leverages Jacobian evaluations to provide a better estimate of the value of $g_\xi(x)$. Again we consider $i=0$ and $i>0$ with
$\mathtt{estimator}_2 (x_i^k, i; \theta)$ defined as
\begin{equation}
\label{eq:sample-mean-plus-smart}
\tag{Est\textsubscript{2}}
\begin{cases}
\widetilde{g}_0^k = \frac{1}{A} \sum_{\xi\in \mathcal{A}_0^k} g_{\xi}(x_0^k) \\
\widetilde{J}_0^k = \frac{1}{B} \sum_{\xi\in \mathcal{B}_0^k} g'_{\xi}(x_0^k) \\
\widetilde{g}_i^k = \frac{1}{a} \sum_{\xi\in \mathcal{A}_i^k} \Bigl( g_{\xi}(x_i^k) - g_{\xi}(x_0^k) - g'_{\xi}(x_0^k) (x_i^k-x_0^k) \Bigr) + \widetilde{g}_0^k + \widetilde{J}_0^k (x_i^k-x_0^k) \\
\widetilde{J}_i^k = \frac{1}{b} \sum_{\xi\in \mathcal{B}_i^k} \Bigl( g'_{\xi}(x_i^k) - g'_{\xi}(x_0^k) \Bigr) + \widetilde{J}_0^k .
\end{cases}
\end{equation}
Here~\eqref{eq:sample-mean-plus-smart} applies a standard variance reduction update~\eqref{eq:basic-VR} to estimate the Jacobian and a first-order corrected update~\eqref{eq:smart-VR} to estimate the value of $g$ itself. Again this estimator is parameterized by the four batch sizes $\theta = (A,B,a,b) \in \mathbb{N}_+^4$. The following lemma characterizes this modified scheme in terms of Assumption~\ref{assumption:general-estimation-error-bounds}.

\begin{lemma}
\label{lemma:sample-mean-plus-smart}
Suppose {\color{blue} Assumptions~\ref{assumption:variance-like},~\ref{assumption:uniform-Lipschitz} hold}. If $\mathtt{estimator}(x_i^k, i; \theta)$ is defined by (\ref{eq:sample-mean-plus-smart}), where $\theta= (A,B,a,b)$ and $\Theta= \mathbb{N}_+^4$, then Assumption \ref{assumption:general-estimation-error-bounds} holds with 
\begin{align*}
    &\mathcal{C}(K, \boldsymbol{\tau}, \Delta) = \left\{ (A,B,a,b) \in \mathbb{N}_+^4 : A,a \ge \frac{4}{9} \log \left( \frac{2(m+ 1) \Sigma_\tau}{\Delta} \right), \text{and } B,b \ge \frac{4}{9} \log \left( \frac{2(m+ n) \Sigma_\tau}{\Delta} \right) \right\} ,\\
    &\gamma_0(K, \boldsymbol{\tau}, \theta, \Delta ) = \frac{2 \sigma_g}{\sqrt{A}} \sqrt{\log \left( \frac{2(m+ 1) \Sigma_\tau}{\Delta} \right)}, \quad \gamma_1(K, \boldsymbol{\tau}, \theta, \Delta ) = \lambda_0(K, \boldsymbol{\tau}, \theta, \Delta ) = \frac{2 \sigma_{g'}}{\sqrt{B}} \sqrt{\log \left( \frac{2(m+ n) \Sigma_\tau}{\Delta} \right)},\\
    &{\color{blue} \gamma_2(K, \boldsymbol{\tau}, \theta, \Delta ) = \frac{2 \widehat{L}_g}{\sqrt{a}} \sqrt{\log \left( \frac{2(m+ 1) \Sigma_\tau}{\Delta} \right)}, \quad \lambda_1(K, \boldsymbol{\tau}, \theta, \Delta ) = \frac{4 \widehat{L}_g}{\sqrt{b}} \sqrt{\log \left( \frac{2(m+ n) \Sigma_\tau}{\Delta} \right)} }.
\end{align*}
\end{lemma}

From this, we have the following two corollaries analyzing $\mathtt{estimator}_2$.

\begin{corollary}[Algorithmic guarantee]
\label{cor:sample-mean-plus-smart}
Consider any $\Delta \in (0,1)$, $M>5 l_f L_g$, integer $\tau>0$, and any sufficiently small $\epsilon >0$. Suppose Assumption \ref{assumption:subroutine} holds for $\mathtt{solver}$, {\color{blue} Assumptions~\ref{assumption:variance-like},~\ref{assumption:uniform-Lipschitz} hold} for function $g$, $\inf_x \Phi(x) > -\infty$, $\mathtt{estimator}$ is defined by (\ref{eq:sample-mean-plus-smart}), and $\boldsymbol{\tau}$ is restricted to the form $\tau_0= \cdots= \tau_{K-1} = \tau$. Set $K = \lceil \frac{C_{\Sigma} \cdot \epsilon^{-1}}{\tau} \rceil$, $A = \lceil C_A \cdot \epsilon^{-2} \cdot \log ( \frac{4(m+ 1) K \tau}{\Delta} ) \rceil$, {\color{blue} $B = \lceil C_B \cdot \tau^2 \cdot \epsilon^{-1} \cdot \log ( \frac{4(m+ n) K \tau}{\Delta} ) \rceil$, $a = \lceil C_a \cdot \tau^4 \cdot \log ( \frac{4(m+ 1) K \tau}{\Delta} ) \rceil$, $b = \lceil C_b \cdot \tau^2 \cdot \log ( \frac{4(m+ n) K \tau}{\Delta} ) \rceil$}, $\overline{\delta} = \Delta/ (2K\tau)$, $\overline{\epsilon} = \epsilon/(5\cdot 30M)$, where $C_{\Sigma}, C_A, C_B, C_a, C_b$ are some constants, 
then with probability at least $1-\Delta$: (i) Algorithm $\ref{algo:general-framework}$'s iterates satisfy:
\begin{equation*}
\frac{1}{K \tau} \sum_{k=0}^{K-1} \sum_{i=0}^{\tau -1} \| \mathcal{G}_M(x_i^k) \|_2 ^2 \le \epsilon,
\end{equation*}
and (ii) provided $\tau = O(\epsilon^{-1})$, the oracle complexities for evaluations and Jacobians of inner components $g_{\xi}(\cdot)$ respectively are at most
\begin{align}
\label{eq: cor-sample-mean-plus-smart,eq-1}
&\widetilde{\Theta} \left( (\epsilon^{-3} \tau^{-1} + \epsilon^{-1} \tau^4) \log(1/\Delta) \right)\\
\label{eq: cor-sample-mean-plus-smart,eq-2}
\text{and}\quad &\widetilde{\Theta} \left( (\epsilon^{-2} \tau + \epsilon^{-1} \tau^4) \log(1/\Delta) \right) .
\end{align}
\end{corollary}

\begin{corollary}[Optimized complexity bounds]
\label{cor:sample-mean-plus-smart-bounds}
(i) The minimal asymptotic evaluation complexity, with respect to $\tau$, of (\ref{eq: cor-sample-mean-plus-smart,eq-1}) is $\widetilde{\Theta} (\epsilon^{-13/5} \log(1/\Delta))$, achieved by setting $\tau = \Theta(\epsilon^{-2/5})$. In this case, (\ref{eq: cor-sample-mean-plus-smart,eq-2}) is also $\widetilde{\Theta} (\epsilon^{-13/5} \log(1/\Delta))$. 
(ii) The minimal asymptotic Jacobian complexity, with respect to $\tau$, of (\ref{eq: cor-sample-mean-plus-smart,eq-2}) is $\widetilde{\Theta} (\epsilon^{-2} \log(1/\Delta))$, achieved at $\tau = \Theta(1)$. In this case, (\ref{eq: cor-sample-mean-plus-smart,eq-1}) is $\widetilde{\Theta} (\epsilon^{-3} \log(1/\Delta))$.
\end{corollary}
\begin{remark} \label{remark:comparison-expectation-methods}
We can compare the asymptotic rates in Corollary \ref{cor:sample-mean-plus-smart-bounds} with those in Corollary \ref{cor:sample-mean-plus-standard-bounds}. Note that $\frac{13}{5} < \frac{8}{3}$ and $\frac{5}{3} < 2$. So Corollary \ref{cor:sample-mean-plus-smart-bounds}(i) suggests the optimized asymptotic evaluation complexity bound of $\mathtt{estimator}_2$ is lower than that of $\mathtt{estimator}_1$. Corollary \ref{cor:sample-mean-plus-smart-bounds}(ii) implies that the asymptotic Jacobian complexity bound of $\mathtt{estimator}_2$ is always higher than the optimized bound of $\mathtt{estimator}_1$. Hence neither method's guarantee uniformly dominates the other. Depending on the relative cost between evaluations and Jacobians, the best method varies.
\end{remark}
{\color{blue} To the best of our knowledge, this is the first guarantee for either SVRG-type estimators~\eqref{eq:basic-VR} or~\eqref{eq:smart-VR} on composite problems in the expectation setting~\eqref{eq:expectation-g}. Depending on the relative sizes of problem constants, these results can improve on the SARAH/SPIDER-based methods studied by~\cite{tran-dinh2020a,Zhang2021}, as discussed in Section~\ref{subsec:relatedWork}.}

\subsubsection{Finite Average Case Methods}
\label{subsubsec:finite-average-methods}
Now we focus on the finite average setting in \eqref{eq:finite-sum-g} and consider the natural extensions of the above estimators. Again, we find neither one of these two estimator's guarantees dominates the other. 
{\color{blue} In this case, the comparison between these two methods' guarantees depends on the relative size of $1/\epsilon$ and the number of summands $N$.}

\noindent {\bf Application of Standard Variance Reduction for Finite Averages.}
First, we consider a variant of $\mathtt{estimator}_1$, defined in two cases, $i=0$ and $i>0$, as
\begin{equation}
\label{eq:exact-eval-plus-standard}
\tag{Est\textsubscript{3}}
\mathtt{estimator}_3 (x_i^k, i; \theta) :
\begin{cases}
\widetilde{g}_0^k = \frac{1}{N} \sum_{j=1}^N g_j(x_0^k) = g(x_0^k) \\
\widetilde{J}_0^k = \frac{1}{N} \sum_{j=1}^N g'_j(x_0^k) = g'(x_0^k) \\
\widetilde{g}_i^k = \frac{1}{a} \sum_{j\in \mathcal{A}_i^k} \Bigl( g_j(x_i^k) - g_j(x_0^k) \Bigr) + \widetilde{g}_0^k \\
\widetilde{J}_i^k = \frac{1}{b} \sum_{j\in \mathcal{B}_i^k} \Bigl( g'_j(x_i^k) - g'_j(x_0^k) \Bigr) + \widetilde{J}_0^k .
\end{cases}
\end{equation}
For each $k=1,...,K$, at the start of the $k$-th epoch, this estimator now constructs $\widetilde{g}_0^k$ and $\widetilde{J}_0^k$ exactly. At the iterations with $i>0$, we generate an index set $\mathcal{A}_i^k$ of size $a$ and another index set $\mathcal{B}_i^k$ of size $b$, both by sampling with replacement from $\{1,\dots,N\}$. 
Note since $g_j(x_0^k)$ and $g'_j(x_0^k)$ are all evaluated for all $j \in \{1,\dots,N\}$, one can store these in memory for use later in the construction of $\widetilde{g}_i^k$ and $\widetilde{J}_i^k$. So at the $(k,i)$-th iteration (for $i>0$), the terms in $\{ g_j(x_0^k): j\in \mathcal{A}_i^k \}$ and $\{ g'_j(x_0^k): j\in \mathcal{B}_i^k \}$ can be simply called from the past data. Only the terms in $\{ g_j(x_i^k): j\in \mathcal{A}_i^k \}$ and $\{ g'_j(x_i^k): j\in \mathcal{B}_i^k \}$, namely those involves $x_i^k$, are needed to be evaluated. 
This estimator is parameterized by $\theta = (a,b) \in \mathbb{N}_+^2$, describing both batch sizes utilized. The set $\mathcal{C}(K, \boldsymbol{\tau}, \Delta)$, and functions $\{\gamma_\ell \}_{\ell=0}^2$, $\{\lambda_\ell \}_{\ell=0}^1$ are given below.

\begin{lemma}
\label{lemma:exact-eval-plus-standard}
{\color{blue} Suppose Assumption~\ref{assumption:uniform-Lipschitz} holds for function $g$.}
If $\mathtt{estimator}(x_i^k, i; \theta)$ is defined by (\ref{eq:exact-eval-plus-standard}), where $\theta= (a,b)$ and $\Theta= \mathbb{N}_+^2$, then Assumption \ref{assumption:general-estimation-error-bounds} holds with the following 
\begin{align*}
    &\mathcal{C}(K, \boldsymbol{\tau}, \Delta) = \left\{ (a,b)\in \mathbb{N}_+^2 : a\ge \frac{4}{9} \log \left( \frac{2(m+ 1) \Sigma_\tau}{\Delta} \right), b\ge \frac{4}{9} \log \left( \frac{2(m+ n) \Sigma_\tau}{\Delta} \right) \right\} ,\\
    &\gamma_0 = \gamma_2= \lambda_0 =0 , \quad {\color{blue} \gamma_1(K, \boldsymbol{\tau}, \theta, \Delta ) = \frac{4 \widehat{l}_g}{\sqrt{a}} \sqrt{\log \left( \frac{2(m+ 1) \Sigma_\tau}{\Delta} \right)} },\\
    &{\color{blue} \lambda_1(K, \boldsymbol{\tau}, \theta, \Delta ) = \frac{4 \widehat{L}_g}{\sqrt{b}} \sqrt{\log \left( \frac{2(m+ n) \Sigma_\tau}{\Delta} \right)} }.
\end{align*}
\end{lemma}

Just as done before, Lemma \ref{lemma:exact-eval-plus-standard} and Theorem \ref{thm:unified-thm-short} provide recommendations for $K, \boldsymbol{\tau}, \theta$ and enable analysis of resulting oracle complexities.

\begin{corollary}[Algorithmic guarantee]
\label{cor:exact-eval-plus-standard}
Consider any $\Delta \in (0,1)$, $M>5 l_f L_g$, integer $\tau>0$, and any sufficiently small $\epsilon >0$. Suppose Assumption \ref{assumption:subroutine} holds for $\mathtt{solver}$, {\color{blue} Assumption~\ref{assumption:uniform-Lipschitz} holds for function $g$,} $\inf_x \Phi(x) > -\infty$, $\mathtt{estimator}$ is defined by (\ref{eq:exact-eval-plus-standard}), and $\boldsymbol{\tau}$ is restricted to the form $\tau_0= \cdots= \tau_{K-1} = \tau$. Set $K= \lceil \frac{C_{\Sigma} \cdot \epsilon^{-1}}{ \tau} \rceil$, {\color{blue} $a= \lceil C_a \cdot \tau^2 \cdot \epsilon^{-1} \cdot \log ( \frac{4(m+ 1) K \tau}{\Delta} ) \rceil$, $b= \lceil C_b \cdot \tau^2 \cdot \log ( \frac{4(m+ n) K \tau}{\Delta} ) \rceil$}, $\overline{\delta} = \Delta/ (2K\tau)$, $\overline{\epsilon} = \epsilon/(5\cdot 30M)$, where $C_{\Sigma}, C_a, C_b$ are some constants, 
then with probability at least $1-\Delta$: (i) Algorithm $\ref{algo:general-framework}$'s iterates satisfy:
\begin{equation*}
\frac{1}{K \tau} \sum_{k=0}^{K-1} \sum_{i=0}^{\tau -1} \| \mathcal{G}_M(x_i^k) \|_2 ^2 \le \epsilon,
\end{equation*}
and (ii) the oracle complexities for evaluations and Jacobians of inner components $g_{\xi}(\cdot)$ respectively are at most 
\begin{align}
\label{eq: cor-exact-eval-plus-standard,eq-1}
&\widetilde{\Theta} \left( N+ \epsilon^{-1} \tau^3 + N\epsilon^{-1} \tau^{-1} + \epsilon^{-2} \tau^2 \right)\\
\label{eq: cor-exact-eval-plus-standard,eq-2}
\text{and}\quad &\widetilde{\Theta} \left( N+\tau^3+ N\epsilon^{-1} \tau^{-1} + \epsilon^{-1} \tau^2 \right) .
\end{align}
\end{corollary}

\begin{corollary}[Optimized complexity bounds]
\label{cor:exact-eval-plus-standard-bounds}
(i) The minimal asymptotic evaluation complexity, with respect to $\tau$, of (\ref{eq: cor-exact-eval-plus-standard,eq-1}) is $\widetilde{\Theta} \left( N+ \epsilon^{-2}+ N^{2/3} \epsilon^{-4/3} \right)$, achieved by setting $\tau = \Theta \left( \max\{ 1, N^{1/3} \epsilon^{1/3} \} \right)$. In this case, (\ref{eq: cor-exact-eval-plus-standard,eq-2}) will become $\widetilde{\Theta} \left( \min \{ N\epsilon^{-1}, N+ N^{2/3} \epsilon^{-4/3} \} \right)$.
(ii) The minimal asymptotic Jacobian complexity, with respect to $\tau$, of (\ref{eq: cor-exact-eval-plus-standard,eq-2}) is $\widetilde{\Theta} \left( N + N^{2/3} \epsilon^{-1} \right)$, achieved by setting $\tau = \Theta \left( N^{1/3} \right)$. In this case, (\ref{eq: cor-exact-eval-plus-standard,eq-1}) becomes $\widetilde{\Theta} \left( N\epsilon^{-1} + N^{2/3} \epsilon^{-2} \right)$.
\end{corollary}

\noindent {\bf Application of Modified Variance Reduction for Finite Averages.}
Similarly, we can also incorporate exact evaluation into $\mathtt{estimator}_2$. The modified estimator $\mathtt{estimator}_4 (x_i^k, i; \theta)$ is defined in two cases, $i=0$ and $i>0$, as
\begin{equation}
\label{eq:exact-eval-plus-smart}
\tag{Est\textsubscript{4}}
\begin{cases}
\widetilde{g}_0^k = \frac{1}{N} \sum_{j=1}^N g_j(x_0^k) = g(x_0^k) \\
\widetilde{J}_0^k = \frac{1}{N} \sum_{j=1}^N g'_j(x_0^k) = g'(x_0^k) \\
\widetilde{g}_i^k = \frac{1}{a} \sum_{j\in \mathcal{A}_i^k} \Bigl( g_j(x_i^k) - g_j(x_0^k) - g'_j(x_0^k) (x_i^k-x_0^k) \Bigr) + \widetilde{g}_0^k + \widetilde{J}_0^k (x_i^k-x_0^k) \\
\widetilde{J}_i^k = \frac{1}{b} \sum_{j\in \mathcal{B}_i^k} \Bigl( g'_j(x_i^k) - g'_j(x_0^k) \Bigr) + \widetilde{J}_0^k .
\end{cases}
\end{equation}
This estimator is the natural generalization of~\eqref{eq:sample-mean-plus-smart} to utilize exact computations at the start of each epoch.
For $\mathtt{estimator}_4$, we have the following sequence of results, including choices of $\mathcal{C}(K, \boldsymbol{\tau}, \Delta)$, $\{\gamma_\ell \}_{\ell=0}^2$, $\{\lambda_\ell \}_{\ell=0}^1$, and algorithmic analysis.

\begin{lemma}
\label{lemma:exact-eval-plus-smart}
{\color{blue} Suppose Assumption~\ref{assumption:uniform-Lipschitz} holds.} If $\mathtt{estimator}(x_i^k, i; \theta)$ is defined by (\ref{eq:exact-eval-plus-smart}), where $\theta= (a,b)$ and $\Theta= \mathbb{N}_+^2$, then Assumption \ref{assumption:general-estimation-error-bounds} holds with the following 
\begin{align*}
    &\mathcal{C}(K, \boldsymbol{\tau}, \Delta) = \left\{ (a,b)\in \mathbb{N}_+^2 : a\ge \frac{4}{9} \log \left( \frac{2(m+ 1) \Sigma_\tau}{\Delta} \right), b\ge \frac{4}{9} \log \left( \frac{2(m+ n) \Sigma_\tau}{\Delta} \right) \right\} ,\\
    &\gamma_0 = \gamma_1= \lambda_0 =0 ,
    \quad {\color{blue} \gamma_2(K, \boldsymbol{\tau}, \theta, \Delta ) = \frac{2 \widehat{L}_g}{\sqrt{a}} \sqrt{\log \left( \frac{2(m+ 1) \Sigma_\tau}{\Delta} \right)} },\\
    &{\color{blue} \lambda_1(K, \boldsymbol{\tau}, \theta, \Delta ) = \frac{4 \widehat{L}_g}{\sqrt{b}} \sqrt{\log \left( \frac{2(m+ n) \Sigma_\tau}{\Delta} \right)} }.
\end{align*}
\end{lemma}

\begin{corollary}[Algorithmic guarantee]
\label{cor:exact-eval-plus-smart}
Consider any $\Delta \in (0,1)$, $M>5 l_f L_g$, integer $\tau>0$, and any sufficiently small $\epsilon >0$. Suppose Assumption \ref{assumption:subroutine} holds for $\mathtt{solver}$, {\color{blue} Assumption~\ref{assumption:uniform-Lipschitz} holds for function $g$,} $\inf_x \Phi(x) > -\infty$, $\mathtt{estimator}$ is defined by (\ref{eq:exact-eval-plus-smart}), and $\boldsymbol{\tau}$ is restricted to the form $\tau_0= \cdots= \tau_{K-1} = \tau$. Set $K= \lceil \frac{C_{\Sigma} \cdot \epsilon^{-1}}{ \tau} \rceil$, {\color{blue} $a= \lceil C_a \cdot \tau^4 \cdot \log ( \frac{4(m+ 1) K \tau}{\Delta} ) \rceil$, $b= \lceil C_b \cdot \tau^2 \cdot \log ( \frac{4(m+ n) K \tau}{\Delta} ) \rceil$}, $\overline{\delta} = \Delta/ (2K\tau)$, $\overline{\epsilon} = \epsilon/(5\cdot 30M)$, where $C_{\Sigma}, C_a, C_b$ are some constants, 
then with probability at least $1-\Delta$: (i) Algorithm $\ref{algo:general-framework}$'s iterates satisfy:
\begin{equation*}
\frac{1}{K \tau} \sum_{k=0}^{K-1} \sum_{i=0}^{\tau -1} \| \mathcal{G}_M(x_i^k) \|_2 ^2 \le \epsilon,
\end{equation*}
and (ii) the oracle complexities for evaluations and Jacobians of inner components $g_{\xi}(\cdot)$ respectively are at most 
\begin{align}
\label{eq: cor-exact-eval-plus-smart,eq-1}
&\widetilde{\Theta} \left( N+\tau^5 + N\epsilon^{-1} \tau^{-1} + \epsilon^{-1} \tau^4 \right)\\
\label{eq: cor-exact-eval-plus-smart,eq-2}
\text{and}\quad &\widetilde{\Theta} \left( N+\tau^3+ N\epsilon^{-1} \tau^{-1} + \epsilon^{-1} \tau^2 \right) .
\end{align}
\end{corollary}
\begin{corollary}[Optimized complexity bounds]
\label{cor:exact-eval-plus-smart-bounds}
(i) The minimal asymptotic evaluation complexity, with respect to $\tau$, of (\ref{eq: cor-exact-eval-plus-smart,eq-1}) is $\widetilde{\Theta} \left( N + N^{4/5} \epsilon^{-1} \right)$, achieved by setting $\tau = \Theta \left(N^{1/5} \right)$. In this case, (\ref{eq: cor-exact-eval-plus-smart,eq-2}) is also $\widetilde{\Theta} \left( N + N^{4/5} \epsilon^{-1} \right)$. 
(ii) The minimal asymptotic Jacobian complexity, with respect to $\tau$, of (\ref{eq: cor-exact-eval-plus-smart,eq-2}) is $\widetilde{\Theta} \left( N + N^{2/3} \epsilon^{-1} \right)$, achieved by setting $\tau = \Theta \left( N^{1/3} \right)$. In this case, (\ref{eq: cor-exact-eval-plus-smart,eq-1}) becomes $\widetilde{\Theta} \left( N^{5/3} + N^{4/3} \epsilon^{-1} \right)$.
\end{corollary}
\begin{remark} \label{remark:comparison-finite-average-methods}
Similar to Remark \ref{remark:comparison-expectation-methods}, we can compare the asymptotic rates in Corollary \ref{cor:exact-eval-plus-standard-bounds} and Corollary \ref{cor:exact-eval-plus-smart-bounds}. The Jacobian complexity parts are the same in these two Corollaries, which is not a surprise, because (\ref{eq: cor-exact-eval-plus-standard,eq-2}) and (\ref{eq: cor-exact-eval-plus-smart,eq-2}) have the same form. We focus on comparing the oracle complexity for evaluations in the two Corollaries. The optimized asymptotic evaluation complexity bound for $\mathtt{estimator}_3$ and $\mathtt{estimator}_4$ are $\widetilde{\Theta} ( N + N^{2/3} \epsilon^{-4/3} + \epsilon^{-2})$ and $\widetilde{\Theta} ( N + N^{4/5} \epsilon^{-1})$ respectively.

There are two parameters $N$ and $\epsilon$ here. Suppose $N = \Theta(\epsilon^{-p})$. Then the previous two asymptotic rates become $\widetilde{\Theta} ( \epsilon^{- p_3})$ and $\widetilde{\Theta} ( \epsilon^{- p_4})$, where $p_3 = \max \{p, \frac{2}{3}p + \frac{4}{3}, 2\}$, $p_4 = \max \{p, \frac{4}{5}p+ 1\}$. Note that
\begin{equation*}
\begin{cases}
p_3 > p_4, \text{ if } p<\frac{5}{2} \\
p_3 < p_4, \text{ if } \frac{5}{2} < p < 5 \\
p_3 = p_4, \text{ if } p=\frac{5}{2} \text{ or } p\ge 5 .
\end{cases}
\end{equation*}
So the optimized asymptotic evaluation complexity bound of $\mathtt{estimator}_3$ is strictly lower if $\frac{5}{2} < p < 5$, and $\mathtt{estimator}_4$ has the lower one if $p<\frac{5}{2}$. This dichotomy suggests neither method’s guarantee uniformly dominates the other. The best method varies depending on the relative rate between $N$ and the accuracy $\epsilon$.
\end{remark}

\subsubsection{Methods with Randomized Epoch Durations}
As a last application, we showcase an example application of Theorem~\ref{thm:unified-thm-short} with random epoch durations $\tau_k$, as were considered by prior variance-reduced works like~\cite{Johnson2013,Konecny2017}.
Consider the following general scheme to determine $K$ and $\boldsymbol{\tau}$ given some $S_{\tau}$: Sample $\tau_0, \tau_1, ...$ independently from some distribution $D_{\tau}$ belonging to a parametric distribution family $ \{ D_{\tau}(\cdot; \tau_+, \theta_{\tau} ): (\tau_+, \theta_{\tau}) \in \mathbb{N}_+ \times \Theta_{\tau} \}$. Then generate $K$ and $\boldsymbol{\tau}$ as
\begin{equation}
\label{eq:randomized-tau}
\begin{cases}
K &\gets \inf\{ N: \sum_{k=0}^{N-1} \tau_k \ge S_{\tau} \}\\
\boldsymbol{\tau} &\gets (\tau_1, ..., \tau_K) . 
\end{cases}
\end{equation}
In \eqref{eq:randomized-tau}, if $S_\tau$ is much larger than each $\tau_k$, then $\sum_{k=0}^{K-1} \tau_k$ will be approximately equal to $S_\tau$. To make this relationship rigorous, we assume the following pair of conditions on the parametric family of generating distribution where the integer parameter $\tau_+$ provides a bound on the size of each $\tau_k$ and control via $C_\tau$ of its expected value.
\begin{assumption}
\label{assumption:varying-tau}
(i) The support of $D_{\tau}(\cdot; \tau_+, \theta_{\tau} )$ is a subset of $\{1,..., \tau_+\}$ for any $(\tau_+, \theta_{\tau}) \in \mathbb{N}_+ \times \Theta_{\tau}$. (ii) There exist a constant $C_{\tau}$, such that $C_{\tau} \mathbb{E}_{\tau \sim D_{\tau}(\cdot; \tau_+, \theta_{\tau} )} [\tau] \ge \tau_+$ for any $(\tau_+, \theta_{\tau}) \in \mathbb{N}_+ \times \Theta_{\tau}$.
\end{assumption}
The intuition behind Assumption \ref{assumption:varying-tau}(ii) is trying to connect this scheme of varying $\tau_k$ with our previous theory of fixed $\tau_k$. Consider a degenerated distribution $\widetilde{D}_{\tau_+}$ where $\tau \equiv \tau_+$, then Assumption \ref{assumption:varying-tau}(ii) controls the expectation ratio between $\widetilde{D}_{\tau_+}$ and $D_{\tau}(\cdot; \tau_+, \theta_{\tau} )$ by a constant upper bound $C_{\tau}$. This intuitively suggests that if we replace $D_{\tau}(\cdot; \tau_+, \theta_{\tau} )$ by $\widetilde{D}_{\tau_+}$ in \eqref{eq:randomized-tau} while keeping $S_\tau$ unchanged, the returned $K$ will increase at most by some constant factor. Indeed, we can prove this holds with high probability, which leads to the following result.

\begin{corollary}[Algorithmic guarantee for the scheme of varying $\tau$]
\label{cor:sample-mean-plus-standard-varying-tau}
Consider any $\Delta \in (0,1)$, $M>5 l_f L_g$ and any sufficiently small $\epsilon >0$. 
Suppose $\mathtt{estimator}$ is defined by (\ref{eq:sample-mean-plus-standard}), $(K, \boldsymbol{\tau})$ is generated by \eqref{eq:randomized-tau}, Assumption \ref{assumption:subroutine} holds for $\mathtt{solver}$, {\color{blue} Assumptions~\ref{assumption:variance-like},~\ref{assumption:uniform-Lipschitz} hold} for function $g$, Assumption \ref{assumption:varying-tau} holds for the generating distribution, and $\inf_x \Phi(x) > -\infty$. 
{\color{blue} For some constants $C_{\Sigma}, C_A, C_B, C_a, C_b, C_p$,} set $\tau_+ = \lceil \epsilon^{-1/3} \rceil$, $S_{\tau} = \lceil C_{\Sigma} \cdot \epsilon^{-1} \rceil$, $A = \lceil C_A \cdot \epsilon^{-2} \cdot \log ( \frac{5(m+ 1) S_\tau }{\Delta} ) \rceil$, $B = \lceil C_B \cdot \epsilon^{-1} \cdot \log ( \frac{5(m+ n) S_\tau }{\Delta} ) \rceil$, {\color{blue} $a = \lceil C_a \cdot \tau_+^2 \cdot \epsilon^{-1} \cdot \log ( \frac{5(m+ 1) S_\tau }{\Delta} ) \rceil$, $b = \lceil C_b \cdot \tau_+^2 \cdot \log ( \frac{5(m+ n) S_\tau }{\Delta} ) \rceil$}, $\overline{\delta} = \frac{\Delta }{2( S_{\tau}+ \tau_+) }$, $\overline{\epsilon} = \epsilon/(5\cdot 30M)$, 
then with probability at least $1-\Delta - \exp( -C_p \epsilon^{-2/3} )$: 
(i) Algorithm~\ref{algo:general-framework}'s iterates satisfy:
\begin{equation*}
\frac{1}{ \Sigma_\tau } \sum_{k=0}^{K-1} \sum_{i=0}^{\tau_k -1} \| \mathcal{G}_M(x_i^k) \|_2 ^2 \le \epsilon,
\end{equation*}
and (ii) the oracle complexity for evaluations and Jacobians of inner components $g_{\xi}(\cdot)$ are at most $\widetilde{\Theta} (\epsilon^{-8/3} \log(1/\Delta) )$ and $\widetilde{\Theta} (\epsilon^{-5/3} \log(1/\Delta) )$ respectively.
\end{corollary}

Note the two complexities here, $\widetilde{\Theta} (\epsilon^{-8/3} \log(1/\Delta))$ and $\widetilde{\Theta} (\epsilon^{-5/3} \log(1/\Delta))$, match the bounds in Corollary \ref{cor:sample-mean-plus-standard-bounds}. Note the probability bound of $1-\Delta - \exp(-C_p \epsilon^{-2/3})$ slightly differs from the $1-\Delta$ in Corollary \ref{cor:sample-mean-plus-standard}. So Corollary \ref{cor:sample-mean-plus-standard-varying-tau} recovers the oracle complexities of fixed epoch duration setting, despite an exponentially small setback in probability guarantee.
{\color{blue} The complexity recovery is mainly because of the constant factor setting, as we remarked after Assumption~\ref{assumption:varying-tau}. We leave the open question of whether randomized (or more interestingly, adaptive) schemes for epoch duration can strictly improve algorithmic guarantees.}
We used $\mathtt{estimator}_1$ as an example to illustrate the idea of randomizing epoch duration. Similar scheme can also be applied to all of the other estimators discussed in previous sections.

\subsection{On the Computational Costs of Solver Subroutines} \label{subsec:inexact-accounting}
To provide a complete accounting for the computational cost of a variance reduced method, one ought to additionally consider the cost of (inexactly) computing proximal steps, i.e., evaluating $\mathtt{solver}$. A prox-linear step is required at every iteration of Algorithm~\ref{algo:general-framework}. Hence by Theorem~\ref{thm:unified-thm-short}, $\Sigma_\tau = O(1/\epsilon)$ (inexact) 
{\color{blue} solver calls} are needed. 

For example, if $f$ is sufficiently simple, one may be able to exactly minimize $s_i^k$, setting $\mathtt{solver}(s,\bar\epsilon,\bar\delta) = \argmin s_i^k$. For example, the subproblem for nonlinear regression problems with $f(z) = \|z\|^2_2$ is least squares minimization, which can be solved exactly as a linear system. Alternatively, if $f(z)=\max_{j=1\dots m} z_j$ {\color{blue} and $h$ is quadratic}, then 
{\color{blue} $s_i^k$} is a quadratic program of dimension {\color{blue} $n$}. 
Hence, the total cost of Algorithm~\ref{algo:general-framework}'s proximal solves is $O(1/\epsilon)$ (inexact) linear system or quadratic program solves, respectively.
As a second example, if $f$ has uniformly $L_{f}$-Lipschitz gradient, a linearly convergent (accelerated) gradient method can be applied to each strongly convex proximal subproblem $s_i^k$. The resulting total number of gradient oracle calls  to $f$ is then
$ O\left(\frac{1}{\epsilon}\log(1/\epsilon)\right)$.

As a more interesting example, consider a doubly stochastic composite problem
$$ \min_{x} \mathbb{E}_{\zeta} f_{\zeta}(\mathbb{E}_{\xi} g_\xi(x)) + h(x) . $$
Given only samples of $\zeta$ and $\xi$, one cannot directly construct unbiased estimators of subgradients of $\mathbb{E}_{\zeta} f_{\zeta}(\mathbb{E}_{\xi} g_\xi(x))$, preventing the application of many direct stochastic first-order methods, see~\cite{Liu2022,Wang2017}.
Regardless, if each $f_\zeta$ is uniformly $l_f$-Lipschitz, a stochastic proximal subgradient method can be applied to minimize the subproblem $s_i^k$. After $O\left(1/\bar \epsilon\right)$ steps, an $\bar\epsilon$-minimizer can be guaranteed~\cite{grimmer2024primaldual}. Hence the total number of subgradient oracle calls needed to $f$ at most
$ O\left(1/\epsilon^2\right). $ Noting we measure stationarity by the gradient norm {\it squared}, this agrees with the subgradient method's nonsmooth, nonconvex $O(1/\epsilon^4)$ rate~\cite{Davis2018StochasticMM} when unbiased subgradients are available. 

\section{Analysis} \label{Sect:Analysis}
Recall that the objective function is $\Phi(x) = f(g(x)) + h(x)$. The standard prox-linear method may consider a linearized proximal subproblem with the following objective function at each iteration:
$$l_i^k(x):= f\Bigl( g(x_i^k) + g'(x_i^k) (x- x_i^k)\Bigr) + h(x) + \frac{M}{2} \|x- x_i^k\|_2^2 .$$
In our algorithm, we replace $g(x_i^k)$ and $g'(x_i^k)$ with stochastic estimates $\widetilde{g}_i^k$ and $\widetilde{J}_i^k$, resulting in the following stochastic linearized objective function at each iteration:
$$s_i^k(x):= f\Bigl( \widetilde{g}_i^k + \widetilde{J}_i^k (x- x_i^k)\Bigr) + h(x) + \frac{M}{2} \|x- x_i^k\|_2^2 .$$
Since $l_i^k(x)$ and $s_i^k(x)$ are $M$-strongly convex, they have unique minimizers, denoted 
$$\widehat{x}_{i+1}^k := \arg\min l_i^k(x) ,
\qquad \text{and}\qquad \widetilde{x}_{i+1}^k := \arg\min s_i^k(x) .$$
Noting $l_i^k(x)$ is the objective function of prox-linear step in (\ref{eq:prox-linear}), by the definition in (\ref{eq:GeneralizedGradient}), our measure of stationarity at the iterate $x_i^k$ is $\|\mathcal{G}_M (x_i^k)\|_2 = M \|x_i^k - \widehat{x}_{i+1}^k \|_2$.

\subsection{Proofs for our Main Unified Convergence Theorem}
In this part, we will prove a sequence of lemmas leading to the unified theory in Theorem \ref{thm:unified-thm}. Our main result, Theorem \ref{thm:unified-thm-short}, is a consequence of Theorem \ref{thm:unified-thm}. Here we give an overview of our analysis.

We first prove three lemmas only depending on the basic setting of prox-linear methods, without the specific assumptions for $\mathtt{estimator}$ or $\mathtt{solver}$. Lemma \ref{lemma:approximation-step-in-f} is a useful result that upper bounds the prox-linear error. The upper bound involves the estimation error terms $\|\widetilde{g}_i^k - g(x_i^k)\|_2$ and $\|\widetilde{J}_i^k - g'(x_i^k) \|_{\rm{op}}$. Lemma \ref{lemma:distance-involves-x-hat} provides a one-step property for $x_i^k$, $\widehat{x}_{i+1}^k$ and $\widetilde{x}_{i+1}^k$. In particular, it upper bounds the distance $\|\widehat{x}_{i+1}^k - x_i^k\|_2$ by $\|\widetilde{x}_{i+1}^k - x_i^k\|_2$ and the estimation error terms. Lemma \ref{lemma:descent-property-for-x-tilde} provides a descent property for the objective function $\Phi$, though only between $\Phi( \widetilde{x}_{i+1}^k )$ and $\Phi(x_i^k)$.

To apply this inductively, we require a descent between $\Phi( x_{i+1}^k )$ and $\Phi(x_i^k)$. Assumption \ref{assumption:subroutine} for $\mathtt{solver}$ enables us to relate $x_{i+1}^k$ to $\widetilde{x}_{i+1}^k$ with high probability. Lemma \ref{lemma:descent-property-for-x} uses this to give such a descent property. Lemma \ref{lemma:one-step-analysis-with-error} ultimate combines our results to give an upper bound for $\|\widehat{x}_{i+1}^k - x_i^k\|_2$, which is proportional to $\| \mathcal{G}_M (x_i^k) \|_2$. 
Assumption \ref{assumption:general-estimation-error-bounds} for $\mathtt{estimator}$ then allows us to uniformly bound error terms with high probability, formalized in Lemma \ref{lemma:general-lemma-one-step}. Applying a careful induction with the upper bound in Lemma \ref{lemma:general-lemma-one-step} to cancel accumulated terms $\|x_i^k - x_0^k\|_2$ and $\|x_{i+1}^k - x_i^k\|_2$ suffices to give our ultimate result in Theorem \ref{thm:unified-thm}, an upper bound for $\frac{1}{\Sigma_\tau} \sum_{k=0}^{K-1} \sum_{i=0}^{\tau_k -1} \| \mathcal{G}_M(x_i^k) \|_2 ^2$.

\begin{lemma}
\label{lemma:approximation-step-in-f}
For any $(k,i)\in \mathcal{I}(K, \boldsymbol{\tau} )$, the following holds for any $x$:
\begin{equation*}
\begin{split}
& \left| f\Bigl(\widetilde{g}_i^k + \widetilde{J}_i^k (x- x_i^k) \Bigr) - f\Bigl(g(x_i^k)+ g'(x_i^k)(x- x_i^k) \Bigr) \right| \\
&\le l_f \|\widetilde{g}_i^k - g(x_i^k)\|_2 + \frac{l_f}{2 L_g} \|\widetilde{J}_i^k - g'(x_i^k) \|_{\rm{op}}^2 + \frac{l_f L_g}{2} \|x- x_i^k\|_2^2.
\end{split}
\end{equation*}
\end{lemma}
\begin{proof}[Proof of Lemma \ref{lemma:approximation-step-in-f}]
Applying in order the Lipschitz continuity of $f$, triangle inequality, operator norm definition, and bounding $a\cdot b$ by $\frac{1}{2 L_g} a^2 + \frac{L_g}{2} b^2$ yields
\begin{equation*}
\begin{split}
& \left| f\Bigl(\widetilde{g}_i^k + \widetilde{J}_i^k (x- x_i^k) \Bigr) - f\Bigl(g(x_i^k)+ g'(x_i^k)(x- x_i^k) \Bigr) \right| \\
&\le l_f \| \widetilde{g}_i^k + \widetilde{J}_i^k (x- x_i^k) - g(x_i^k) - g'(x_i^k)(x- x_i^k) \|_2 \\
&\le l_f \|\widetilde{g}_i^k - g(x_i^k)\|_2 + l_f \| \widetilde{J}_i^k (x- x_i^k) - g'(x_i^k)(x- x_i^k) \|_2 \\
&\le l_f \|\widetilde{g}_i^k - g(x_i^k)\|_2 + l_f \|\widetilde{J}_i^k - g'(x_i^k) \|_{\rm{op}} \cdot \|x- x_i^k\|_2 \\
&\le l_f \|\widetilde{g}_i^k - g(x_i^k)\|_2 + \frac{l_f}{2 L_g} \|\widetilde{J}_i^k - g'(x_i^k) \|_{\rm{op}}^2 + \frac{l_f L_g}{2} \|x- x_i^k\|_2^2 . 
\end{split}
\end{equation*}
\end{proof}

\begin{lemma}
\label{lemma:distance-involves-x-hat}
For any $(k,i)\in \mathcal{I}(K, \boldsymbol{\tau} )$,
\begin{equation*}
\left(\frac{M}{2} - \frac{l_f L_g}{2} \right) \|\widehat{x}_{i+1}^k - x_i^k\|_2^2 \le 2 l_f \|\widetilde{g}_i^k - g(x_i^k)\|_2 + \frac{l_f}{L_g} \|\widetilde{J}_i^k - g'(x_i^k) \|_{\rm{op}}^2 + \left(M+ \frac{l_f L_g}{2}\right) \|\widetilde{x}_{i+1}^k- x_i^k\|_2^2 .
\end{equation*}
\end{lemma}
\begin{proof}[Proof of Lemma \ref{lemma:distance-involves-x-hat}]
Recall that $\widehat{x}_{i+1}^k$ and $\widetilde{x}_{i+1}^k$ are the minimizers of the $M$-strongly convex functions $l_i^k(x)$ and $s_i^k(x)$ respectively. So $l_i^k(\widetilde{x}_{i+1}^k) \ge l_i^k(\widehat{x}_{i+1}^k) + \frac{M}{2} \|\widetilde{x}_{i+1}^k - \widehat{x}_{i+1}^k\|_2^2$ and $s_i^k(\widehat{x}_{i+1}^k) \ge s_i^k(\widetilde{x}_{i+1}^k) + \frac{M}{2} \|\widehat{x}_{i+1}^k - \widetilde{x}_{i+1}^k\|_2^2$, i.e.
\begin{equation*}
\begin{split}
& f\Bigl( g(x_i^k) + g'(x_i^k) (\widetilde{x}_{i+1}^k- x_i^k)\Bigr) + h(\widetilde{x}_{i+1}^k) + \frac{M}{2} \|\widetilde{x}_{i+1}^k- x_i^k\|_2^2 \\
&\ge f\Bigl( g(x_i^k) + g'(x_i^k) (\widehat{x}_{i+1}^k- x_i^k)\Bigr) + h(\widehat{x}_{i+1}^k) + \frac{M}{2} \|\widehat{x}_{i+1}^k- x_i^k\|_2^2 + \frac{M}{2} \|\widetilde{x}_{i+1}^k - \widehat{x}_{i+1}^k\|_2^2
\end{split}
\end{equation*}
and
\begin{equation*}
\begin{split}
& f\Bigl( \widetilde{g}_i^k + \widetilde{J}_i^k (\widehat{x}_{i+1}^k- x_i^k)\Bigr) + h(\widehat{x}_{i+1}^k) + \frac{M}{2} \|\widehat{x}_{i+1}^k- x_i^k\|_2^2 \\
&\ge f\Bigl( \widetilde{g}_i^k + \widetilde{J}_i^k (\widetilde{x}_{i+1}^k- x_i^k)\Bigr) + h(\widetilde{x}_{i+1}^k) + \frac{M}{2} \|\widetilde{x}_{i+1}^k- x_i^k\|_2^2 + \frac{M}{2} \|\widehat{x}_{i+1}^k - \widetilde{x}_{i+1}^k\|_2^2.
\end{split}
\end{equation*}
Summing the two inequalities above, we get
\begin{equation}
\label{eq: proof-lemma-distance-involves-x-hat,eq-1}
\begin{split}
M \|\widehat{x}_{i+1}^k - \widetilde{x}_{i+1}^k\|_2^2 \le &\ f(g(x_i^k) + g'(x_i^k) (\widetilde{x}_{i+1}^k- x_i^k)) - f(\widetilde{g}_i^k + \widetilde{J}_i^k (\widetilde{x}_{i+1}^k- x_i^k)) \\
&+ f(\widetilde{g}_i^k + \widetilde{J}_i^k (\widehat{x}_{i+1}^k- x_i^k)) - f(g(x_i^k) + g'(x_i^k) (\widehat{x}_{i+1}^k- x_i^k)).
\end{split}
\end{equation}
Let $x = \widetilde{x}_{i+1}^k$ and $\widehat{x}_{i+1}^k$ in Lemma \ref{lemma:approximation-step-in-f} respectively, we have
\begin{equation}
\label{eq: proof-lemma-distance-involves-x-hat,eq-2}
\begin{split}
&f\Bigl( g(x_i^k) + g'(x_i^k) (\widetilde{x}_{i+1}^k- x_i^k)\Bigr) - f\Bigl( \widetilde{g}_i^k + \widetilde{J}_i^k (\widetilde{x}_{i+1}^k- x_i^k)\Bigr) \\
&\le l_f \|\widetilde{g}_i^k - g(x_i^k)\|_2 + \frac{l_f}{2 L_g} \|\widetilde{J}_i^k - g'(x_i^k) \|_{\rm{op}}^2 + \frac{l_f L_g}{2} \|\widetilde{x}_{i+1}^k- x_i^k\|_2^2 ,
\end{split}
\end{equation}
and
\begin{equation}
\label{eq: proof-lemma-distance-involves-x-hat,eq-3}
\begin{split}
&f\Bigl( \widetilde{g}_i^k + \widetilde{J}_i^k (\widehat{x}_{i+1}^k- x_i^k)\Bigr) - f\Bigl( g(x_i^k) + g'(x_i^k) (\widehat{x}_{i+1}^k- x_i^k)\Bigr) \\
&\le l_f \|\widetilde{g}_i^k - g(x_i^k)\|_2 + \frac{l_f}{2 L_g} \|\widetilde{J}_i^k - g'(x_i^k) \|_{\rm{op}}^2 + \frac{l_f L_g}{2} \|\widehat{x}_{i+1}^k- x_i^k\|_2^2 .
\end{split}
\end{equation}
Combining (\ref{eq: proof-lemma-distance-involves-x-hat,eq-1}), (\ref{eq: proof-lemma-distance-involves-x-hat,eq-2}), and (\ref{eq: proof-lemma-distance-involves-x-hat,eq-3}) yields
\begin{equation*}
M \|\widehat{x}_{i+1}^k - \widetilde{x}_{i+1}^k\|_2^2 \le 2 l_f \|\widetilde{g}_i^k - g(x_i^k)\|_2 + \frac{l_f}{L_g} \|\widetilde{J}_i^k - g'(x_i^k) \|_{\rm{op}}^2 + \frac{l_f L_g}{2} \|\widetilde{x}_{i+1}^k- x_i^k\|_2^2 + \frac{l_f L_g}{2} \|\widehat{x}_{i+1}^k- x_i^k\|_2^2.
\end{equation*}
Then noting that $\|\widehat{x}_{i+1}^k - \widetilde{x}_{i+1}^k\|_2^2 \ge -\|\widetilde{x}_{i+1}^k - x_i^k\|_2^2 + \frac{1}{2}\|\widehat{x}_{i+1}^k - x_i^k\|_2^2$ gives the claim.
\end{proof}

\begin{lemma}
\label{lemma:descent-property-for-x-tilde}
For any $(k,i)\in \mathcal{I}(K, \boldsymbol{\tau} )$,
\begin{equation*}
\Phi(\widetilde{x}_{i+1}^k) \le \Phi(x_i^k) - (M - l_f L_g) \|\widetilde{x}_{i+1}^k- x_i^k\|_2^2 + 2 l_f \|\widetilde{g}_i^k- g(x_i^k)\|_2 + \frac{l_f}{2 L_g} \|\widetilde{J}_i^k - g'(x_i^k) \|_{\rm{op}}^2.
\end{equation*}
\end{lemma}
\begin{proof}[Proof of Lemma \ref{lemma:descent-property-for-x-tilde}]
By strong convexity, $s_i^k( x_i^k ) \ge s_i^k( \widetilde{x}_{i+1}^k ) + \frac{M}{2} \| \widetilde{x}_{i+1}^k - x_i^k \|_2^2$, i.e.
\begin{equation}
\label{eq: proof-lemma-descent-property-for-x-tilde,eq-1}
s_i^k( x_i^k ) \ge f(\widetilde{g}_i^k + \widetilde{J}_i^k (\widetilde{x}_{i+1}^k- x_i^k)) + h(\widetilde{x}_{i+1}^k) + M \|\widetilde{x}_{i+1}^k- x_i^k\|_2^2.
\end{equation}
By Lipschitz continuity of $f$,
\begin{equation}
\label{eq: proof-lemma-descent-property-for-x-tilde,eq-2}
s_i^k(x_i^k) - \Phi(x_i^k) = f(\widetilde{g}_i^k) + h(x_i^k) - f(g(x_i^k)) - h(x_i^k) \le l_f \|\widetilde{g}_i^k- g(x_i^k)\|_2.
\end{equation}
From (\ref{eq: proof-lemma-descent-property-for-x-tilde,eq-1}) and (\ref{eq: proof-lemma-descent-property-for-x-tilde,eq-2}), we have
\begin{equation}
\label{eq: proof-lemma-descent-property-for-x-tilde,eq-3}
\Phi(x_i^k) + l_f \|\widetilde{g}_i^k- g(x_i^k)\|_2 \ge f(\widetilde{g}_i^k + \widetilde{J}_i^k (\widetilde{x}_{i+1}^k- x_i^k)) + h(\widetilde{x}_{i+1}^k) + M \|\widetilde{x}_{i+1}^k- x_i^k\|_2^2.
\end{equation}
Let $x=\widetilde{x}_{i+1}^k$ and $y=x_i^k$ in Proposition \ref{prop:linear-approx-in-f},
\begin{equation}
\label{eq: proof-lemma-descent-property-for-x-tilde,eq-4}
\Phi(\widetilde{x}_{i+1}^k)- h(\widetilde{x}_{i+1}^k) = f(g(\widetilde{x}_{i+1}^k)) \le  f\Bigl(g(x_i^k)+ g'(x_i^k) (\widetilde{x}_{i+1}^k- x_i^k) \Bigr) + \frac{l_f L_g}{2} \|\widetilde{x}_{i+1}^k- x_i^k\|_2^2.
\end{equation}
Combining (\ref{eq: proof-lemma-descent-property-for-x-tilde,eq-3}) and (\ref{eq: proof-lemma-descent-property-for-x-tilde,eq-4}), we have
\begin{equation}
\label{eq: proof-lemma-descent-property-for-x-tilde,eq-5}
\begin{split}
\Phi(\widetilde{x}_{i+1}^k) \le & \ h(\widetilde{x}_{i+1}^k)+ f\Bigl(g(x_i^k)+ g'(x_i^k)(\widetilde{x}_{i+1}^k- x_i^k) \Bigr) + \frac{l_f L_g}{2} \|\widetilde{x}_{i+1}^k- x_i^k\|_2^2 \\
\le & \ \Phi(x_i^k) + l_f \|\widetilde{g}_i^k- g(x_i^k)\|_2 + \left(\frac{l_f L_g}{2}- M \right) \|\widetilde{x}_{i+1}^k- x_i^k\|_2^2 \\
&  + f\Bigl(g(x_i^k)+ g'(x_i^k)(\widetilde{x}_{i+1}^k- x_i^k) \Bigr) - f\Bigl(\widetilde{g}_i^k + \widetilde{J}_i^k (\widetilde{x}_{i+1}^k- x_i^k) \Bigr) .
\end{split}
\end{equation}
By Lemma \ref{lemma:approximation-step-in-f}, 
\begin{equation}
\label{eq: proof-lemma-descent-property-for-x-tilde,eq-6}
\begin{split}
& f\Bigl(g(x_i^k)+ g'(x_i^k)(\widetilde{x}_{i+1}^k- x_i^k) \Bigr) - f\Bigl(\widetilde{g}_i^k + \widetilde{J}_i^k (\widetilde{x}_{i+1}^k- x_i^k) \Bigr) \\
&\le l_f \| g(x_i^k)- \widetilde{g}_i^k \|_2 +\frac{l_f}{2 L_g} \|g'(x_i^k) - \widetilde{J}_i^k \|_{\rm{op}}^2 + \frac{l_f L_g}{2} \|\widetilde{x}_{i+1}^k- x_i^k\|_2^2.
\end{split}
\end{equation}
Finally, combining (\ref{eq: proof-lemma-descent-property-for-x-tilde,eq-5}) and (\ref{eq: proof-lemma-descent-property-for-x-tilde,eq-6}) gives the claim.
\end{proof}

\begin{lemma}
\label{lemma:descent-property-for-x}
Suppose Assumption \ref{assumption:subroutine} holds for $\mathtt{solver}$, then for an arbitrary $(k,i)\in \mathcal{I}(K, \boldsymbol{\tau} )$, the following holds with probability at least $1- \overline{\delta}$:
\begin{equation*}
\begin{split}
\Phi(x_{i+1}^k) - \Phi(x_i^k) \le &\ \overline{\epsilon} - \left( \frac{M}{2} - l_f L_g\right) \|x_{i+1}^k- x_i^k\|_2^2 - \left(\frac{M}{2} - 2 l_f L_g\right) \|\widetilde{x}_{i+1}^k- x_i^k\|_2^2 \\
& + 4 l_f \|\widetilde{g}_i^k- g(x_i^k)\|_2 + \frac{3 l_f}{2 L_g} \| \widetilde{J}_i^k - g'(x_i^k) \|_{\rm{op}}^2 .
\end{split}
\end{equation*}
With probability at least $1- \overline{\delta} \Sigma_\tau$, the inequality above holds for all $(k,i)\in \mathcal{I}(K, \boldsymbol{\tau} )$.
\end{lemma}
\begin{proof}[Proof of Lemma \ref{lemma:descent-property-for-x}]
Fix an arbitrary $(k,i)\in \mathcal{I}(K, \boldsymbol{\tau} )$. We can split $\Phi(x_{i+1}^k) - \Phi(x_i^k)$ into the sum of three parts:
\begin{equation}
\label{eq: proof-lemma-descent-property-for-x,eq-1}
\begin{split}
\Phi(x_{i+1}^k) - \Phi(x_i^k) 
&= \left[ s_i^k(x_{i+1}^k) - s_i^k(\widetilde{x}_{i+1}^k) \right] + \left[ (\Phi-s_i^k)(x_{i+1}^k) - (\Phi-s_i^k)(\widetilde{x}_{i+1}^k) \right]\\
&\quad + \left[ \Phi(\widetilde{x}_{i+1}^k) - \Phi(x_i^k) \right].
\end{split}
\end{equation}
By Lemma \ref{lemma:descent-property-for-x-tilde},
\begin{equation}
\label{eq: proof-lemma-descent-property-for-x,eq-2}
\Phi(\widetilde{x}_{i+1}^k) - \Phi(x_i^k) \le (l_f L_g- M) \|\widetilde{x}_{i+1}^k- x_i^k\|_2^2 + 2 l_f \|\widetilde{g}_i^k- g(x_i^k)\|_2 + \frac{l_f}{2 L_g} \|\widetilde{J}_i^k - g'(x_i^k) \|_{\rm{op}}^2.
\end{equation}
{\color{blue} We define $\mathcal{E}_{k, i}$ as the subset of the whole probability space, on which
\begin{equation}
\label{eq: proof-lemma-descent-property-for-x,eq-3}
s_i^k(x_{i+1}^k) - s_i^k (\widetilde{x}_{i+1}^k) \le \overline{\epsilon}.
\end{equation}
By Assumption \ref{assumption:subroutine}, $\mathbb{P}( \mathcal{E}_{k, i}) \ge 1- \overline{\delta}$.}

Then it remains to deal with $(\Phi-s_i^k)(x_{i+1}^k) - (\Phi-s_i^k)(\widetilde{x}_{i+1}^k)$.
Note that $(\Phi-s_i^k)(x) = f(g(x)) - f(\widetilde{g}_i^k + \widetilde{J}_i^k (x- x_i^k)) - \frac{M}{2} \|x- x_i^k\|_2^2$, so
\begin{equation}
\label{eq: proof-lemma-descent-property-for-x,eq-4}
\begin{split}
& (\Phi-s_i^k)(x_{i+1}^k) - (\Phi-s_i^k)(\widetilde{x}_{i+1}^k) \\
=& \ f(g(x_{i+1}^k)) - f\Bigl( \widetilde{g}_i^k + \widetilde{J}_i^k (x_{i+1}^k- x_i^k)\Bigr) - \frac{M}{2} \|x_{i+1}^k- x_i^k\|_2^2 \\
&- f(g(\widetilde{x}_{i+1}^k)) + f\Bigl( \widetilde{g}_i^k + \widetilde{J}_i^k (\widetilde{x}_{i+1}^k- x_i^k)\Bigr) + \frac{M}{2} \|\widetilde{x}_{i+1}^k- x_i^k\|_2^2 \\
\le& \ \left| f(g(x_{i+1}^k)) - f\Bigl( \widetilde{g}_i^k + \widetilde{J}_i^k (x_{i+1}^k- x_i^k)\Bigr)\right| - \frac{M}{2} \|x_{i+1}^k- x_i^k\|_2^2 \\
&+ {\color{blue} \left| f(g(\widetilde{x}_{i+1}^k)) - f\Bigl( \widetilde{g}_i^k + \widetilde{J}_i^k (\widetilde{x}_{i+1}^k- x_i^k)\Bigr)\right| } + \frac{M}{2} \|\widetilde{x}_{i+1}^k- x_i^k\|_2^2 .
\end{split}
\end{equation}
Let $y= x_i^k$ in Proposition \ref{prop:linear-approx-in-f}, and combine it with Lemma \ref{lemma:approximation-step-in-f}, we get the following inequality for any $x$:
\begin{equation}
\label{eq: proof-lemma-descent-property-for-x,eq-5}
\begin{split}
|f(g(x)) - f(\widetilde{g}_i^k + \widetilde{J}_i^k (x- x_i^k))| 
&\le \left| f(g(x)) - f\Bigl(g(x_i^k)+ g'(x_i^k) (x-x_i^k) \Bigr) \right|\\
& \quad + \left| f\Bigl(g(x_i^k)+ g'(x_i^k)(x- x_i^k) \Bigr) - f\Bigl(\widetilde{g}_i^k + \widetilde{J}_i^k (x- x_i^k) \Bigr) \right| \\
&\le l_f \|\widetilde{g}_i^k - g(x_i^k)\|_2 + \frac{l_f}{2 L_g} \|\widetilde{J}_i^k - g'(x_i^k) \|_{\rm{op}}^2 + l_f L_g \|x- x_i^k\|_2^2 . 
\end{split}
\end{equation}
Let $x=x_{i+1}^k$ and $\widetilde{x}_{i+1}^k$ in (\ref{eq: proof-lemma-descent-property-for-x,eq-5}) respectively, and plug into (\ref{eq: proof-lemma-descent-property-for-x,eq-4}),
\begin{equation}
\label{eq: proof-lemma-descent-property-for-x,eq-6}
\begin{split}
&(\Phi-s_i^k)(x_{i+1}^k) - (\Phi-s_i^k)(\widetilde{x}_{i+1}^k) \\
&\le 2 l_f \|\widetilde{g}_i^k - g(x_i^k)\|_2 + \frac{l_f}{L_g} \|\widetilde{J}_i^k - g'(x_i^k) \|_{\rm{op}}^2 + \left(l_f L_g -\frac{M}{2} \right) \|x_{i+1}^k- x_i^k\|_2^2 + \left(l_f L_g +\frac{M}{2} \right) \|\widetilde{x}_{i+1}^k- x_i^k\|_2^2 .
\end{split}
\end{equation}
Finally, on the set $\mathcal{E}_{k, i}$, we can use (\ref{eq: proof-lemma-descent-property-for-x,eq-2}), (\ref{eq: proof-lemma-descent-property-for-x,eq-3}), and (\ref{eq: proof-lemma-descent-property-for-x,eq-6}) to upper bound the three parts on the right side of (\ref{eq: proof-lemma-descent-property-for-x,eq-1}) :
\begin{equation*}
\begin{split}
&\Phi(x_{i+1}^k) - \Phi(x_i^k) \\
&= s_i^k(x_{i+1}^k) - s_i^k(\widetilde{x}_{i+1}^k) + (\Phi-s_i^k)(x_{i+1}^k) - (\Phi-s_i^k)(\widetilde{x}_{i+1}^k) + \Phi(\widetilde{x}_{i+1}^k) - \Phi(x_i^k) \\
&\leq \overline{\epsilon} + \left(l_f L_g- \frac{M}{2}\right) \|x_{i+1}^k- x_i^k\|_2^2 + \left(2 l_f L_g- \frac{M}{2}\right) \|\widetilde{x}_{i+1}^k- x_i^k\|_2^2 \\
& + 4 l_f \|\widetilde{g}_i^k- g(x_i^k)\|_2 + \frac{3 l_f}{2 L_g} \| \widetilde{J}_i^k - g'(x_i^k) \|_{\rm{op}}^2 .
\end{split}
\end{equation*}
The inequality above holds for all $(k,i)\in \mathcal{I}(K, \boldsymbol{\tau} )$ on the set $\cap_{(k,i)\in \mathcal{I}(K, \boldsymbol{\tau} )} \mathcal{E}_{k, i}$, which has probability at least $1- \overline{\delta} \Sigma_\tau$ by a simple union bound.
\end{proof}

\begin{lemma}
\label{lemma:one-step-analysis-with-error}
Suppose Assumption \ref{assumption:subroutine} holds for $\mathtt{solver}$. If $M>5 l_f L_g$, then with probability at least $1- \overline{\delta} \Sigma_\tau$, the following holds for all $(k,i)\in \mathcal{I}(K, \boldsymbol{\tau} )$:
\begin{equation*}
\begin{split}
\frac{2M}{5} \|\widehat{x}_{i+1}^k - x_i^k\|_2^2 \le &\ 12 \overline{\epsilon} + 12\left( \Phi(x_i^k) - \Phi(x_{i+1}^k) \right)+ 50 l_f \|\widetilde{g}_i^k- g(x_i^k)\|_2\\
&  + 19\frac{l_f}{L_g} \| \widetilde{J}_i^k - g'(x_i^k) \|_{\rm{op}}^2  -12 \left(\frac{M}{2} - l_f L_g\right) \|x_{i+1}^k- x_i^k\|_2^2 .
\end{split}
\end{equation*}
\end{lemma}
\begin{proof}[Proof of Lemma \ref{lemma:one-step-analysis-with-error}]
$M>5l_f L_g$ implies $\frac{M}{2} - \frac{l_f L_g}{2} > \frac{2M}{5}$. Then from Lemma \ref{lemma:distance-involves-x-hat}, we have
\begin{equation}
\label{eq: proof-lemma-one-step-analysis-with-error,eq-1}
\frac{2M}{5} \|\widehat{x}_{i+1}^k - x_i^k\|_2^2 \le 2 l_f \|\widetilde{g}_i^k - g(x_i^k)\|_2 + \frac{l_f}{L_g} \|\widetilde{J}_i^k - g'(x_i^k) \|_{\rm{op}}^2 + \left(M+ \frac{l_f L_g}{2}\right) \|\widetilde{x}_{i+1}^k- x_i^k\|_2^2 .
\end{equation}
By Lemma \ref{lemma:descent-property-for-x}, there exists a subset $\mathcal{E}$ of the whole probability space, such that $\mathbb{P}( \mathcal{E} ) \ge 1- \overline{\delta} \Sigma_\tau$, and the following inequality holds for all $(k,i)\in \mathcal{I}(K, \boldsymbol{\tau} )$ on $\mathcal{E}$:
\begin{equation}
\label{eq: proof-lemma-one-step-analysis-with-error,eq-2}
\begin{split}
& \left(\frac{M}{2} - l_f L_g\right) \|x_{i+1}^k- x_i^k\|_2^2 + \left(\frac{M}{2} - 2 l_f L_g\right) \|\widetilde{x}_{i+1}^k- x_i^k\|_2^2 \\
&\le \overline{\epsilon} + \Phi(x_i^k) - \Phi(x_{i+1}^k) + 4 l_f \|\widetilde{g}_i^k- g(x_i^k)\|_2 + \frac{3 l_f}{2 L_g} \| \widetilde{J}_i^k - g'(x_i^k) \|_{\rm{op}}^2 .
\end{split}
\end{equation}
Note that $12(\frac{M}{2} - 2 l_f L_g) > M+ \frac{l_f L_g}{2}$, we can multiply (\ref{eq: proof-lemma-one-step-analysis-with-error,eq-2}) by 12 and combine with (\ref{eq: proof-lemma-one-step-analysis-with-error,eq-1}) to get the claim
for all $(k,i)\in \mathcal{I}(K, \boldsymbol{\tau} )$ on $\mathcal{E}$, with probability at least $1- \overline{\delta} \Sigma_\tau$.
\end{proof}

\begin{lemma}
\label{lemma:general-lemma-one-step}
Suppose Assumption \ref{assumption:general-estimation-error-bounds} holds for $\mathtt{estimator}$, and Assumption \ref{assumption:subroutine} holds for $\mathtt{solver}$. Fix an $M>5 l_f L_g$. Then for any $K\in \mathbb{N}_+$, $\boldsymbol{\tau} \in \mathbb{N}_+^K$, $\Delta\in (0,1)$, $\theta \in \mathcal{C}(K, \boldsymbol{\tau}, \Delta)$, and an arbitrary set of positive reals $\{\alpha_{(k,i)} >0: (k,i)\in \mathcal{I}(K, \boldsymbol{\tau} ) \}$, with probability at least $1- \overline{\delta} \Sigma_\tau - \Delta$, the following holds for all $(k,i)\in \mathcal{I}(K, \boldsymbol{\tau} )$:
\begin{equation*}
\begin{split}
\frac{2M}{5} \|\widehat{x}_{i+1}^k - x_i^k\|_2^2 & \le 12 \overline{\epsilon} + 12\left( \Phi(x_i^k) - \Phi(x_{i+1}^k) \right) -12 \left(\frac{M}{2} - l_f L_g\right) \|x_{i+1}^k- x_i^k\|_2^2 \\
& + \left(50 l_f \gamma_0(K, \boldsymbol{\tau}, \theta, \Delta)+ 25 l_f \alpha_{(k,i)} \gamma_1(K, \boldsymbol{\tau}, \theta, \Delta) + 38\frac{l_f}{L_g} \lambda_0^2(K, \boldsymbol{\tau}, \theta, \Delta) \right) \\
& + \left(50 l_f \gamma_2(K, \boldsymbol{\tau}, \theta, \Delta)+ \frac{25 l_f \gamma_1(K, \boldsymbol{\tau}, \theta, \Delta)}{\alpha_{(k,i)}} + 38\frac{l_f}{L_g} \lambda_1^2(K, \boldsymbol{\tau}, \theta, \Delta) \right) \|x_i^k - x_0^k\|_2^2.
\end{split}
\end{equation*}
\end{lemma}
\begin{proof}[Proof of Lemma \ref{lemma:general-lemma-one-step}]
As a shorthand, we use $\gamma_\ell$, $\lambda_\ell$ for $\gamma_\ell (\cdot, \cdot, \cdot, \cdot)$, $\lambda_\ell (\cdot, \cdot, \cdot, \cdot)$ as the arguments are clear from context. Note that $M>5 l_f L_g$, then by Lemma \ref{lemma:one-step-analysis-with-error}, there exists a subset $\mathcal{E}_1$ of the whole probability space, such that $\mathbb{P} (\mathcal{E}_1) \ge 1- \overline{\delta} \Sigma_\tau$, and the following inequality holds for all $(k,i)\in \mathcal{I}(K, \boldsymbol{\tau} )$ on $\mathcal{E}_1$:
\begin{equation}
\label{eq: proof-lemma-general-lemma-one-step,eq-1}
\begin{split}
\frac{2M}{5} \|\widehat{x}_{i+1}^k - x_i^k\|_2^2 \le & \ 12 \overline{\epsilon} + 12\left( \Phi(x_i^k) - \Phi(x_{i+1}^k) \right) + 50 l_f \|\widetilde{g}_i^k- g(x_i^k)\|_2 \\
&+ 19\frac{l_f}{L_g} \| \widetilde{J}_i^k - g'(x_i^k) \|_{\rm{op}}^2 - 12 \left(\frac{M}{2} - l_f L_g\right) \|x_{i+1}^k- x_i^k\|_2^2 .
\end{split}
\end{equation}
By Assumption \ref{assumption:general-estimation-error-bounds}, there exists a subset $\mathcal{E}_2$ of the whole probability space, such that $\mathbb{P} (\mathcal{E}_2) \ge 1-\Delta$, and the following two inequalities hold for all $(k,i)\in \mathcal{I}(K, \boldsymbol{\tau} )$ on $\mathcal{E}_2$:
\begin{align*}
& \|\widetilde{g}_i^k- g(x_i^k)\|_2 \le \gamma_0 + \gamma_1 \|x_i^k - x_0^k\|_2 + \gamma_2 \|x_i^k - x_0^k\|_2^2 ,\\
& \|\widetilde{J}_i^k- g'(x_i^k) \|_{\rm{op}} \le \lambda_0 + \lambda_1 \|x_i^k - x_0^k\|_2 .
\end{align*}
Then the next two inequalities also hold for all $(k,i)\in \mathcal{I}(K, \boldsymbol{\tau} )$ on $\mathcal{E}_2$:
\begin{align}
\label{eq: proof-lemma-general-lemma-one-step,eq-2}
& \|\widetilde{g}_i^k- g(x_i^k)\|_2 \le (\gamma_0+ \frac{\alpha_{(k,i)} \gamma_1}{2}) + (\gamma_2+ \frac{\gamma_1}{2 \alpha_{(k,i)}}) \|x_i^k - x_0^k\|_2^2 ,\\
\label{eq: proof-lemma-general-lemma-one-step,eq-3}
& \|\widetilde{J}_i^k- g'(x_i^k) \|_{\rm{op}} ^2 \le 2\lambda_0^2 + 2\lambda_1^2 \|x_i^k - x_0^k\|_2^2 ,
\end{align}
where the $\alpha_{(k,i)}$ in (\ref{eq: proof-lemma-general-lemma-one-step,eq-2}) can be arbitrary positive real number. 
Use (\ref{eq: proof-lemma-general-lemma-one-step,eq-2}) and (\ref{eq: proof-lemma-general-lemma-one-step,eq-3}) to upper bound $\|\widetilde{g}_i^k- g(x_i^k)\|_2$ and $\| \widetilde{J}_i^k - g'(x_i^k) \|_{\rm{op}}^2$ in (\ref{eq: proof-lemma-general-lemma-one-step,eq-1}) on the set $\mathcal{E}_1 \cap \mathcal{E}_2$:
\begin{equation*}
\begin{split}
\frac{2M}{5} \|\widehat{x}_{i+1}^k - x_i^k\|_2^2
& \le 12 \overline{\epsilon} + 12\left( \Phi(x_i^k) - \Phi(x_{i+1}^k) \right) -12 \left(\frac{M}{2} - l_f L_g\right) \|x_{i+1}^k- x_i^k\|_2^2 \\
& + (50 l_f \gamma_0+ 25 l_f \alpha_{(k,i)} \gamma_1 + 38\frac{l_f}{L_g} \lambda_0^2) + (50 l_f \gamma_2+ \frac{25 l_f \gamma_1}{\alpha_{(k,i)}} + 38\frac{l_f}{L_g} \lambda_1^2) \|x_i^k - x_0^k\|_2^2 ,
\end{split}
\end{equation*}
for all $(k,i)\in \mathcal{I}(K, \boldsymbol{\tau} )$ on $\mathcal{E}_1 \cap \mathcal{E}_2$, which has probability at least $1- \overline{\delta} \Sigma_\tau - \Delta$.
\end{proof}

\begin{theorem}
\label{thm:unified-thm}
Suppose Assumption \ref{assumption:general-estimation-error-bounds} holds for $\mathtt{estimator}$, and Assumption \ref{assumption:subroutine} holds for $\mathtt{solver}$. Fix an $M>5 l_f L_g$.
If some $K\in \mathbb{N}_+$, $\boldsymbol{\tau} \in \mathbb{N}_+^K$, $\Delta\in (0,1)$ and $\theta \in \mathcal{C}(K, \boldsymbol{\tau}, \Delta)$ satisfy
\begin{equation}
\label{eq: thm-unified-thm,eq-1}
\color{blue} 25 \tau_{\text{max}}^2 \gamma_2(K, \boldsymbol{\tau}, \theta, \Delta) \le 6 L_g \quad\text{and}\quad 19 \tau_{\text{max}}^2 \lambda_1^2(K, \boldsymbol{\tau}, \theta, \Delta) \le 6 L_g^2 
\end{equation}
where $\tau_{\text{max}} = \max\{ \tau_0,..., \tau_{K-1}\}$, 
then the following holds with probability at least $1- \overline{\delta} \Sigma_\tau - \Delta$:
\begin{equation}
\label{eq: thm-unified-thm,eq-2}
\begin{split}
& \frac{1}{\Sigma_\tau} \sum_{k=0}^{K-1} \sum_{i=0}^{\tau_k -1} \| \mathcal{G}_M(x_i^k) \|_2 ^2\\
& \le 30M \overline{\epsilon} + 30M \frac{\Phi(x_0^0) - \Phi(x_0^{K})}{ \Sigma_\tau} + 125M l_f \gamma_0(K, \boldsymbol{\tau}, \theta, \Delta)\\
&\quad + {\color{blue} 135M \frac{l_f}{L_g} \tau_{\text{max}}^2 \gamma_1^2(K, \boldsymbol{\tau}, \theta, \Delta) } + 95M\frac{l_f}{L_g} \lambda_0^2(K, \boldsymbol{\tau}, \theta, \Delta).
\end{split}
\end{equation}
\end{theorem}

{\color{blue}
\begin{proof}[Proof of Theorem \ref{thm:unified-thm}]
As a shorthand, we use $\gamma_\ell$, $\lambda_\ell$ for $\gamma_\ell (\cdot, \cdot, \cdot, \cdot)$, $\lambda_\ell (\cdot, \cdot, \cdot, \cdot)$ as the arguments are clear from context. By letting $\alpha_{(k, i)} = \overline{\alpha}_k := \frac{25 \gamma_1}{12 L_g} \tau_k^2$ in Lemma \ref{lemma:general-lemma-one-step}, there exists a subset $\mathcal{E}$ of the whole probability space, such that $\mathbb{P} (\mathcal{E}) \ge 1- \overline{\delta} \Sigma_\tau - \Delta$, and the inequality below holds for all $(k,i)\in \mathcal{I}(K, \boldsymbol{\tau} )$ on $\mathcal{E}$:
\begin{equation*}
\begin{split}
&\frac{2M}{5} \|\widehat{x}_{i+1}^k - x_i^k\|_2^2 \\
& \le 12 \overline{\epsilon} + 12\left( \Phi(x_i^k) - \Phi(x_{i+1}^k) \right) - 12 (\frac{M}{2} - l_f L_g) \|x_{i+1}^k- x_i^k\|_2^2 \\
& \quad + \left( 50 l_f \gamma_0+ 25 l_f \overline{\alpha}_k \gamma_1 + 38\frac{l_f}{L_g} \lambda_0^2 \right) + \left( 50 l_f \gamma_2+ \frac{12 l_f L_g}{\tau_k^2} + 38\frac{l_f}{L_g} \lambda_1^2 \right) \|x_i^k - x_0^k\|_2^2 \\
&= 12 \overline{\epsilon} + 12\left( \Phi(x_i^k) - \Phi(x_{i+1}^k) \right) - C \|x_{i+1}^k- x_i^k\|_2^2 + a_k + b_k \|x_i^k - x_0^k\|_2^2 ,
\end{split}
\end{equation*}
where we denote $C=12 (\frac{M}{2} - l_f L_g)$, $a_k= 50 l_f \gamma_0+ 25 l_f \overline{\alpha}_k \gamma_1 + 38\frac{l_f}{L_g} \lambda_0^2$, $b_k= 50 l_f \gamma_2+ \frac{12 l_f L_g}{\tau_k^2} + 38\frac{l_f}{L_g} \lambda_1^2$.
Then we can fix an arbitrary $k$, let $i$ range over $0,..., \tau_k-1$ and take the sum. So the following holds for all $k=0,..., K-1$ on $\mathcal{E}$:
\begin{equation}
\label{eq: proof-thm-unified-thm,eq-1}
\begin{split}
\frac{2M}{5} \sum_{i=0}^{\tau_k -1} \|\widehat{x}_{i+1}^k - x_i^k\|_2^2 \le 
&\ 12\tau_k \overline{\epsilon} + 12\left( \Phi(x_0^k) - \Phi(x_{\tau_k}^k) \right) + \tau_k a_k \\
& + b_k \sum_{i=0}^{\tau_k -1} \|x_i^k - x_0^k\|_2^2 - C \sum_{i=0}^{\tau_k -1} \|x_{i+1}^k- x_i^k\|_2^2.
\end{split}
\end{equation}
For any $i=1,...,\tau_k-1$, Cauchy-Schwarz and triangle inequality give
\begin{equation*}
\|x_i^k - x_0^k\|_2^2 \le \Bigl( \sum_{j=0}^{i-1} \| x_{j+1}^k- x_j^k \|_2 \Bigr)^2 \le i \sum_{j=0}^{i-1} \| x_{j+1}^k- x_j^k \|_2^2.
\end{equation*}
If $\tau_k\ge 2$, summing the previous inequality over $i=1,...,\tau_k-1$ gives
\begin{equation*}
\begin{split}
& \sum_{i=0}^{\tau_k -1} \|x_i^k - x_0^k\|_2^2 = \sum_{i=1}^{\tau_k -1} \|x_i^k - x_0^k\|_2^2 \le \sum_{i=1}^{\tau_k -1} \sum_{j=0}^{i-1} i \| x_{j+1}^k- x_j^k \|_2^2 \\
& \le \sum_{i=1}^{\tau_k-1} \sum_{j=0}^{\tau_k-2} i \| x_{j+1}^k- x_j^k \|_2^2 = \frac{(\tau_k-1)\tau_k}{2} \sum_{j=0}^{\tau_k-2} \| x_{j+1}^k- x_j^k \|_2^2 \le \frac{\tau_k^2}{2} \sum_{j=0}^{\tau_k-1} \| x_{j+1}^k- x_j^k \|_2^2.
\end{split}
\end{equation*}
If $\tau_k=1$, the inequality above trivially holds. In either case, we have
\begin{equation*}
\sum_{i=0}^{\tau_k -1} \|x_i^k - x_0^k\|_2^2 \le \frac{\tau_k^2}{2} \sum_{j=0}^{\tau_k-1} \| x_{j+1}^k- x_j^k \|_2^2.
\end{equation*}
Note that $b_k\ge 0$ and $C= 12 (\frac{M}{2} - l_f L_g) \ge 18 l_f L_g \ge b_k \cdot\frac{\tau_k^2}{2}$, we have\footnote{
{\color{blue}
It remains to check $b_k \cdot\frac{\tau_k^2}{2} \le 18 l_f L_g$. This can be seen from (\ref{eq: thm-unified-thm,eq-1}):
\begin{equation*}
\frac{\tau_k^2}{2} b_k = \frac{\tau_k^2}{2} \left(50 l_f \gamma_2+ \frac{12 l_f L_g}{\tau_k^2} + 38\frac{l_f}{L_g} \lambda_1^2 \right) \le 6l_f L_g + 6l_f L_g + 6l_f L_g .
\end{equation*}
}
}
\begin{equation}
\label{eq: proof-thm-unified-thm,eq-2}
b_k \sum_{i=0}^{\tau_k -1} \|x_i^k - x_0^k\|_2^2 \le b_k \cdot\frac{\tau_k^2}{2} \sum_{j=0}^{\tau_k-1} \| x_{j+1}^k- x_j^k \|_2^2 \le C \sum_{i=0}^{\tau_k-1} \| x_{i+1}^k- x_i^k \|_2^2.
\end{equation}
For all $k=0,...,K-1$ on $\mathcal{E}$, combining~\eqref{eq: proof-thm-unified-thm,eq-1} and~\eqref{eq: proof-thm-unified-thm,eq-2} leads to
\begin{equation*}
\begin{split}
\frac{2M}{5} \sum_{i=0}^{\tau_k -1} \|\widehat{x}_{i+1}^k - x_i^k\|_2^2 &\le 12\tau_k \overline{\epsilon} + 12\left( \Phi(x_0^k) - \Phi(x_{\tau_k}^k) \right) + \tau_k a_k \\
&= 12\tau_k \overline{\epsilon} + 12\left( \Phi(x_0^k) - \Phi(x_0^{k+1}) \right) + \tau_k ( 50 l_f \gamma_0+ 25 l_f \overline{\alpha}_k \gamma_1 + 38\frac{l_f}{L_g} \lambda_0^2 ) ,
\end{split}
\end{equation*}
where the last step is because $x_0^{k+1} = x_{\tau_k}^k$ (see Algorithm \ref{algo:general-framework}). Finally, we sum over $k=0,..., K-1$, and use the fact that $\overline{\alpha}_k = \frac{25 \gamma_1}{12 L_g} \tau_k^2 \le \frac{25 \gamma_1}{12 L_g} \tau_{\text{max}}^2$: On $\mathcal{E}$, we have
\begin{equation*}
\begin{split}
\frac{2M}{5} \sum_{k=0}^{K-1} \sum_{i=0}^{\tau_k -1} \|\widehat{x}_{i+1}^k - x_i^k\|_2^2 &\le 12 \Sigma_\tau \cdot \overline{\epsilon} + 12\left( \Phi(x_0^0) - \Phi(x_0^{K}) \right) \\
&\quad + \Sigma_\tau \left(50 l_f \gamma_0+ 54 \frac{l_f}{L_g} \tau_{\text{max}}^2 \gamma_1^2 + 38\frac{l_f}{L_g} \lambda_0^2 \right).
\end{split}
\end{equation*}
Multiplying $5M/ (2 \Sigma_\tau)$ on both sides completes the proof.
\end{proof}
}

\begin{proof}[Proof of Theorem \ref{thm:unified-thm-short}]
{\color{blue} Let $\Delta' = \frac{\Delta}{2}$ and}
replace $\Delta$ by {\color{blue} $\Delta'$} in Theorem \ref{thm:unified-thm}, then Theorem \ref{thm:unified-thm} requires (\ref{eq: thm-unified-thm-short,eq-1}), and (\ref{eq: thm-unified-thm,eq-1}) becomes (\ref{eq: thm-unified-thm-short,eq-8}), (\ref{eq: thm-unified-thm-short,eq-9}). Under conditions (\ref{eq: thm-unified-thm-short,eq-3})--(\ref{eq: thm-unified-thm-short,eq-7}), the $5$ terms on the right hand side of (\ref{eq: thm-unified-thm,eq-2}) are all at most $\epsilon/5$. In addition, suppose (\ref{eq: thm-unified-thm-short,eq-2}) holds, i.e.,  $\overline{\delta} \Sigma_\tau \le \frac{\Delta}{2}$, then the probability bound in Theorem \ref{thm:unified-thm} becomes {\color{blue} $1-\overline{\delta} \Sigma_\tau -\Delta' \ge 1-\Delta$}. Therefore, under conditions (\ref{eq: thm-unified-thm-short,eq-1})--(\ref{eq: thm-unified-thm-short,eq-9}), Theorem \ref{thm:unified-thm} gives $\frac{1}{\Sigma_\tau} \sum_{k=0}^{K-1} \sum_{i=0}^{\tau_k -1} \| \mathcal{G}_M(x_i^k) \|_2 ^2 \le \epsilon$ with probability at least $1-\Delta$.
\end{proof}

\subsection{Sample Derivations of Corollaries \ref{cor:sample-mean-plus-standard} and \ref{cor:sample-mean-plus-standard-bounds}} \label{sec:sample-derivations}
We provide the direct calculations of the claimed guarantees for $\mathtt{estimator}_1$ discussed in Section~\ref{subsec:vr-corollaries}. 
To do this, we first introduce a needed concentration inequality (Section~\ref{subsubsec:concentration-ineq}) and technical bounds related to establishing $\mathtt{estimator}_1$ satisfies Assumption~\ref{assumption:general-estimation-error-bounds} (Section~\ref{subsubsec:technical-lemmas}). From these calculations, Corollaries \ref{cor:sample-mean-plus-standard} and \ref{cor:sample-mean-plus-standard-bounds} both follow (Section~\ref{subsubsec:proof-of-corollaries}). 
The derivations of the remaining corollaries in Section \ref{subsec:vr-corollaries} are in the technical report~\cite{wu2024unifiedtheory}.

\subsubsection{Concentration Inequality} \label{subsubsec:concentration-ineq}

\begin{lemma}[Matrix Bernstein]
\label{lemma:Matrix-Bernstein-corollary}
Let $X_1,...,X_n$ be independent random matrices of common dimension $d_1 \times d_2$. Assume $\mathbb{E}[X_k]=0$ and $\|X_k\|_{\rm{op}} \le L$ for each $k=1,...,n$ where $L$ is some constant. If $n\ge \frac{4}{9} \log \left( 
\frac{d_1+ d_2}{\delta} \right)$ for some $\delta\in (0,1)$, then
\begin{equation*}
\left\| \frac{1}{n} \sum_{k=1}^n X_k \right\|_{\rm{op}} \le \frac{2L}{\sqrt{n}} \sqrt{\log \left( \frac{d_1+ d_2}{\delta} \right)}
\end{equation*}
holds with probability at least $1-\delta$.
\end{lemma}
\begin{proof}[Proof of Lemma \ref{lemma:Matrix-Bernstein-corollary}]
Using a classic Matrix Bernstein bound (see \cite{Tropp2015}), one has 
\begin{equation}\label{eq:MatrixBernstein}
\mathbb{P} \left( \left\| \sum_{k=1}^n X_k \right\|_{\rm{op}} \ge t \right) \le (d_1+d_2) \cdot \exp \left( \frac{-t^2/2} {V+Lt/3} \right),
\end{equation}
where $V := \max \left\{ \left\| \sum_{k=1}^n \mathbb{E} \left[X_k X_k^T\right] \right\|_{\rm{op}}, \left\| \sum_{k=1}^n \mathbb{E} \left[X_k^T X_k\right] \right\|_{\rm{op}} \right\}$. By some basic properties,
\begin{equation*}
\left\| \sum_{k=1}^n \mathbb{E} \Bigl[X_k X_k^T\Bigr] \right\|_{\rm{op}} \le \sum_{k=1}^n \mathbb{E} \left\| X_k X_k^T \right\|_{\rm{op}} \le \sum_{k=1}^n \mathbb{E} \Bigl[ \bigl\| X_k \bigr\|_{\rm{op}} \cdot \bigl\| X_k^T \bigr\|_{\rm{op}} \Bigr] \le n L^2.
\end{equation*}
Similarly, $\left\| \sum_{k=1}^n \mathbb{E} \left[X_k^T X_k\right] \right\|_{\rm{op}} \le n L^2$. So the $V$ defined above is at most $n L^2$. Consequently,
\begin{equation*}
\mathbb{P} \left( \left\| \sum_{k=1}^n X_k \right\|_{\rm{op}} \ge t \right) \le (d_1+d_2) \cdot \exp \left( \frac{-t^2/2} {V+Lt/3} \right) \le (d_1+d_2) \cdot \exp \left( \frac{-t^2/2} {n L^2+ Lt/3} \right)
\end{equation*}
for any $t\ge 0$. 
With $t= 2L \sqrt{n \log \left( \frac{d_1+ d_2}{\delta} \right)}$, note that $\frac{1}{3} Lt = \frac{2 L^2}{3} \sqrt{n\log \left( \frac{d_1+ d_2}{\delta} \right)} \le nL^2$, where the last step is from the assumption that $n\ge \frac{4}{9} \log \left(  \frac{d_1+ d_2}{\delta} \right)$. The claim then follows directly from~\eqref{eq:MatrixBernstein}.
\end{proof}

\subsubsection{Technical Bounds for Estimators} \label{subsubsec:technical-lemmas}

\begin{proposition}
\label{prop:error-bound-i=0-sample-mean}
Suppose Assumption \ref{assumption:variance-like} holds. For an arbitrary fixed pair $(k,i)\in \mathcal{I}(K, \boldsymbol{\tau} )$ and any $\delta\in (0,1)$, assume $\widetilde{g}_i^k$ and $\widetilde{J}_i^k$ are constructed by (\ref{eq:mini-batch}). (i) If $A\ge \frac{4}{9} \log \left( \frac{m+ 1}{\delta} \right)$, then the following holds with probability at least $1-\delta$,
\begin{equation*}
\left\| \widetilde{g}_i^k - g(x_i^k) \right\|_2 \le \frac{2 \sigma_g}{\sqrt{A}} \sqrt{\log \left( \frac{m+ 1}{\delta} \right)} ;
\end{equation*}
(ii) If $B\ge \frac{4}{9}\log \left( \frac{m+ n}{\delta} \right)$, then the following holds with probability at least $1-\delta$,
\begin{equation*}
\left\| \widetilde{J}_i^k - g'(x_i^k) \right\|_{\rm{op}} \le \frac{2 \sigma_{g'} }{\sqrt{B}} \sqrt{\log \left( \frac{m+ n}{\delta} \right)} .
\end{equation*}
\end{proposition}
\begin{proof}[Proof of Proposition \ref{prop:error-bound-i=0-sample-mean}]
We first prove part (i). From the construction of $\widetilde{g}_i^k$ in (\ref{eq:mini-batch}), $\widetilde{g}_i^k - g(x_i^k) = \frac{1}{A} \sum_{\xi\in \mathcal{A}_i^k} \left( g_{\xi}(x_i^k) - g(x_i^k) \right)$. Suppose $\mathcal{A}_i^k = \{\xi_1,..., \xi_A\}$, then $\xi_1,..., \xi_A$ are  independently drawn from distribution $D$. For each $r=1,..., A$, denote $Y_r= g_{\xi_r} (x_i^k) - g(x_i^k)$. So $\widetilde{g}_i^k - g(x_i^k) = \frac{1}{A} \sum_{r=1}^A Y_r$.

Use $\mathbf{x_i^k}$ to denote the sequence of iterates $\{x_0^0, x_1^0, ..., x_i^k\}$ in the rest of this proof. Then if conditioning on $\mathbf{x_i^k}$, all the randomness at the $(k,i)$-th  iteration comes from the sampling of $\mathcal{A}_i^k$. So $Y_1,..., Y_A$ are independent conditioning on $\mathbf{x_i^k}$. Since $g= \mathbb{E}_{\xi\sim D} [g_{\xi}]$, we immediately have $\mathbb{E} [Y_r| \mathbf{x_i^k}] = 0$. By Assumption \ref{assumption:variance-like}, $\sigma_g$ is a constant upper bound of $\|Y_r\|_2$. Note $Y_r$ is an $m\times 1$ matrix, $\|Y_r\|_{\rm{op}} = \|Y_r\|_2$. Then we can apply Lemma \ref{lemma:Matrix-Bernstein-corollary} to $Y_1,..., Y_A$ (conditioning on $\mathbf{x_i^k}$): if $A\ge \frac{4}{9} \log(\frac{m+1}{\delta})$ for some $\delta\in (0,1)$, then 
\begin{equation*}
\mathbb{P} \left( \|\widetilde{g}_i^k - g(x_i^k)\|_2 \ge t \mid \mathbf{x_i^k} \right) \le \delta
\end{equation*}
where $t= \frac{2 \sigma_g}{\sqrt{A}} \sqrt{\log \left( \frac{m+ 1}{\delta} \right)}$. This implies the unconditional probability is also upper bounded by $\delta$,\footnote{
To see this, we can rewrite all the probabilities as the expectations of corresponding indicator functions, then use the law of total expectation. Let $\mathbf{I} = 1$ if $\| \widetilde{g}_i^k - g(x_i^k) \|_2 \ge t$, and $\mathbf{I} = 0$ otherwise. It follows that
\begin{equation*}
\label{eq: proof-prop-error-bound-i=0-sample-mean,eq-1}
\mathbb{P} \left( \left\| \widetilde{g}_i^k - g(x_i^k) \right\|_2 \ge t \right) = \mathbb{E}\left[ \mathbf{I} \right] = \mathbb{E}\left[ \mathbb{E}\left[ \mathbf{I} \mid \mathbf{x_i^k} \right] \right] = \mathbb{E}\left[ \mathbb{P}\left( \left\| \widetilde{g}_i^k - g(x_i^k) \right\|_2 \ge t \mid \mathbf{x_i^k} \right) \right] \le \mathbb{E}\left[ \delta \right] = \delta .
\end{equation*}
}
i.e.,
\begin{equation*}
\mathbb{P} \left( \|\widetilde{g}_i^k - g(x_i^k)\|_2 \ge t \right) \le \delta
\end{equation*}
which finishes the proof for part (i).

The proof for part (ii) is similar. Suppose $\mathcal{B}_0^k = \{\xi'_1,..., \xi'_B\}$, and denote $Z_r= g'_{\xi'_r} (x_i^k) - g'(x_i^k)$ for each $r=1,..., B$. Then $\widetilde{J}_i^k - g'(x_i^k) = \frac{1}{B} \sum_{r=1}^B Z_r$. Applying Lemma \ref{lemma:Matrix-Bernstein-corollary} to $Z_1,..., Z_B$ (conditioning on $\mathbf{x_i^k}$) finishes the proof, since $\|Z_r\|_{\rm{op}} \le \sigma_{g'}$.
\end{proof}

\begin{proposition}
\label{prop:error-bound-g-standard}
{\color{blue} Suppose Assumption~\ref{assumption:uniform-Lipschitz} holds.} For an arbitrary fixed pair $(k,i)\in \mathcal{I}(K, \boldsymbol{\tau} )$ and any $\delta\in (0,1)$, if $\widetilde{g}_i^k$ is constructed by (\ref{eq:sample-mean-plus-standard}) or (\ref{eq:exact-eval-plus-standard}), and $a\ge \frac{4}{9} \log \left( \frac{m+ 1}{\delta} \right)$, then the following holds with probability at least $1-\delta$:
\begin{equation*}
\left\| \widetilde{g}_i^k - g(x_i^k) \right\|_2 \le \left\| \widetilde{g}_0^k - g(x_0^k) \right\|_2 + {\color{blue} \frac{4 \widehat{l}_g}{\sqrt{a}} \sqrt{\log \left( \frac{m+ 1}{\delta} \right)} \|x_i^k -x_0^k\|_2 }.
\end{equation*}
\end{proposition}
\begin{proof}[Proof of Proposition \ref{prop:error-bound-g-standard}]
Noting that $\widetilde{g}_i^k - g(x_i^k) = \left( \widetilde{g}_0^k - g(x_0^k) \right) + \frac{1}{a} \sum_{r=1}^a Y_r$ where $Y_r=g_{\xi_r}(x_i^k) - g_{\xi_r}(x_0^k) - g(x_i^k) + g(x_0^k)$ for each $r=1,...,a$, it suffices to bound $\|\sum_{r=1}^a Y_r\|_2$ and then apply the triangle inequality. The claimed bound then follows directly from Lemma~\ref{lemma:Matrix-Bernstein-corollary} as {\color{blue} $\|Y_r\|_{\rm{op}} \leq 2\widehat{l}_g \|x_i^k - x_0^k\|_2$}.
\end{proof}

\begin{proposition}
\label{prop:error-bound-J}
{\color{blue} Suppose Assumption~\ref{assumption:uniform-Lipschitz} holds.} For an arbitrary fixed pair $(k,i)\in \mathcal{I}(K, \boldsymbol{\tau} )$ and any $\delta\in (0,1)$, if $\widetilde{J}_i^k$ is constructed by any method among (\ref{eq:sample-mean-plus-standard}), (\ref{eq:sample-mean-plus-smart}), (\ref{eq:exact-eval-plus-standard}), (\ref{eq:exact-eval-plus-smart}), and $b\ge \frac{4}{9} \log \left( \frac{m+ n}{\delta} \right)$, then the following holds with probability at least $1-\delta$:
\begin{equation*}
\left\| \widetilde{J}_i^k - g'(x_i^k) \right\|_{\rm{op}} \le \left\| \widetilde{J}_0^k - g'(x_0^k) \right\|_{\rm{op}} + {\color{blue} \frac{4 \widehat{L}_g}{\sqrt{b}} \sqrt{\log \left( \frac{m+ n}{\delta} \right)} \|x_i^k -x_0^k\|_2 }.
\end{equation*}
\end{proposition}
\begin{proof}[Proof of Proposition \ref{prop:error-bound-J}]
Noting that $\widetilde{J}_i^k - g'(x_i^k) = \left( \widetilde{J}_0^k - g'(x_0^k) \right) + \frac{1}{b} \sum_{r=1}^b Z_r$ where $Z_r = g'_{\xi_r}(x_i^k) - g'_{\xi_r}(x_0^k) - g'(x_i^k) + g'(x_0^k)$ for each $r=1,...,b$, it suffices to bound $\|\sum_{r=1}^b Z_r\|_{\rm{op}}$ and then apply the triangle inequality. The claimed bound then follows directly from Lemma~\ref{lemma:Matrix-Bernstein-corollary} as {\color{blue} $\|Z_r\|_{\rm{op}} \leq 2\widehat{L}_g \|x_i^k - x_0^k\|_2$}.
\end{proof}

\subsubsection{Proofs of Corollaries \ref{cor:sample-mean-plus-standard} and \ref{cor:sample-mean-plus-standard-bounds}} \label{subsubsec:proof-of-corollaries}
Using the technical propositions above, we can prove the lemmas and corollaries for $\mathtt{estimator}_1$ claimed in Section~\ref{subsec:vr-corollaries}.


\begin{proof}[Proof of Lemma \ref{lemma:sample-mean-plus-standard}]
For any $K\in \mathbb{N}_+$, $\boldsymbol{\tau} \in \mathbb{N}_+^K$ and $\Delta \in (0,1)$, let $\delta = \frac{\Delta}{2 \Sigma_\tau}$. By Proposition \ref{prop:error-bound-i=0-sample-mean}, for an arbitrary $k\in \{0,..., K-1\}$, any $A\ge \frac{4}{9} \log \left( \frac{m+ 1}{\delta} \right)$ and any $B\ge \frac{4}{9}\log \left( \frac{m+ n}{\delta} \right)$, the following two inequalities hold with probability at least $1- 2\delta$,
\begin{equation}
\label{proof-lemma-sample-mean-plus-standard,eq-1}
\left\| \widetilde{g}_0^k - g(x_0^k) \right\|_2 \le \frac{2 \sigma_g}{\sqrt{A}} \sqrt{\log \left( \frac{m+ 1}{\delta} \right)} , \quad \left\| \widetilde{J}_0^k - g'(x_0^k) \right\|_{\rm{op}} \le \frac{2 \sigma_{g'} }{\sqrt{B}} \sqrt{\log \left( \frac{m+ n}{\delta} \right)} .
\end{equation}
By Proposition \ref{prop:error-bound-g-standard}, for an arbitrary $(k,i)\in \mathcal{I}(K, \boldsymbol{\tau} )$ and any $a\ge \frac{4}{9} \log \left( \frac{m+ 1}{\delta} \right)$, the following holds with probability at least $1-\delta$:
\begin{equation}
\label{proof-lemma-sample-mean-plus-standard,eq-2}
\left\| \widetilde{g}_i^k - g(x_i^k) \right\|_2 \le \left\| \widetilde{g}_0^k - g(x_0^k) \right\|_2 + {\color{blue} \frac{4 \widehat{l}_g}{\sqrt{a}} \sqrt{\log \left( \frac{m+ 1}{\delta} \right)} \|x_i^k -x_0^k\|_2 }.
\end{equation}
By Proposition \ref{prop:error-bound-J}, for an arbitrary $(k,i)\in \mathcal{I}(K, \boldsymbol{\tau} )$ and any $b\ge \frac{4}{9} \log \left( \frac{m+ n}{\delta} \right)$, the following holds with probability at least $1-\delta$:
\begin{equation}
\label{proof-lemma-sample-mean-plus-standard,eq-3}
\left\| \widetilde{J}_i^k - g'(x_i^k) \right\|_{\rm{op}} \le \left\| \widetilde{J}_0^k - g'(x_0^k) \right\|_{\rm{op}} + {\color{blue} \frac{4 \widehat{L}_g}{\sqrt{b}} \sqrt{\log \left( \frac{m+ n}{\delta} \right)} \|x_i^k -x_0^k\|_2 }.
\end{equation}
Let $\mathcal{C}(K, \boldsymbol{\tau}, \Delta) = \{(A,B,a,b) \in \mathbb{N}_+^4 : A,a \ge \frac{4}{9} \log ( \frac{2(m+ 1) \Sigma_\tau}{\Delta} ), \text{and } B,b \ge \frac{4}{9} \log ( \frac{2(m+ n) \Sigma_\tau}{\Delta} )\}$. Then for any $(A,B,a,b) \in \mathcal{C}(K, \boldsymbol{\tau}, \Delta)$, by using a union probability bound, (\ref{proof-lemma-sample-mean-plus-standard,eq-1}), (\ref{proof-lemma-sample-mean-plus-standard,eq-2}) and (\ref{proof-lemma-sample-mean-plus-standard,eq-3}) hold for all $(k,i)\in \mathcal{I}(K, \boldsymbol{\tau} )$ with probability at least $1- 2\Sigma_\tau \delta$. Note that $1- 2\Sigma_\tau \delta = 1- \Delta$, so we can set $\gamma_0(K, \boldsymbol{\tau}, \theta, \Delta ) = \frac{2 \sigma_g}{\sqrt{A}} \sqrt{\log ( \frac{2(m+ 1) \Sigma_\tau}{\Delta} )}$, {\color{blue} $\gamma_1(K, \boldsymbol{\tau}, \theta, \Delta ) = \frac{4 \widehat{l}_g}{\sqrt{a}} \sqrt{\log ( \frac{2(m+ 1) \Sigma_\tau}{\Delta} )}$}, $\gamma_2= 0$, $\lambda_0(K, \boldsymbol{\tau}, \theta, \Delta ) = \frac{2 \sigma_{g'}}{\sqrt{B}} \sqrt{\log ( \frac{2(m+ n) \Sigma_\tau}{\Delta} )}$ and {\color{blue} $\lambda_1(K, \boldsymbol{\tau}, \theta, \Delta ) = \frac{4 \widehat{L}_g}{\sqrt{b}} \sqrt{\log ( \frac{2(m+ n) \Sigma_\tau}{\Delta} )}$} to satisfy Assumption \ref{assumption:general-estimation-error-bounds}. 
\end{proof}

\begin{proof}[Proof of Corollary \ref{cor:sample-mean-plus-standard}] 
We can obtain the explicit form of $\mathcal{C}(K, \boldsymbol{\tau}, \Delta)$, $\{\gamma_\ell \}_{\ell=0}^2$ and $\{\lambda_\ell \}_{\ell=0}^1$ from Lemma \ref{lemma:sample-mean-plus-standard}, and plug them into Theorem \ref{thm:unified-thm-short}. Then by Theorem \ref{thm:unified-thm-short}, to get $\frac{1}{K \tau} \sum_{k=0}^{K-1} \sum_{i=0}^{\tau -1} \| \mathcal{G}_M(x_i^k) \|_2 ^2 \le \epsilon$ with probability at least $1-\Delta$, we need $\overline{\delta} \le \Delta/(2 K \tau )$, $\overline{\epsilon} \le \epsilon/(5\cdot 30M)$, and the following inequalities:
\begin{align*}
\begin{cases}
& A,a \ge \frac{4}{9} \log \left( \frac{4(m+ 1) K \tau}{\Delta} \right), \text{and } B,b \ge \frac{4}{9} \log \left( \frac{4(m+ n) K \tau}{\Delta} \right) \\
& K \tau \ge 5\cdot 30M (\Phi(x_0^0) - \Phi^*) / \epsilon \\
& \frac{2 \sigma_g}{\sqrt{A}} \sqrt{\log \left( \frac{4(m+ 1) K \tau}{\Delta} \right)} \le \epsilon/(5\cdot 125 l_f M) \\
& \frac{4 \sigma_{g'}^2}{B} \log \left( \frac{4(m+ n) K \tau}{\Delta} \right) \le L_g \epsilon/(5\cdot 95 l_f M) \\
& \color{blue} \tau^2 \cdot \frac{16 \widehat{l}_g^2}{a} \log \left( \frac{4(m+ 1) K \tau}{\Delta} \right) \le L_g \epsilon/(5\cdot 135 l_f M) \\
& \color{blue} \tau^2 \cdot \frac{16 \widehat{L}_g^2}{b} \log \left( \frac{4(m+ n) K \tau}{\Delta} \right) \le 6 L_g^2 /19
\end{cases}
\end{align*}
which reduce to 
\begin{equation*}
\begin{cases}
K \tau \ge C_{\Sigma} \cdot \epsilon^{-1} \\
A \ge C_A \cdot \epsilon^{-2} \cdot \log \left( \frac{4(m+ 1) K \tau}{\Delta} \right) \\
B \ge C_B \cdot \epsilon^{-1} \cdot \log \left( \frac{4(m+ n) K \tau}{\Delta} \right) \\
\color{blue} a \ge C_a \cdot \tau^2 \cdot \epsilon^{-1} \cdot \log \left( \frac{4(m+ 1) K \tau}{\Delta} \right) \\
\color{blue} b \ge C_b \cdot \tau^2 \cdot \log \left( \frac{4(m+ n) K \tau}{\Delta} \right) 
\end{cases}
\end{equation*}
providing that $1/\epsilon$ is sufficiently large. Here $C_{\Sigma}, C_A, C_B, C_a, C_b$ are some constants.

For any positive integer $\tau$, let $K = \lceil \frac{C_{\Sigma} \cdot \epsilon^{-1}}{\tau} \rceil$, $A = \lceil C_A \cdot \epsilon^{-2} \cdot \log ( \frac{4(m+ 1) K \tau}{\Delta} ) \rceil$, $B = \lceil C_B \cdot \epsilon^{-1} \cdot \log ( \frac{4(m+ n) K \tau}{\Delta} ) \rceil$, {\color{blue} $a = \lceil C_a \cdot \tau^2 \cdot \epsilon^{-1} \cdot \log ( \frac{4(m+ 1) K \tau}{\Delta} ) \rceil$, $b = \lceil C_b \cdot \tau^2 \cdot \log ( \frac{4(m+ n) K \tau}{\Delta} ) \rceil$}, then the conditions above hold for sufficiently small $\epsilon$. So Theorem \ref{thm:unified-thm-short} guarantees that $\frac{1}{K \tau} \sum_{k=0}^{K-1} \sum_{i=0}^{\tau -1} \| \mathcal{G}_M(x_i^k) \|_2 ^2 \le \epsilon$ with probability at least $1-\Delta$.

In (\ref{eq:sample-mean-plus-standard}), at the $(k,0)$-th iteration, we evaluate $g_{\xi}(\cdot)$ for $A$ times and $g'_{\xi}(\cdot)$ for $B$ times. At the $(k,i)$-th iteration (with $i>0$), we evaluate $g_{\xi}(\cdot)$ for $2a$ times and $g'_{\xi}(\cdot)$ for $2b$ times.
Supposing $\tau = O(\epsilon^{-1})$, the oracle complexity for evaluations of $g_{\xi}(\cdot)$ is 
\begin{equation*}
\begin{split}
KA + 2K(\tau-1)a \le KA + 2K\tau a 
&= \widetilde{\Theta} \left( (\epsilon^{-3} \tau^{-1} + \epsilon^{-2} \tau^2) \log(1/\Delta) \right),
\end{split}
\end{equation*}
and the oracle complexity for evaluations of Jacobians $g'_{\xi}(\cdot)$ is 
\begin{equation*}
\begin{split}
KB + 2K(\tau-1)b \le KB + 2K\tau b 
&= \widetilde{\Theta} \left( (\epsilon^{-2} \tau^{-1} + \epsilon^{-1} \tau^2) \log(1/\Delta) \right).
\end{split}
\end{equation*}
\end{proof}

\begin{proof}[Proof of Corollary \ref{cor:sample-mean-plus-standard-bounds}]
Suppose $\tau = \Theta(\epsilon^{ -\beta})$ for some $\beta \ge 0$. Then (\ref{eq: cor-sample-mean-plus-standard,eq-1}) can be simplified as
$$\widetilde{\Theta} ((\epsilon^{-3} \tau^{-1} + \epsilon^{-2} \tau^2) \log(1/\Delta) ) = \widetilde{\Theta} (\epsilon^{- \max\{3-\beta, 2+2\beta\}} \log(1/\Delta) ) ,$$
and (\ref{eq: cor-sample-mean-plus-standard,eq-2}) can be simplified as
$$\widetilde{\Theta} ((\epsilon^{-2} \tau^{-1} + \epsilon^{-1} \tau^2) \log(1/\Delta) ) = \widetilde{\Theta} (\epsilon^{- \max\{2-\beta, 1+2\beta\}} \log(1/\Delta) ) .$$
The asymptotic rates for two bounds are both minimized by $\beta = \frac{1}{3}$. At $\tau = \Theta(\epsilon^{-1/3})$, the two bounds become $\widetilde{\Theta}(\epsilon^{-8/3} \log(1/\Delta))$ and $\widetilde{\Theta}(\epsilon^{-5/3} \log(1/\Delta))$ respectively.
\end{proof}

    \paragraph{Acknowledgements.} This work was supported in part by the Air Force Office of Scientific Research under award number FA9550-23-1-0531. Benjamin Grimmer was additionally supported as a fellow of the Alfred P. Sloan Foundation.
    
    {\small
    \bibliographystyle{unsrt}
    \bibliography{references}
    }
    \appendix
    \section{Derivations of Remaining Corollaries}

\subsection{More Concentration Inequality and Technical Bounds}
\label{appendix-subsec:technical-lemmas}

\begin{lemma}[Hoeffding's inequality]
\label{lemma:Hoeffding-corollary}
Let $Y_1,...,Y_n$ be independent random variables bounded by $a_i \le Y_i \le b_i$. Then for any $t\ge 0$,  $S_n = \sum_{i=1}^n Y_i$ has
\begin{equation*}
\mathbb{P} \left( S_n \le \mathbb{E}[ S_n ] - t \right) \le \exp \left( -\frac{2 t^2}{\sum_{i=1}^n (b_i-a_i)^2 } \right).
\end{equation*}
\end{lemma}
\begin{proof}[Proof of Lemma \ref{lemma:Hoeffding-corollary}]
This is a restatement of Hoeffding's inequality~\cite{Hoeffding1963}.
\end{proof}

\begin{proposition}
\label{prop:error-bound-g-smart}
{\color{blue} Suppose Assumption~\ref{assumption:uniform-Lipschitz} holds.} For an arbitrary fixed pair $(k,i)\in \mathcal{I}(K, \boldsymbol{\tau} )$ and any $\delta\in (0,1)$, if $\widetilde{g}_i^k$ is constructed by (\ref{eq:sample-mean-plus-smart}) or (\ref{eq:exact-eval-plus-smart}), and $a\ge \frac{4}{9} \log \left( \frac{m+ 1}{\delta} \right)$, then the following holds with probability at least $1-\delta$:
\begin{equation*}
\left\| \widetilde{g}_i^k - g(x_i^k) \right\|_2 \le \left\| \widetilde{g}_0^k - g(x_0^k) \right\|_2 + \left\| \widetilde{J}_0^k - g'(x_0^k) \right\|_{\rm{op}} \|x_i^k -x_0^k\|_2 + {\color{blue} \frac{2 \widehat{L}_g}{\sqrt{a}} \sqrt{\log \left( \frac{m+ 1}{\delta} \right)} \|x_i^k -x_0^k\|_2^2 }.
\end{equation*}
\end{proposition}
\begin{proof}[Proof of Proposition \ref{prop:error-bound-g-smart}]
The proof is similar to the proofs in Section \ref{subsubsec:technical-lemmas}. 
Suppose $\mathcal{A}_i^k = \{\xi_1,..., \xi_a\}$, and denote $Y_r = g_{\xi_r}(x_i^k) - g_{\xi_r}(x_0^k) - g'_{\xi_r}(x_0^k) (x_i^k-x_0^k) - g(x_i^k) + g(x_0^k) + g'(x_0^k) (x_i^k-x_0^k)$ for each $r=1,..., a$. Then we have
\begin{equation}
\label{eq: proof-prop-error-bound-g-smart,eq-1}
\widetilde{g}_i^k - g(x_i^k) = \left( \widetilde{g}_0^k - g(x_0^k) \right) + \left( \widetilde{J}_0^k - g'(x_0^k) \right) (x_i^k-x_0^k) + \frac{1}{a} \sum_{r=1}^a Y_r .
\end{equation}
Use $\mathbf{x_i^k}$ to denote the sequence of iterates $\{x_0^0, x_1^0, ..., x_i^k\}$ in the rest of this proof. Then $Y_1,..., Y_a$ are independent conditioning on $\mathbf{x_i^k}$, and $\mathbb{E}[Y_r| \mathbf{x_i^k}]=0$. 
{\color{blue} By Assumption~\ref{assumption:uniform-Lipschitz}, $g'_{\xi_r} (\cdot)$ and $g'(\cdot)$ are both $\widehat{L}_g$-Lipschitz, which implies $\|g_{\xi_r}(x_i^k) - g_{\xi_r}(x_0^k) - g'_{\xi_r}(x_0^k) (x_i^k-x_0^k) \|_2 \le \frac{\widehat{L}_g}{2} \|x_i^k- x_0^k\|_2^2$ and $\|g(x_i^k) - g(x_0^k) - g'(x_0^k) (x_i^k-x_0^k) \|_2 \le \frac{\widehat{L}_g}{2} \|x_i^k- x_0^k\|_2^2$. So $\widehat{L}_g \|x_i^k- x_0^k\|_2^2$} is an upper bound of $\|Y_r\|_2$. Applying Lemma \ref{lemma:Matrix-Bernstein-corollary} to $Y_1,..., Y_a$ (conditioning on $\mathbf{x_i^k}$) and following a similar proof as Proposition \ref{prop:error-bound-i=0-sample-mean}, we have
\begin{equation*}
\left\| \frac{1}{a} \sum_{r=1}^a Y_r \right\|_2 \le {\color{blue} \frac{2 \widehat{L}_g}{\sqrt{a}} \|x_i^k -x_0^k\|_2^2 \sqrt{\log \left( \frac{m+ 1}{\delta} \right)} }
\end{equation*}
with probability at least $1-\delta$. Combining it with (\ref{eq: proof-prop-error-bound-g-smart,eq-1}) completes the proof.
\end{proof}

\subsection{Proofs of Remaining Corollaries}
\label{appendix-subsec:proof-of-corollaries}
Using the technical Propositions from Section \ref{subsubsec:technical-lemmas} and Appendix \ref{appendix-subsec:technical-lemmas}, we can prove the Lemmas and Corollaries for $\mathtt{estimator}_0$ and $\mathtt{estimator}_2$--$\mathtt{estimator}_4$.

\noindent {\bf Proofs for $\mathtt{estimator}_0$.}

\begin{proof}[Proof of Lemma \ref{lemma:mini-batch}]
For any $K\in \mathbb{N}_+$, $\boldsymbol{\tau} \in \mathbb{N}_+^K$ and $\Delta \in (0,1)$, let $\delta = \frac{\Delta}{2 \Sigma_\tau}$. By Proposition \ref{prop:error-bound-i=0-sample-mean}, for an arbitrary $(k,i)\in \mathcal{I}(K, \boldsymbol{\tau} )$, any $A\ge \frac{4}{9} \log \left( \frac{m+ 1}{\delta} \right)$ and any $B\ge \frac{4}{9}\log \left( \frac{m+ n}{\delta} \right)$, the following two inequalities hold with probability at least $1- 2\delta$,
\begin{equation}
\label{proof-lemma-mini-batch,eq-1}
\left\| \widetilde{g}_i^k - g(x_i^k) \right\|_2 \le \frac{2 \sigma_g}{\sqrt{A}} \sqrt{\log \left( \frac{m+ 1}{\delta} \right)} , \quad \left\| \widetilde{J}_i^k - g'(x_i^k) \right\|_{\rm{op}} \le \frac{2 \sigma_{g'} }{\sqrt{B}} \sqrt{\log \left( \frac{m+ n}{\delta} \right)} .
\end{equation}
Let $\mathcal{C}(K, \boldsymbol{\tau}, \Delta) = \{(A,B) \in \mathbb{N}_+^2 : A \ge \frac{4}{9} \log ( \frac{2(m+ 1) \Sigma_\tau}{\Delta} ), \text{and } B \ge \frac{4}{9} \log ( \frac{2(m+ n) \Sigma_\tau}{\Delta} )\}$. Then for any $(A,B) \in \mathcal{C}(K, \boldsymbol{\tau}, \Delta)$, by using a union probability bound, (\ref{proof-lemma-mini-batch,eq-1}) holds for all $(k,i)\in \mathcal{I}(K, \boldsymbol{\tau} )$ with probability at least $1- 2\Sigma_\tau \delta$. Note that $1- 2\Sigma_\tau \delta = 1- \Delta$, so we can set $\gamma_0(K, \boldsymbol{\tau}, \theta, \Delta ) = \frac{2 \sigma_g}{\sqrt{A}} \sqrt{\log ( \frac{2(m+ 1) \Sigma_\tau}{\Delta} )}$, $\lambda_0(K, \boldsymbol{\tau}, \theta, \Delta ) = \frac{2 \sigma_{g'}}{\sqrt{B}} \sqrt{\log ( \frac{2(m+ n) \Sigma_\tau}{\Delta} )}$ and $\gamma_1= \gamma_2= \lambda_1=0$ to satisfy Assumption \ref{assumption:general-estimation-error-bounds}. 
\end{proof}

\begin{proof}[Proof of Corollary \ref{cor:mini-batch}]
We can obtain the explicit form of $\mathcal{C}(K, \boldsymbol{\tau}, \Delta)$, $\{\gamma_\ell \}_{\ell=0}^2$ and $\{\lambda_\ell \}_{\ell=0}^1$ from Lemma \ref{lemma:mini-batch}, and plug them into Theorem \ref{thm:unified-thm-short}. Then by Theorem \ref{thm:unified-thm-short}, to get $\frac{1}{\Sigma_\tau} \sum_{k=0}^{K-1} \sum_{i=0}^{\tau_k -1} \| \mathcal{G}_M(x_i^k) \|_2 ^2 \le \epsilon$ with probability at least $1-\Delta$, we need $\overline{\delta} \le \Delta/(2 \Sigma_\tau )$, $\overline{\epsilon} \le \epsilon/(5\cdot 30M)$, and the following inequalities:
\begin{align*}
\begin{cases}
& A \ge \frac{4}{9} \log \left( \frac{4(m+ 1) \Sigma_\tau}{\Delta} \right), \text{and } B \ge \frac{4}{9} \log \left( \frac{4(m+ n) \Sigma_\tau}{\Delta} \right) \\
& \Sigma_\tau \ge 5\cdot 30M (\Phi(x_0^0) - \Phi^*) / \epsilon \\
& \frac{2 \sigma_g}{\sqrt{A}} \sqrt{\log \left( \frac{4(m+ 1) \Sigma_\tau}{\Delta} \right)} \le \epsilon/(5\cdot 125 l_f M) \\
& \frac{4 \sigma_{g'}^2}{B} \log \left( \frac{4(m+ n) \Sigma_\tau}{\Delta} \right) \le L_g \epsilon/(5\cdot 95 l_f M) 
\end{cases}
\end{align*}
which reduces to 
\begin{equation*}
\begin{cases}
\Sigma_\tau \ge C_{\Sigma} \cdot \epsilon^{-1} \\
A \ge C_A \cdot \epsilon^{-2} \cdot \log \left( \frac{4(m+ 1) \Sigma_\tau}{\Delta} \right) \\
B \ge C_B \cdot \epsilon^{-1} \cdot \log \left( \frac{4(m+ n) \Sigma_\tau}{\Delta} \right) 
\end{cases}
\end{equation*}
providing that $1/\epsilon$ is sufficiently large. Here $C_{\Sigma}, C_A, C_B$ are some constants.

Let $\Sigma_\tau = \lceil C_{\Sigma} \cdot \epsilon^{-1} \rceil$, $A = \lceil C_A \cdot \epsilon^{-2} \cdot \log ( \frac{4(m+ 1) \Sigma_\tau}{\Delta} ) \rceil$, $B = \lceil C_B \cdot \epsilon^{-1} \cdot \log ( \frac{4(m+ n) \Sigma_\tau}{\Delta} ) \rceil$, then the conditions above hold for sufficiently small $\epsilon$. So Theorem \ref{thm:unified-thm-short} guarantees that $\frac{1}{\Sigma_\tau} \sum_{k=0}^{K-1} \sum_{i=0}^{\tau_k -1} \| \mathcal{G}_M(x_i^k) \|_2 ^2 \le \epsilon$ with probability at least $1-\Delta$.

In (\ref{eq:mini-batch}), at the $(k,i)$-th iteration, we evaluate $g_{\xi}(\cdot)$ for $A$ times and $g'_{\xi}(\cdot)$ for $B$ times. Then the oracle complexity for evaluations of $g_{\xi}(\cdot)$ is 
\begin{equation*}
\Sigma_\tau A = \Theta(\epsilon^{-1}) \cdot \widetilde{\Theta} \left( \epsilon^{-2} \log(1/\Delta) \right) = \widetilde{\Theta} \left( \epsilon^{-3} \log(1/\Delta) \right) ,
\end{equation*}
and the oracle complexity for evaluations of Jacobians $g'_{\xi}(\cdot)$ is 
\begin{equation*}
\Sigma_\tau B = \Theta(\epsilon^{-1}) \cdot \widetilde{\Theta} \left( \epsilon^{-1} \log(1/\Delta) \right) = \widetilde{\Theta} \left( \epsilon^{-2} \log(1/\Delta) \right) . 
\end{equation*}
\end{proof}

\noindent {\bf Proofs for $\mathtt{estimator}_2$.}

\begin{proof}[Proof of Lemma \ref{lemma:sample-mean-plus-smart}]
For any $K\in \mathbb{N}_+$, $\boldsymbol{\tau} \in \mathbb{N}_+^K$ and $\Delta \in (0,1)$, let $\delta = \frac{\Delta}{2 \Sigma_\tau}$. By Proposition \ref{prop:error-bound-i=0-sample-mean}, for an arbitrary $k\in \{0,..., K-1\}$, any $A\ge \frac{4}{9} \log \left( \frac{m+ 1}{\delta} \right)$ and any $B\ge \frac{4}{9}\log \left( \frac{m+ n}{\delta} \right)$, the following two inequalities hold with probability at least $1- 2\delta$,
\begin{equation}
\label{proof-lemma-sample-mean-plus-smart,eq-1}
\left\| \widetilde{g}_0^k - g(x_0^k) \right\|_2 \le \frac{2 \sigma_g}{\sqrt{A}} \sqrt{\log \left( \frac{m+ 1}{\delta} \right)} , \quad \left\| \widetilde{J}_0^k - g'(x_0^k) \right\|_{\rm{op}} \le \frac{2 \sigma_{g'} }{\sqrt{B}} \sqrt{\log \left( \frac{m+ n}{\delta} \right)} .
\end{equation}
By Proposition \ref{prop:error-bound-g-smart}, for an arbitrary $(k,i)\in \mathcal{I}(K, \boldsymbol{\tau} )$ and any $a\ge \frac{4}{9} \log \left( \frac{m+ 1}{\delta} \right)$, the following holds with probability at least $1-\delta$:
\begin{equation}
\label{proof-lemma-sample-mean-plus-smart,eq-2}
\begin{split}
\left\| \widetilde{g}_i^k - g(x_i^k) \right\|_2 \le \left\| \widetilde{g}_0^k - g(x_0^k) \right\|_2 + \left\| \widetilde{J}_0^k - g'(x_0^k) \right\|_{\rm{op}} \|x_i^k -x_0^k\|_2 + {\color{blue} \frac{2 \widehat{L}_g}{\sqrt{a}} \sqrt{\log \left( \frac{m+ 1}{\delta} \right)} \|x_i^k -x_0^k\|_2^2 }.
\end{split}
\end{equation}
By Proposition \ref{prop:error-bound-J}, for an arbitrary $(k,i)\in \mathcal{I}(K, \boldsymbol{\tau} )$ and any $b\ge \frac{4}{9} \log \left( \frac{m+ n}{\delta} \right)$, the following holds with probability at least $1-\delta$:
\begin{equation}
\label{proof-lemma-sample-mean-plus-smart,eq-3}
\left\| \widetilde{J}_i^k - g'(x_i^k) \right\|_{\rm{op}} \le \left\| \widetilde{J}_0^k - g'(x_0^k) \right\|_{\rm{op}} + {\color{blue} \frac{4 \widehat{L}_g}{\sqrt{b}} \sqrt{\log \left( \frac{m+ n}{\delta} \right)} \|x_i^k -x_0^k\|_2 }.
\end{equation}
Let $\mathcal{C}(K, \boldsymbol{\tau}, \Delta) = \{(A,B,a,b) \in \mathbb{N}_+^4 : A,a \ge \frac{4}{9} \log ( \frac{2(m+ 1) \Sigma_\tau}{\Delta} ), \text{and } B,b \ge \frac{4}{9} \log ( \frac{2(m+ n) \Sigma_\tau}{\Delta} )\}$. Then for any $(A,B,a,b) \in \mathcal{C}(K, \boldsymbol{\tau}, \Delta)$, by using a union probability bound, (\ref{proof-lemma-sample-mean-plus-smart,eq-1}), (\ref{proof-lemma-sample-mean-plus-smart,eq-2}) and (\ref{proof-lemma-sample-mean-plus-smart,eq-3}) hold for all $(k,i)\in \mathcal{I}(K, \boldsymbol{\tau} )$ with probability at least $1- 2\Sigma_\tau \delta$. Note that $1- 2\Sigma_\tau \delta = 1- \Delta$, so we can set $\gamma_0(K, \boldsymbol{\tau}, \theta, \Delta ) = \frac{2 \sigma_g}{\sqrt{A}} \sqrt{\log ( \frac{2(m+ 1) \Sigma_\tau}{\Delta} )}$, $\gamma_1(K, \boldsymbol{\tau}, \theta, \Delta ) = \lambda_0(K, \boldsymbol{\tau}, \theta, \Delta ) = \frac{2 \sigma_{g'}}{\sqrt{B}} \sqrt{\log ( \frac{2(m+ n) \Sigma_\tau}{\Delta} )}$, {\color{blue} $\gamma_2(K, \boldsymbol{\tau}, \theta, \Delta ) = \frac{2 \widehat{L}_g}{\sqrt{a}} \sqrt{\log ( \frac{2(m+ 1) \Sigma_\tau}{\Delta} )}$ and $\lambda_1(K, \boldsymbol{\tau}, \theta, \Delta ) = \frac{4 \widehat{L}_g}{\sqrt{b}} \sqrt{\log ( \frac{2(m+ n) \Sigma_\tau}{\Delta} )}$} to satisfy Assumption \ref{assumption:general-estimation-error-bounds}. 
\end{proof}

\begin{proof}[Proof of Corollary \ref{cor:sample-mean-plus-smart}]
We can obtain the explicit form of $\mathcal{C}(K, \boldsymbol{\tau}, \Delta)$, $\{\gamma_\ell \}_{\ell=0}^2$ and $\{\lambda_\ell \}_{\ell=0}^1$ from Lemma \ref{lemma:sample-mean-plus-smart}, and plug them into Theorem \ref{thm:unified-thm-short}. Then by Theorem \ref{thm:unified-thm-short}, to get $\frac{1}{K \tau} \sum_{k=0}^{K-1} \sum_{i=0}^{\tau -1} \| \mathcal{G}_M(x_i^k) \|_2 ^2 \le \epsilon$ with probability at least $1-\Delta$, we need $\overline{\delta} \le \Delta/(2 K \tau )$, $\overline{\epsilon} \le \epsilon/(5\cdot 30M)$, and the following inequalities:
\begin{align*}
\begin{cases}
& A,a \ge \frac{4}{9} \log \left( \frac{4(m+ 1) K\tau}{\Delta} \right), \text{and } B,b \ge \frac{4}{9} \log \left( \frac{4(m+ n) K\tau}{\Delta} \right)  \\
&K\tau \ge 5\cdot 30M (\Phi(x_0^0) - \Phi^*) / \epsilon \\
& \frac{2 \sigma_g}{\sqrt{A}} \sqrt{\log \left( \frac{4(m+ 1) K\tau}{\Delta} \right)} \le \epsilon/(5\cdot 125 l_f M) \\
& \frac{4 \sigma_{g'}^2}{B} \log \left( \frac{4(m+ n) K\tau}{\Delta} \right) \le L_g \epsilon/(5\cdot 95 l_f M) \\
& \color{blue} \tau^2 \cdot \frac{4 \sigma_{g'}^2}{B} \log \left( \frac{4(m+ n) K\tau}{\Delta} \right) \le L_g \epsilon/(5\cdot 135 l_f M) \\
& \color{blue} \tau^2 \cdot \frac{2 \widehat{L}_g}{\sqrt{a}} \sqrt{\log \left( \frac{4(m+ 1) K\tau}{\Delta} \right)} \le 6 L_g /25 \\
& \color{blue} \tau^2 \cdot \frac{16 \widehat{L}_g^2}{b} \log \left( \frac{4(m+ n) K\tau}{\Delta} \right) \le 6 L_g^2 /19
\end{cases}
\end{align*}
So it reduces to 
\begin{equation*}
\begin{cases}
K \tau \ge C_{\Sigma} \cdot \epsilon^{-1} \\
A \ge C_A \cdot \epsilon^{-2} \cdot \log \left( \frac{4(m+ 1) K \tau}{\Delta} \right) \\
\color{blue} B \ge C_B \cdot \tau^2 \cdot \epsilon^{-1} \cdot \log \left( \frac{4(m+ n) K \tau}{\Delta} \right) \\
\color{blue} a \ge C_a \cdot \tau^4 \cdot \log \left( \frac{4(m+ 1) K \tau}{\Delta} \right) \\
\color{blue} b \ge C_b \cdot \tau^2 \cdot \log \left( \frac{4(m+ n) K \tau}{\Delta} \right) 
\end{cases}
\end{equation*}
providing that $1/\epsilon$ is sufficiently large. Here $C_{\Sigma}, C_A, C_B, C_a, C_b$ are some constants.

For any positive integer $\tau$, let $K = \lceil \frac{C_{\Sigma} \cdot \epsilon^{-1}}{\tau} \rceil$, $A = \lceil C_A \cdot \epsilon^{-2} \cdot \log ( \frac{4(m+ 1) K \tau}{\Delta} ) \rceil$, {\color{blue} $B = \lceil C_B \cdot \tau^2 \cdot \epsilon^{-1} \cdot \log ( \frac{4(m+ n) K \tau}{\Delta} ) \rceil$, $a = \lceil C_a \cdot \tau^4 \cdot \log ( \frac{4(m+ 1) K \tau}{\Delta} ) \rceil$, $b = \lceil C_b \cdot \tau^2 \cdot \log ( \frac{4(m+ n) K \tau}{\Delta} ) \rceil$}, then the conditions above hold for sufficiently small $\epsilon$. So Theorem \ref{thm:unified-thm-short} guarantees that $\frac{1}{K \tau} \sum_{k=0}^{K-1} \sum_{i=0}^{\tau -1} \| \mathcal{G}_M(x_i^k) \|_2 ^2 \le \epsilon$ with probability at least $1-\Delta$.

In (\ref{eq:sample-mean-plus-smart}), at the $(k,0)$-th iteration, we evaluate $g_{\xi}(\cdot)$ for $A$ times and $g'_{\xi}(\cdot)$ for $B$ times. At the $(k,i)$-th iteration (with $i>0$), we evaluate $g_{\xi}(\cdot)$ for $2a$ times and $g'_{\xi}(\cdot)$ for $a+2b$ times.
Suppose $\tau = O(\epsilon^{-1})$, then the oracle complexity for evaluations of $g_{\xi}(\cdot)$ is 
\begin{equation*}
\begin{split}
KA + 2K(\tau-1)a \le KA + 2K\tau a 
&= \widetilde{\Theta} \left( (\epsilon^{-3} \tau^{-1} + \epsilon^{-1} \tau^4) \log(1/\Delta) \right),
\end{split}
\end{equation*}
and the oracle complexity for evaluations of Jacobians $g'_{\xi}(\cdot)$ is 
\begin{equation*}
\begin{split}
KB + K(\tau-1) (a+2b)
\le KB + K\tau a + 2K\tau b 
&= \widetilde{\Theta} \left( (\epsilon^{-2} \tau + \epsilon^{-1} \tau^4) \log(1/\Delta) \right). 
\end{split}
\end{equation*}
\end{proof}

\begin{proof}[Proof of Corollary \ref{cor:sample-mean-plus-smart-bounds}]
Suppose $\tau = \Theta(\epsilon^{ -\beta})$ for some $\beta \ge 0$. Then (\ref{eq: cor-sample-mean-plus-smart,eq-1}) can be simplified as
\begin{equation}
\label{eq: proof-cor-sample-mean-plus-smart-bounds,eq-1}
\widetilde{\Theta} \left( (\epsilon^{-3} \tau^{-1} + \epsilon^{-1} \tau^4) \log(1/\Delta) \right) = \widetilde{\Theta} (\epsilon^{- \max\{3-\beta, 1+4\beta\}} \log(1/\Delta) ) ,
\end{equation}
and (\ref{eq: cor-sample-mean-plus-smart,eq-2}) can be simplified as
\begin{equation}
\label{eq: proof-cor-sample-mean-plus-smart-bounds,eq-2}
\widetilde{\Theta} \left( (\epsilon^{-2} \tau + \epsilon^{-1} \tau^4) \log(1/\Delta) \right) = \widetilde{\Theta} (\epsilon^{- \max\{2+\beta, 1+4\beta\}} \log(1/\Delta) ) .
\end{equation}
(\ref{eq: proof-cor-sample-mean-plus-smart-bounds,eq-1}) is minimized by $\beta=\frac{2}{5}$. When $\beta=\frac{2}{5}$,(\ref{eq: proof-cor-sample-mean-plus-smart-bounds,eq-1}) and (\ref{eq: proof-cor-sample-mean-plus-smart-bounds,eq-2}) are $\widetilde{\Theta}(\epsilon^{-13/5} \log(1/\Delta))$. (\ref{eq: proof-cor-sample-mean-plus-smart-bounds,eq-2}) is minimized by $\beta=0$. When $\beta=0$, (\ref{eq: proof-cor-sample-mean-plus-smart-bounds,eq-2}) becomes $\widetilde{\Theta}(\epsilon^{-2} \log(1/\Delta))$ and (\ref{eq: proof-cor-sample-mean-plus-smart-bounds,eq-1}) becomes $\widetilde{\Theta}(\epsilon^{-3} \log(1/\Delta))$.
\end{proof}

\noindent {\bf Proofs for $\mathtt{estimator}_3$.}

\begin{proof}[Proof of Lemma \ref{lemma:exact-eval-plus-standard}]
We have $\widetilde{g}_0^k = g(x_0^k)$ and $\widetilde{J}_0^k = g'(x_0^k)$ from (\ref{eq:exact-eval-plus-standard}). For any $K\in \mathbb{N}_+$, $\boldsymbol{\tau} \in \mathbb{N}_+^K$ and $\Delta \in (0,1)$, let $\delta = \frac{\Delta}{2 \Sigma_\tau}$. By Proposition \ref{prop:error-bound-g-standard}, for an arbitrary $(k,i)\in \mathcal{I}(K, \boldsymbol{\tau} )$ and any $a\ge \frac{4}{9} \log \left( \frac{m+ 1}{\delta} \right)$, the following holds with probability at least $1-\delta$:
\begin{equation}
\label{proof-lemma-exact-eval-plus-standard,eq-1}
\left\| \widetilde{g}_i^k - g(x_i^k) \right\|_2 \le {\color{blue} \frac{4 \widehat{l}_g}{\sqrt{a}} \sqrt{\log \left( \frac{m+ 1}{\delta} \right)} \|x_i^k -x_0^k\|_2 }.
\end{equation}
By Proposition \ref{prop:error-bound-J}, for an arbitrary $(k,i)\in \mathcal{I}(K, \boldsymbol{\tau} )$ and any $b\ge \frac{4}{9} \log \left( \frac{m+ n}{\delta} \right)$, the following holds with probability at least $1-\delta$:
\begin{equation}
\label{proof-lemma-exact-eval-plus-standard,eq-2}
\left\| \widetilde{J}_i^k - g'(x_i^k) \right\|_{\rm{op}} \le {\color{blue} \frac{4 \widehat{L}_g}{\sqrt{b}} \sqrt{\log \left( \frac{m+ n}{\delta} \right)} \|x_i^k -x_0^k\|_2 }.
\end{equation}
Let $\mathcal{C}(K, \boldsymbol{\tau}, \Delta) = \{(a,b)\in \mathbb{N}_+^2 : a\ge \frac{4}{9} \log ( \frac{2(m+ 1) \Sigma_\tau}{\Delta} ), b\ge \frac{4}{9} \log ( \frac{2(m+ n) \Sigma_\tau}{\Delta} )\}$. Then for any $(a,b)\in \mathcal{C}(K, \boldsymbol{\tau}, \Delta)$, by using a union probability bound, (\ref{proof-lemma-exact-eval-plus-standard,eq-1}) and (\ref{proof-lemma-exact-eval-plus-standard,eq-2}) hold for all $(k,i)\in \mathcal{I}(K, \boldsymbol{\tau} )$ with probability at least $1- 2\Sigma_\tau \delta$. Note that $1- 2\Sigma_\tau \delta = 1- \Delta$, so we can set $\gamma_0 = \gamma_2 = \lambda_0= 0$, {\color{blue} $\gamma_1(K, \boldsymbol{\tau}, \theta, \Delta ) = \frac{4 \widehat{l}_g}{\sqrt{a}} \sqrt{\log ( \frac{2(m+ 1) \Sigma_\tau}{\Delta} )}$ and $\lambda_1(K, \boldsymbol{\tau}, \theta, \Delta ) = \frac{4 \widehat{L}_g}{\sqrt{b}} \sqrt{\log ( \frac{2(m+ n) \Sigma_\tau}{\Delta} )}$} to satisfy Assumption \ref{assumption:general-estimation-error-bounds}. 
\end{proof}

\begin{proof}[Proof of Corollary \ref{cor:exact-eval-plus-standard}]
We can obtain the explicit form of $\mathcal{C}(K, \boldsymbol{\tau}, \Delta)$, $\{\gamma_\ell \}_{\ell=0}^2$ and $\{\lambda_\ell \}_{\ell=0}^1$ from Lemma \ref{lemma:exact-eval-plus-standard}, and plug them into Theorem \ref{thm:unified-thm-short}. Then by Theorem \ref{thm:unified-thm-short}, to get $\frac{1}{K \tau} \sum_{k=0}^{K-1} \sum_{i=0}^{\tau -1} \| \mathcal{G}_M(x_i^k) \|_2 ^2 \le \epsilon$ with probability at least $1-\Delta$, we need $\overline{\delta} \le \Delta/(2 K \tau )$, $\overline{\epsilon} \le \epsilon/(5\cdot 30M)$, and the following inequalities:
\begin{align*}
\begin{cases}
& a\ge \frac{4}{9} \log \left( \frac{4(m+ 1) K \tau}{\Delta} \right), \text{and } b\ge \frac{4}{9} \log \left( \frac{4(m+ n) K \tau}{\Delta} \right)\\
& K \tau \ge 5\cdot 30M (\Phi(x_0^0) - \Phi^*) / \epsilon \\
& \color{blue} \tau^2 \cdot \frac{16 \widehat{l}_g^2}{a} \log \left( \frac{4(m+ 1) K \tau}{\Delta} \right) \le L_g \epsilon/(5\cdot 135 l_f M) \\
& \color{blue} \tau^2 \cdot \frac{16 \widehat{L}_g^2}{b} \log \left( \frac{4(m+ n) K \tau}{\Delta} \right) \le 6 L_g^2 /19
\end{cases}
\end{align*}
which reduces to 
\begin{equation*}
\begin{cases}
K \tau \ge C_{\Sigma} \cdot \epsilon^{-1} \\
\color{blue} a \ge C_a \cdot \tau^2 \cdot \epsilon^{-1} \cdot \log \left( \frac{4(m+ 1) K \tau}{\Delta} \right) \\
\color{blue} b \ge C_b \cdot \tau^2 \cdot \log \left( \frac{4(m+ n) K \tau}{\Delta} \right) 
\end{cases}
\end{equation*}
providing that $1/\epsilon$ is sufficiently large. Here $C_{\Sigma}, C_a, C_b$ are some constants.

For any positive integer $\tau$, let $K= \lceil \frac{C_{\Sigma} \cdot \epsilon^{-1}}{ \tau} \rceil$, {\color{blue} $a= \lceil C_a \cdot \tau^2 \cdot \epsilon^{-1} \cdot \log ( \frac{4(m+ 1) K \tau}{\Delta} ) \rceil$, $b= \lceil C_b \cdot \tau^2 \cdot \log ( \frac{4(m+ n) K \tau}{\Delta} ) \rceil$}. Then the conditions above are satisfied, so Theorem \ref{thm:unified-thm-short} guarantees that $\frac{1}{K \tau} \sum_{k=0}^{K-1} \sum_{i=0}^{\tau -1} \| \mathcal{G}_M(x_i^k) \|_2 ^2 \le \epsilon$ with probability at least $1-\Delta$.

In (\ref{eq:exact-eval-plus-standard}), at the $(k,0)$-th iteration, we evaluate $g_{j}(\cdot)$ for $N$ times and $g'_{j}(\cdot)$ for $N$ times. At the $(k,i)$-th iteration (with $i>0$), we evaluate $g_{j}(\cdot)$ for $a$ times and $g'_{j}(\cdot)$ for $b$ times. So the oracle complexity for evaluations of $g_{\xi}(\cdot)$ is 
\begin{equation*}
\begin{split}
KN + K(\tau-1)a \le K(N + \tau \cdot a) &\le (1+ \frac{C_{\Sigma}}{ \tau \epsilon} ) \cdot ( N + \widetilde{\Theta} (\epsilon^{-1} \tau^3 ) ) \\
&= \widetilde{\Theta} \left( N+ \epsilon^{-1} \tau^3 + N\epsilon^{-1} \tau^{-1} + \epsilon^{-2} \tau^2 \right),
\end{split}
\end{equation*}
and the oracle complexity for evaluations of Jacobians $g'_{\xi}(\cdot)$ is 
\begin{equation*}
\begin{split}
KN + K(\tau-1)b \le K(N + \tau \cdot b) &\le (1+ \frac{C_{\Sigma}}{ \tau \epsilon} ) \cdot ( N + \widetilde{\Theta} (\tau^3) ) \\
&= \widetilde{\Theta} \left( N+\tau^3+ N\epsilon^{-1} \tau^{-1} + \epsilon^{-1} \tau^2 \right). 
\end{split}
\end{equation*}
\end{proof}

\begin{proof}[Proof of Corollary \ref{cor:exact-eval-plus-standard-bounds}]
Note that $\tau \ge 1$, so 
\begin{equation*}
N+ \epsilon^{-1} \tau^3 + N\epsilon^{-1} \tau^{-1} + \epsilon^{-2} \tau^2 \ge N + \epsilon^{-2} \tau^2 \ge N+ \epsilon^{-2} \ge 0.
\end{equation*}
We also have
\begin{equation*}
N+ \epsilon^{-1} \tau^3 + N\epsilon^{-1} \tau^{-1} + \epsilon^{-2} \tau^2 \ge \max \{ N\epsilon^{-1} \tau^{-1}, \epsilon^{-2} \tau^2 \} \ge N^{2/3} \epsilon^{-4/3} \ge 0
\end{equation*}
Combine the two inequalities above, the rate in (\ref{eq: cor-exact-eval-plus-standard,eq-1}) is at least $\widetilde{\Theta} \left( \max \{ N+ \epsilon^{-2}, N^{2/3} \epsilon^{-4/3} \} \right) = \widetilde{\Theta} \left( N+ \epsilon^{-2}+ N^{2/3} \epsilon^{-4/3} \right)$. We claim that it can be attained by $\tau = \Theta \left( \max\{ 1, N^{1/3} \epsilon^{1/3} \} \right)$. We consider the following two subcases under this choice of $\tau$:
\begin{itemize}
\item If $N = O( \epsilon^{-1})$, then $\tau = \Theta(1)$, (\ref{eq: cor-exact-eval-plus-standard,eq-1}) becomes $\widetilde{\Theta} \left( N+ \epsilon^{-1} + N\epsilon^{-1} + \epsilon^{-2} \right) = \widetilde{\Theta} \left( \epsilon^{-2} \right) = \widetilde{\Theta} \left( N+ N^{2/3} \epsilon^{-4/3} + \epsilon^{-2} \right)$, and (\ref{eq: cor-exact-eval-plus-standard,eq-2}) becomes $\widetilde{\Theta} \left( N+1+ N\epsilon^{-1} + \epsilon^{-1} \right) = \widetilde{\Theta} \left( N\epsilon^{-1} \right) \\ = \widetilde{\Theta} \left( \min \{ N\epsilon^{-1}, N+ N^{2/3} \epsilon^{-4/3} \} \right)$.
\item If $N = \Omega( \epsilon^{-1})$, then $\tau = \Theta(N^{1/3} \epsilon^{1/3})$, (\ref{eq: cor-exact-eval-plus-standard,eq-1}) becomes $\widetilde{\Theta} \left( N+ N^{2/3} \epsilon^{-4/3} \right) = \widetilde{\Theta} \left( N+ N^{2/3} \epsilon^{-4/3} + \epsilon^{-2} \right)$, and (\ref{eq: cor-exact-eval-plus-standard,eq-2}) becomes $\widetilde{\Theta} \left( N+ N\epsilon + N^{2/3} \epsilon^{-4/3} + N^{2/3} \epsilon^{-1/3} \right) = \widetilde{\Theta} \left( N+ N^{2/3} \epsilon^{-4/3} \right) \\ = \widetilde{\Theta} \left( \min \{ N\epsilon^{-1}, N+ N^{2/3} \epsilon^{-4/3} \} \right)$.
\end{itemize}
By the two subcases above, when $\tau = \Theta \left( \max\{ 1, N^{1/3} \epsilon^{1/3} \} \right)$, (\ref{eq: cor-exact-eval-plus-standard,eq-1}) attains the aforementioned minimal asymptotic rate, and (\ref{eq: cor-exact-eval-plus-standard,eq-2}) becomes $\widetilde{\Theta} \left( \min \{ N\epsilon^{-1}, N+ N^{2/3} \epsilon^{-4/3} \} \right)$, which finishes the proof of part (i).

Next, we consider minimizing the asymptotic rate of (\ref{eq: cor-exact-eval-plus-standard,eq-2}). Note that $\tau \ge 0$ and $N\epsilon^{-1} \tau^{-1} + \epsilon^{-1} \tau^2 \ge \max \{ N\tau^{-1}, \tau^2 \} \epsilon^{-1} \ge N^{2/3} \epsilon^{-1}$, so we have $N+\tau^3+ N\epsilon^{-1} \tau^{-1} + \epsilon^{-1} \tau^2 \ge N + N^{2/3} \epsilon^{-1} \ge 0$. So the rate in (\ref{eq: cor-exact-eval-plus-standard,eq-2}) is at least $\widetilde{\Theta} \left( N + N^{2/3} \epsilon^{-1} \right)$. When $\tau = \Theta \left( N^{1/3} \right)$, (\ref{eq: cor-exact-eval-plus-standard,eq-2}) attains $\widetilde{\Theta} \left( N + N^{2/3} \epsilon^{-1} \right)$, and (\ref{eq: cor-exact-eval-plus-standard,eq-1}) becomes $\widetilde{\Theta} \left( N+ N\epsilon^{-1} + N^{2/3} \epsilon^{-1} + N^{2/3} \epsilon^{-2} \right) = \widetilde{\Theta} \left( N\epsilon^{-1} + N^{2/3} \epsilon^{-2} \right)$.
\end{proof}

\noindent {\bf Proofs for $\mathtt{estimator}_4$.}

\begin{proof}[Proof of Lemma \ref{lemma:exact-eval-plus-smart}]
We have $\widetilde{g}_0^k = g(x_0^k)$ and $\widetilde{J}_0^k = g'(x_0^k)$ from (\ref{eq:exact-eval-plus-smart}). For any $K\in \mathbb{N}_+$, $\boldsymbol{\tau} \in \mathbb{N}_+^K$ and $\Delta \in (0,1)$, let $\delta = \frac{\Delta}{2 \Sigma_\tau}$. By Proposition \ref{prop:error-bound-g-smart}, for an arbitrary $(k,i)\in \mathcal{I}(K, \boldsymbol{\tau} )$ and any $a\ge \frac{4}{9} \log \left( \frac{m+ 1}{\delta} \right)$, the following holds with probability at least $1-\delta$:
\begin{equation}
\label{proof-lemma-exact-eval-plus-smart,eq-1}
\left\| \widetilde{g}_i^k - g(x_i^k) \right\|_2 \le {\color{blue} \frac{2 \widehat{L}_g}{\sqrt{a}} \sqrt{\log \left( \frac{m+ 1}{\delta} \right)} \|x_i^k -x_0^k\|_2^2 }.
\end{equation}
By Proposition \ref{prop:error-bound-J}, for an arbitrary $(k,i)\in \mathcal{I}(K, \boldsymbol{\tau} )$ and any $b\ge \frac{4}{9} \log \left( \frac{m+ n}{\delta} \right)$, the following holds with probability at least $1-\delta$:
\begin{equation}
\label{proof-lemma-exact-eval-plus-smart,eq-2}
\left\| \widetilde{J}_i^k - g'(x_i^k) \right\|_{\rm{op}} \le {\color{blue} \frac{4 \widehat{L}_g}{\sqrt{b}} \sqrt{\log \left( \frac{m+ n}{\delta} \right)} \|x_i^k -x_0^k\|_2 }.
\end{equation}
Let $\mathcal{C}(K, \boldsymbol{\tau}, \Delta) = \{(a,b)\in \mathbb{N}_+^2 : a\ge \frac{4}{9} \log ( \frac{2(m+ 1) \Sigma_\tau}{\Delta} ), b\ge \frac{4}{9} \log ( \frac{2(m+ n) \Sigma_\tau}{\Delta} )\}$. Then for any $(a,b)\in \mathcal{C}(K, \boldsymbol{\tau}, \Delta)$, by using a union probability bound, (\ref{proof-lemma-exact-eval-plus-smart,eq-1}) and (\ref{proof-lemma-exact-eval-plus-smart,eq-2}) hold for all $(k,i)\in \mathcal{I}(K, \boldsymbol{\tau} )$ with probability at least $1- 2\Sigma_\tau \delta$. Note that $1- 2\Sigma_\tau \delta = 1- \Delta$, so we can set $\gamma_0 = \gamma_1 = \lambda_0= 0$, {\color{blue} $\gamma_2(K, \boldsymbol{\tau}, \theta, \Delta ) = \frac{2 \widehat{L}_g}{\sqrt{a}} \sqrt{\log ( \frac{2(m+ 1) \Sigma_\tau}{\Delta} )}$ and $\lambda_1(K, \boldsymbol{\tau}, \theta, \Delta ) = \frac{4 \widehat{L}_g}{\sqrt{b}} \sqrt{\log ( \frac{2(m+ n) \Sigma_\tau}{\Delta} )}$} to satisfy Assumption \ref{assumption:general-estimation-error-bounds}. 
\end{proof}

\begin{proof}[Proof of Corollary \ref{cor:exact-eval-plus-smart}]
We can obtain the explicit form of $\mathcal{C}(K, \boldsymbol{\tau}, \Delta)$, $\{\gamma_\ell \}_{\ell=0}^2$ and $\{\lambda_\ell \}_{\ell=0}^1$ from Lemma \ref{lemma:exact-eval-plus-smart}, and plug them into Theorem \ref{thm:unified-thm-short}. Then by Theorem \ref{thm:unified-thm-short}, to get $\frac{1}{K \tau} \sum_{k=0}^{K-1} \sum_{i=0}^{\tau -1} \| \mathcal{G}_M(x_i^k) \|_2 ^2 \le \epsilon$ with probability at least $1-\Delta$, we need $\overline{\delta} \le \Delta/(2 K \tau )$, $\overline{\epsilon} \le \epsilon/(5\cdot 30M)$, and the following inequalities:
\begin{align*}
\begin{cases}
& a \ge \frac{4}{9} \log \left( \frac{4(m+ 1) K \tau}{\Delta} \right), \text{and } b \ge \frac{4}{9} \log \left( \frac{4(m+ n) K \tau}{\Delta} \right) \\
& K \tau \ge 5\cdot 30M (\Phi(x_0^0) - \Phi^*) / \epsilon \\
& \color{blue} \tau^2 \cdot \frac{2 \widehat{L}_g}{\sqrt{a}} \sqrt{\log \left( \frac{4(m+ 1) K \tau}{\Delta} \right)} \le 6 L_g /25 \\
& \color{blue} \tau^2 \cdot \frac{16 \widehat{L}_g^2}{b} \log \left( \frac{4(m+ n) \Sigma_\tau}{\Delta} \right) \le 6 L_g^2 /19
\end{cases}
\end{align*}
So it reduces to 
\begin{equation*}
\begin{cases}
K \tau \ge C_{\Sigma} \cdot \epsilon^{-1} \\
\color{blue} a \ge C_a \cdot \tau^4 \cdot \log \left( \frac{4(m+ 1) K \tau}{\Delta} \right) \\
\color{blue} b \ge C_b \cdot \tau^2 \cdot \log \left( \frac{4(m+ n) K \tau}{\Delta} \right) 
\end{cases}
\end{equation*}
where $C_{\Sigma}, C_a, C_b$ are some constants.

For any positive integer $\tau$, let $K= \lceil \frac{C_{\Sigma} \cdot \epsilon^{-1}}{ \tau} \rceil$, {\color{blue} $a= \lceil C_a \cdot \tau^4 \cdot \log ( \frac{4(m+ 1) K \tau}{\Delta} ) \rceil$, $b= \lceil C_b \cdot \tau^2 \cdot \log ( \frac{4(m+ n) K \tau}{\Delta} ) \rceil$}. Then the conditions above are satisfied, so Theorem \ref{thm:unified-thm-short} guarantees that $\frac{1}{K \tau} \sum_{k=0}^{K-1} \sum_{i=0}^{\tau -1} \| \mathcal{G}_M(x_i^k) \|_2 ^2 \le \epsilon$ with probability at least $1-\Delta$.

In (\ref{eq:exact-eval-plus-smart}), at the $(k,0)$-th iteration, we evaluate $g_{j}(\cdot)$ for $N$ times and $g'_{j}(\cdot)$ for $N$ times. At the $(k,i)$-th iteration (with $i>0$), we evaluate $g_{j}(\cdot)$ for $a$ times and $g'_{j}(\cdot)$ for $b$ times.
So the oracle complexity for evaluations of $g_{\xi}(\cdot)$ is 
\begin{equation*}
\begin{split}
KN + K(\tau-1)a \le K(N + \tau \cdot a) &\le (1+ \frac{C_{\Sigma}}{ \tau \epsilon} ) \cdot ( N + \widetilde{\Theta} (\tau^5) ) \\
&= \widetilde{\Theta} \left( N+\tau^5 + N\epsilon^{-1} \tau^{-1} + \epsilon^{-1} \tau^4 \right),
\end{split}
\end{equation*}
and the oracle complexity for evaluations of Jacobians $g'_{\xi}(\cdot)$ is 
\begin{equation*}
\begin{split}
KN + K(\tau-1)b \le K(N + \tau \cdot b) &\le (1+ \frac{C_{\Sigma}}{ \tau \epsilon} ) \cdot ( N + \widetilde{\Theta} (\tau^3) ) \\
&= \widetilde{\Theta} \left( N+\tau^3+ N\epsilon^{-1} \tau^{-1} + \epsilon^{-1} \tau^2 \right). 
\end{split}
\end{equation*}
\end{proof}

\begin{proof}[Proof of Corollary \ref{cor:exact-eval-plus-smart-bounds}]
Note that $N\epsilon^{-1} \tau^{-1} + \epsilon^{-1} \tau^4 \ge N^{4/5} \epsilon^{-1}$, so we have $N+\tau^5 + N\epsilon^{-1} \tau^{-1} + \epsilon^{-1} \tau^4 \ge N + N^{4/5} \epsilon^{-1} \ge 0$. So the rate in (\ref{eq: cor-exact-eval-plus-smart,eq-1}) is at least $\widetilde{\Theta} \left( N + N^{4/5} \epsilon^{-1} \right)$. When $\tau = \Theta \left( N^{1/5} \right)$, (\ref{eq: cor-exact-eval-plus-smart,eq-1}) attains $\widetilde{\Theta} \left( N + N^{4/5} \epsilon^{-1} \right)$ and (\ref{eq: cor-exact-eval-plus-smart,eq-2}) is also $\widetilde{\Theta} \left( N + N^{4/5} \epsilon^{-1} \right)$.

Similarly, by the fact $N\epsilon^{-1} \tau^{-1} + \epsilon^{-1} \tau^2 \ge N^{2/3} \epsilon^{-1}$, we have $N+\tau^3+ N\epsilon^{-1} \tau^{-1} + \epsilon^{-1} \tau^2 \ge N + N^{2/3} \epsilon^{-1} \ge 0$. So the rate in (\ref{eq: cor-exact-eval-plus-smart,eq-2}) is at least $\widetilde{\Theta} \left( N + N^{2/3} \epsilon^{-1} \right)$. When $\tau = \Theta \left( N^{1/3} \right)$, (\ref{eq: cor-exact-eval-plus-smart,eq-2}) attains $\widetilde{\Theta} \left( N + N^{2/3} \epsilon^{-1} \right)$ and (\ref{eq: cor-exact-eval-plus-smart,eq-1}) becomes $\widetilde{\Theta} \left( N^{5/3} + N^{4/3} \epsilon^{-1} \right)$.
\end{proof}

\subsection{Proof for Randomized Epoch Durations}

\begin{proof}[Proof of Corollary \ref{cor:sample-mean-plus-standard-varying-tau}]
We first analyze part (i) in a similar way as the proof of Corollary \ref{cor:sample-mean-plus-standard}. Use the explicit form of $\mathcal{C}(K, \boldsymbol{\tau}, \Delta)$, $\{\gamma_\ell \}_{\ell=0}^2$ and $\{\lambda_\ell \}_{\ell=0}^1$ from Lemma \ref{lemma:sample-mean-plus-standard}, and plug them into Theorem \ref{thm:unified-thm-short}. Then by Theorem \ref{thm:unified-thm-short}, to get $\frac{1}{\Sigma_\tau} \sum_{k=0}^{K-1} \sum_{i=0}^{\tau_k -1} \| \mathcal{G}_M(x_i^k) \|_2 ^2 \le \epsilon$ with probability at least $1-\Delta$, we need
\begin{equation}
\label{eq: proof-cor-sample-mean-plus-standard-varying-tau,eq-1}
\overline{\delta} \le \Delta/(2 \Sigma_\tau ) \quad\text{and}\quad \overline{\epsilon} \le \epsilon/(5\cdot 30M)
\end{equation}
and the following inequalities:
\begin{align}
\label{eq: proof-cor-sample-mean-plus-standard-varying-tau,eq-2}
\begin{cases}
& A,a \ge \frac{4}{9} \log \left( \frac{4(m+ 1) \Sigma_\tau}{\Delta} \right), \text{and } B,b \ge \frac{4}{9} \log \left( \frac{4(m+ n) \Sigma_\tau}{\Delta} \right) \\
& \Sigma_\tau \ge 5\cdot 30M (\Phi(x_0^0) - \Phi^*) / \epsilon \\
& \frac{2 \sigma_g}{\sqrt{A}} \sqrt{\log \left( \frac{4(m+ 1) \Sigma_\tau}{\Delta} \right)} \le \epsilon/(5\cdot 125 l_f M) \\
& \frac{4 \sigma_{g'}^2 }{B} \log \left( \frac{4(m+ n) \Sigma_\tau}{\Delta} \right) \le L_g \epsilon/(5\cdot 95 l_f M) \\
& \color{blue} \tau_{\text{max}}^2 \frac{16 \widehat{l}_g^2}{a} \log \left( \frac{4(m+ 1) \Sigma_\tau}{\Delta} \right) \le L_g \epsilon/(5\cdot 135 l_f M) \\
& \color{blue} \tau_{\text{max}}^2 \frac{16 \widehat{L}_g^2}{b} \log \left( \frac{4(m+ n) \Sigma_\tau}{\Delta} \right) \le 6 L_g^2 /19
\end{cases}
\end{align}
By \eqref{eq:randomized-tau} and Assumption \ref{assumption:varying-tau}, the constructed $\boldsymbol{\tau}$ satisfies $\tau_{\text{max}} \le \tau_+$ and $S_{\tau} \le \Sigma_{\tau} \le S_{\tau}+ \tau_+$. So the choices $\overline{\delta} = \frac{\Delta }{2( S_{\tau}+ \tau_+) }$ and $\overline{\epsilon} = \epsilon/(5\cdot 30M)$ imply (\ref{eq: proof-cor-sample-mean-plus-standard-varying-tau,eq-1}), and the following inequalities suffice to imply (\ref{eq: proof-cor-sample-mean-plus-standard-varying-tau,eq-2}):
\begin{align}
\label{eq: proof-cor-sample-mean-plus-standard-varying-tau,eq-3}
\begin{cases}
& S_\tau \ge C_{\Sigma} \cdot \epsilon^{-1} \\
& A \ge C_A \cdot \epsilon^{-2} \cdot \log \left( \frac{4(m+ 1) (S_{\tau}+ \tau_+) }{\Delta} \right) \\
& B \ge C_B \cdot \epsilon^{-1} \cdot \log \left( \frac{4(m+ n) (S_{\tau}+ \tau_+) }{\Delta} \right) \\
& \color{blue} a \ge C_a \cdot \tau_+^2 \cdot \epsilon^{-1} \cdot \log \left( \frac{4(m+ 1) (S_{\tau}+ \tau_+) }{\Delta} \right) \\
& \color{blue} b \ge C_b \cdot \tau_+^2 \cdot \log \left( \frac{4(m+ n) (S_{\tau}+ \tau_+) }{\Delta} \right)
\end{cases}
\end{align}
providing that $1/\epsilon$ is sufficiently large. Here $C_{\Sigma}, C_A, C_B, C_a, C_b$ are some constants.

Let $\tau_+ = \lceil \epsilon^{-1/3} \rceil$, $S_{\tau} = \lceil C_{\Sigma} \cdot \epsilon^{-1} \rceil$, $A = \lceil C_A \cdot \epsilon^{-2} \cdot \log ( \frac{5(m+ 1) S_\tau }{\Delta} ) \rceil$, $B = \lceil C_B \cdot \epsilon^{-1} \cdot \log ( \frac{5(m+ n) S_\tau }{\Delta} ) \rceil$, {\color{blue} $a = \lceil C_a \cdot \tau_+^2 \cdot \epsilon^{-1} \cdot \log ( \frac{5(m+ 1) S_\tau }{\Delta} ) \rceil$, $b = \lceil C_b \cdot \tau_+^2 \cdot \log ( \frac{5(m+ n) S_\tau }{\Delta} ) \rceil$}, then (\ref{eq: proof-cor-sample-mean-plus-standard-varying-tau,eq-3}) holds for sufficiently small $\epsilon$. 
So part (i) of Corollary \ref{cor:sample-mean-plus-standard-varying-tau} holds with probability at least $1-\Delta$ by choosing these parameters.

In the rest of the proof, we analyze part (ii) of Corollary \ref{cor:sample-mean-plus-standard-varying-tau}. We need to provide a high probability upper bound on $K$. For any positive integer $M$, \eqref{eq:randomized-tau} implies that $K>M$ if and only if $\sum_{k=0}^{M-1} \tau_k < S_{\tau}$, so $\mathbb{P}(K>M) = \mathbb{P} \left( \sum_{k=0}^{M-1} \tau_k < S_{\tau} \right)$. Note that the random variables $\{\tau_k\}$ are independent and bounded between $[0, \tau_+]$, so we can use Hoeffding's Inequality. Denote $\mu_{\tau} := \mathbb{E}_{\tau \sim D_{\tau}(\cdot; \tau_+, \theta_{\tau} )} [\tau] $. By Lemma \ref{lemma:Hoeffding-corollary}, for any $t\ge 0$,
\begin{equation}
\label{eq: proof-cor-sample-mean-plus-standard-varying-tau,eq-4}
\mathbb{P} \left( \sum_{k=0}^{M-1} \tau_k \le M \mu_{\tau} -t \right) \le \exp\left( -\frac{2 t^2}{M \tau_+^2} \right).
\end{equation}
Let $M= \lceil \frac{2C_{\tau} S_{\tau}}{ \tau_+} \rceil$. By Assumption \ref{assumption:varying-tau}, $\mu_{\tau} >0$ and $C_{\tau} \mu_{\tau} \ge \tau_+ > 0$, so $M\ge \frac{2C_{\tau} S_{\tau}}{ \tau_+} \ge \frac{2 S_{\tau}}{ \mu_{\tau}} > \frac{S_{\tau}}{ \mu_{\tau}}$. Let $t = M \mu_\tau - S_\tau \ge 0$ in (\ref{eq: proof-cor-sample-mean-plus-standard-varying-tau,eq-4}), then we get
\begin{equation}
\label{eq: proof-cor-sample-mean-plus-standard-varying-tau,eq-5}
\mathbb{P}(K>M) = \mathbb{P} \left( \sum_{k=0}^{M-1} \tau_k < S_{\tau} \right) \le \mathbb{P} \left( \sum_{k=0}^{M-1} \tau_k \le S_{\tau} \right) \le \exp\left( -\frac{2 (M \mu_{\tau} - S_{\tau} )^2}{M \tau_+^2} \right).
\end{equation}
Note that the mapping $\phi(x) = \frac{1}{x}(x \mu_{\tau} - S_{\tau} )^2$ is increasing on the interval $[\frac{S_{\tau}}{ \mu_{\tau}}, +\infty)$, so $\phi(M) \ge \phi( \frac{2C_{\tau} S_{\tau}}{ \tau_+} )$, which further implies
\begin{equation}
\label{eq: proof-cor-sample-mean-plus-standard-varying-tau,eq-6}
\begin{split}
\exp\left( -\frac{2 (M \mu_{\tau} - S_{\tau} )^2}{M \tau_+^2} \right) &\le \exp\left( -\frac{2 }{(\frac{2C_{\tau} S_{\tau}}{ \tau_+}) \tau_+^2} \left( \frac{2C_{\tau} S_{\tau}}{ \tau_+} \mu_{\tau} - S_{\tau} \right)^2 \right) \\
&= \exp\left( -\frac{S_{\tau} }{C_{\tau} \tau_+} \left( \frac{2C_{\tau} \mu_{\tau} }{ \tau_+} - 1 \right)^2 \right) \le \exp\left( -\frac{S_{\tau} }{C_{\tau} \tau_+} \right)
\end{split}
\end{equation}
where the last step is because $\frac{2C_{\tau} \mu_{\tau} }{ \tau_+} - 1 \ge 1$. By (\ref{eq: proof-cor-sample-mean-plus-standard-varying-tau,eq-5}) and (\ref{eq: proof-cor-sample-mean-plus-standard-varying-tau,eq-6}), $\mathbb{P}(K > \lceil \frac{2C_{\tau} S_{\tau}}{ \tau_+} \rceil) \le \exp\left( -\frac{S_{\tau} }{C_{\tau} \tau_+} \right)$. 
So $K\le \lceil \frac{2C_\tau S_\tau}{ \tau_+} \rceil = \Theta(\epsilon^{-2/3})$ with probability at least $1-\exp(-\frac{S_\tau}{C_\tau \tau_+}) \ge 1 - \exp(- C_p \epsilon^{-2/3})$ for some constant $C_p$.

In (\ref{eq:sample-mean-plus-standard}), at the $(k,0)$-th iteration, we evaluate $g_{\xi}(\cdot)$ for $A$ times and $g'_{\xi}(\cdot)$ for $B$ times. At the $(k,i)$-th iteration (with $i>0$), we evaluate $g_{\xi}(\cdot)$ for $2a$ times and $g'_{\xi}(\cdot)$ for $2b$ times.
On the high probability set where $K= O(\epsilon^{-2/3})$, the oracle complexity for evaluations of $g_{\xi}(\cdot)$ is 
$$KA + 2(\Sigma_{\tau} - K)a \le KA + 2\Sigma_{\tau} \cdot a \le KA + 2(S_{\tau}+ \tau_+) a = \widetilde{O}(\epsilon^{-8/3} \log(1/\Delta)) ,$$
and the oracle complexity for evaluations of Jacobians $g'_{\xi}(\cdot)$ is 
$$KB + 2(\Sigma_{\tau} - K)b \le KB + 2\Sigma_{\tau} \cdot b \le KB + 2(S_{\tau}+ \tau_+) b = \widetilde{O}(\epsilon^{-5/3} \log(1/\Delta)) .$$
Using a union probability bound for part (i) and part (ii), they hold simultaneously with probability at least $1 - \Delta - \exp(- C_p \epsilon^{-2/3})$, which finishes the proof.
\end{proof}

\end{document}